\documentclass[11pt, dvipsnames]{amsart}
\usepackage[utf8]{inputenc}
\usepackage{tikzscale}
\usepackage{graphicx}
\usepackage{a4wide}
\usepackage{color}
\usepackage{xcolor}
\usepackage{amsmath,amssymb}
\usepackage{float}
\usepackage{soul}
\usepackage{lipsum} 
\usepackage{hyperref}
\usepackage{cleveref}
\usepackage{bm}
\usepackage{enumitem}
\usepackage{float}
\usepackage{amsthm}
\usepackage{pgfplots}
\pgfplotsset{compat=1.18, 
    legend image with text/.style={
        legend image code/.code={            \node[anchor=center] at (0.3cm,0cm) {#1};
        }
    },
}

\usepackage{pgfplotstable}
\usepackage{yhmath} \usepackage{lineno}

\usepackage{comment}

\usepackage{xcolor} 
\definecolor{myblue}{HTML}{007BFF} \definecolor{myred}{HTML}{FF5733}  
\newtheorem{theorem}{Theorem}[section]
\newtheorem{lemma}[theorem]{Lemma}
\newtheorem{example}{Example}
\newtheorem{assumption}{Assumption}

\newtheorem{remark}{Remark}

\usepackage{mathtools}
\usepackage{etoolbox}
\usepackage{bm}
\usepackage{amsmath}      
\usepackage{amssymb}      
\usepackage{algorithm}
\usepackage{algpseudocode}
\usepackage{fancybox}     
\usepackage{caption}
\usepackage{overpic}
\usepackage[foot]{amsaddr}

\numberwithin{example}{section} 

\DeclareFontFamily{U}{matha}{\hyphenchar\font45}
\DeclareFontShape{U}{matha}{m}{n}{
<-6> matha5 <6-7> matha6 <7-8> matha7
<8-9> matha8 <9-10> matha9
<10-12> matha10 <12-> matha12
}{}
\DeclareSymbolFont{matha}{U}{matha}{m}{n}

\DeclareFontFamily{U}{mathx}{\hyphenchar\font45}
\DeclareFontShape{U}{mathx}{m}{n}{
<-6> mathx5 <6-7> mathx6 <7-8> mathx7
<8-9> mathx8 <9-10> mathx9
<10-12> mathx10 <12-> mathx12
}{}
\DeclareSymbolFont{mathx}{U}{mathx}{m}{n}

\DeclareMathDelimiter{\vvvert} {0}{matha}{"7E}{mathx}{"17}
\DeclarePairedDelimiterX{\normiii}[1]
{\vvvert}
{\vvvert}
{\ifblank{#1}{\:\cdot\:}{#1}}

\newcommand{\R}{\mathbb{R}}

\newcommand{\dd}{\mathop{}\!\mathrm{d}}

\newcommand{\fa}{\text{ for all }}

\definecolor{color0}{rgb}{0.7843, 0.7843, 0.7843}
\definecolor{color1}{rgb}{0, 0.4470, 0.7410}
\definecolor{color2}{rgb}{0.8500, 0.3250, 0.0980} \definecolor{color3}{rgb}{0.9290, 0.6940, 0.1250}
\definecolor{color4}{rgb}{0.7060, 0.3840, 0.7650}
\definecolor{color5}{rgb}{0.4660, 0.6740, 0.1880}
\definecolor{color6}{rgb}{0.3010, 0.7450, 0.9330}
\definecolor{color7}{rgb}{0.6350, 0.0780, 0.1840}
\definecolor{color8}{rgb}{0.0, 0.4078, 0.3412}
\usepackage{multicol}
\usepackage{multirow}
\title{The proximal Galerkin method for non-symmetric variational inequalities }
\date{\today}
\author{Guosheng Fu$^1$}
 \address{$^1$ Department of Applied and Computational Mathematics and Statistics (ACMS), University of Notre Dame, Notre Dame, IN 46556}
 \email{gfu@nd.edu}
 \author{Brendan Keith$^2$}
 \author{Dohyun Kim$^2$}
 \address{$^2$ Division of Applied Mathematics, Brown University, Providence, RI 02912}
 \email{brendan\_keith@brown.edu, dohyun\_kim@brown.edu, rami\_masri@brown.edu}
 \author{Rami Masri$^2$}
 \author{Will Pazner$^3$}
  \address{$^3$ Department of Mathematics and Statistics, Portland State University, Portland, OR 97201}

 \email{pazner@pdx.edu}
 \thanks{
 GF was supported in part by NSF DMS-2410740.
 BK, DK, and RM were supported in part by the U.S.\ Department of Energy, Office of Science Early Career Research Program under Award Number DE-SC0024335 and by the Center for Information Geometric Mechanics and Optimization (CIGMO), a PSAAP-IV Focused Investigatory Center funded by the U.S.\ Department of Energy, National Nuclear Security Administration under Award Number DE-NA0004261.
 BK was also supported in part by the Alfred P.\ Sloan Foundation via a Sloan Research Fellowship in Mathematics.
 WP was supported in part by NSF DMS-2136228 and CC*-2346732.}

\begin{document}
\begin{abstract}
We introduce the proximal Galerkin (PG) method for non-symmetric variational inequalities. The proposed approach is asymptotically mesh-independent and yields constraint-preserving approximations. We present both a conforming PG formulation and a hybrid mixed first-order system variant (FOSPG). We establish optimal a priori error estimates for each variant, which are verified numerically. We conclude by applying the method to American option pricing, free boundary problems in porous media, advection–diffusion with a semipermeable boundary, and the enforcement of discrete maximum principles.

\vspace{1em}
  \smallskip

\end{abstract}
\maketitle

\section{Introduction}
The proximal Galerkin (PG) method is a numerical framework for solving variational inequalities (VIs) \cite{keith2023proximal}, combining ideas from Bregman proximal point methods and finite element theory.  The PG framework recently demonstrated competitive computational efficiency for a diverse set of mathematical problems, delivering fast mesh-independent convergence and constraint-preserving approximations \cite{dokken2025latent,papadopoulos2024hierarchical,kim2024simple}. These properties have been recently rigorously established in \cite{aprioriPG} for quadratic energy minimization problems.

In this paper, we extend the PG framework \cite{keith2023proximal,dokken2025latent,aprioriPG} to VIs with non-symmetric bilinear forms (a.k.a.\ non-symmetric VIs) by introducing a generalized formulation that retains the aforementioned desirable properties of the original method.   Non-symmetric VIs are used to model a variety of systems, including flow through porous media \cite{Baiocchi73a}, semi-permeable membranes \cite{duvaut1976inequalities}, large ice sheets modeling \cite{jouvet2012steady}, and pricing American options in quantitative finance \cite{heston1993closed}.

We note that an alternative approach to non-symmetric VIs was proposed in \cite[Section 5.2]{keith2023proximal}.
This approach can be formally derived using a Bregman divergence to regularize a well-known fixed-point operator that converges to the solution of the underlying VI.
Although more general in theory, the simplest and most practical setting reduces to splitting the symmetric and non-symmetric components of the underlying bilinear form, treating the symmetric part implicitly and the non-symmetric part explicitly in each proximal subproblem; cf.\ Algorithm 4 in \cite{keith2023proximal} with $\rho = 1/\epsilon$.
Unfortunately, this entire class of approaches is generally unstable for large step sizes due to the explicit part of the bilinear form in each subproblem, thereby limiting the overall convergence rate.
Instead, in this work, we adopt a simpler, more efficient approach by abandoning generalized operator splitting and treating all contributions to the bilinear form implicitly.

Further elements of the literature focus primarily on symmetric VIs, which are associated with minimization problems; see \cite[Section 3]{aprioriPG}, \cite{karkkainen2003augmented}, and \cite{gustafsson2017finite} for detailed reviews. 
Here, we highlight two popular approaches. The first is the quadratic penalty method, which relaxes constraints but suffers from mesh-dependent ill-conditioning as penalty parameters must scale inversely with the mesh size to maintain accuracy \cite{scholz1984numerical}. Alternatively, one can discretize the VI directly \cite{brezzi1977error,brezzi1978error} and then apply techniques from nonlinear programming, such as the primal-dual active set \cite{hintermuller2002primal} and the augmented Lagrangian method \cite{glowinski1989augmented} to solve the resulting discrete VI. Unfortunately, the latter class of approaches also tends to exhibit mesh-dependence; i.e., the number of nonlinear solves, not just the cost per linear solve, increases with mesh refinement \cite{bueler2024full}. This shortcoming can be mitigated by multigrid methods, though theoretical guarantees are lacking \cite{graser2009multigrid}.

At the discrete level, the non-symmetric matrix arising from the non-adjoint operator prevents the application of well-established quadratic programming solvers.
Thus, in the context of parabolic VIs, operator-splitting methods that only involve the symmetric part of the operator have been proposed \cite{mezzan2025lagrangian, ikonen2004operator, ikonen2009operator}. Here, for the PG framework, we do not explore such splitting approaches as our focus is on the steady state problem. However, we note that as demonstrated in \Cref{subsec:american_option} and \Cref{example:option_pricing}, the combination of PG with backward Euler in time is stable and efficient.

\subsection{Main contributions}
\begin{itemize}
    \item We extend the PG framework from energy minimization problems to the class of non-symmetric VIs.
    
    \item We provide a general framework for conforming discretizations.
    In \Cref{thm:existence_conforming}, we prove that the discrete subproblems are well posed under certain compatibility conditions of the Galerkin subspaces. A stability result for the discrete variables is also established, see \Cref{lemma:stability}.
         \Cref{thm:best_approximation} and \Cref{thm:error_rate_conforming} provide the best approximation result and error rates, respectively.
    \item We introduce hybridized first-order system PG (FOSPG) methods for obstacle-type advection-diffusion VIs. This spatial discretization is favorable over the conforming method in advection-dominated regimes. We prove its well-posedness and error rates in \Cref{thm:FOSPG_existence} and \Cref{thm:FOSPG_err_rate}, respectively.
    The FOSPG method is also extended to semi-permeable boundary conditions (which includes the Signorini problem) in \Cref{sec:semi_permeable}.
\end{itemize}
\subsection{Outline} We end this section by introducing the basic notation used throughout the paper. \Cref{sec:non_symmetricVIs} introduces the general model problem that we study and presents four applications:  option pricing, semi-permeable boundary conditions, free boundary problems in porous media, and advection-diffusion problems. In \Cref{sec:conforming_PG}, we present the conforming PG method for non-symmetric VIs and provide a detailed stability and error analysis. The hybridizable first-order system PG (FOSPG) method is introduced and analyzed in \Cref{sec:FOSPG}. We present numerical experiments in \Cref{sec:numerical_experiments}, which verify our theoretical findings and illustrate the performance of our methods.
 \subsection{Notation}
In this article, $\Omega$ denotes an open bounded Lipschitz domain in $\mathbb{R}^n$ $(n=1,2,3)$. The dual space of a Banach space $V$ is denoted by $V'$ with duality pairing $\langle \cdot, \cdot \rangle$. We use the standard notation for the Sobolev--Hilbert spaces $H^m(\Omega)$. For non-integer $s$, $H^s(\Omega)$  denotes the Sobolev--Slobodeckij spaces \cite[Chapter 2]{ern2017finite}. The notation $(\cdot, \cdot)_{\omega}$ denotes the $L^2(\omega)$-inner product over a measurable set $\omega \subset \overline{\Omega}$. The trace of $v\in H^1(\Omega)$ on a part of the boundary $\Gamma \subset \partial \Omega$ is denoted by $\operatorname{tr} v$.  If $\omega = \Omega$, we drop the subscript and denote the $L^2$-inner product over $\Omega$ by $(\cdot,\cdot)$. For an extended real valued function $f\colon\R^n \rightarrow \R \cup \{+\infty\}$, we denote by $\operatorname{dom} f : = \{ x \in \R^n : f(x) < \infty\}$ the essential domain of $f$. For a linear continuous operator $B \in \mathcal{L}(U,V)$ where $U,V$ are normed vector spaces, the topological transpose (adjoint) operator $B' \in \mathcal{L}(V',U')$ is defined as
\begin{equation}
\label{eq:DualOperator}
\langle B' v' , u \rangle = \langle v' , B u \rangle  \text{ for all } u \in U, v'  \in V'.
\end{equation}

We consider a conforming simplicial shape regular partition  $\mathcal{T}_h$  of $\Omega$ into elements $T$. Denote by  $\Gamma_h$ the set of facets $F$ (edges in 2D/faces in 3D) of the partition $\mathcal{T}_h$, and
denote by $\partial\mathcal{T}_h$ the set of all element boundaries $\partial T$ with outward unit normal $\bm{n}$. Further, we denote by $\Gamma_h^0$ the set of interior facets and by  $\Gamma_h^\partial$ the set of boundary facets. 
  We denote by $\mathbb P_p( T)$ (resp. $\mathbb P_p( F)$) the space
of polynomials of degree at most $p$ on $T$ (resp. $ F$).
We also use the Raviart--Thomas element \cite{RaviartThomas77} of degree $p$ on $T$, denoted by
$\mathrm{RT}_p(T):=[\mathbb P_p(T)]^n+  \mathbf{x}\cdot\mathbb{P}_{p}(T)$.
The space $H^1(\mathcal{T}_h)$ denotes the broken $H^1$ space corresponding to the mesh $\mathcal{T}_h$:
\[H^1(\mathcal{T}_h) := \{ u \in L^2(\Omega): \;\; u_{\vert_{T}}  \in H^1(T), \;\; \forall T \in \mathcal{T}_h \}. \] The broken gradient and divergence are denoted by $\nabla_h$ and $\nabla_h \cdot$ respectively, meaning that $(\nabla_h v)_{\vert_{T}} = \nabla (v_{\vert_T})$ and  $(\nabla_h \cdot v )_{\vert_T} =\nabla  \cdot (v_{\vert_T}) $ for $v \in H^1(\mathcal{T}_h)$. Further for all $q, \varphi \in L^2(\partial \mathcal{T}_h)$, we use the notation
\begin{equation}
(q,\varphi)_{\partial \mathcal{T}_h} = \sum_{T \in \mathcal{T}_h} \int_{\partial T} q\varphi \mathrm{d }s.
\end{equation}
We will often use the notation $A \lesssim B$ to indicate that there is a positive constant $C$ independent of $h$, the iteration count $k$, and the proximity parameters $\{\alpha_k\}$ such that $A \leq C \, B$.
\section{Non-symmetric variational inequalities}\label{sec:non_symmetricVIs}
This section introduces the abstract setup and provides four examples. We consider a Hilbert space $V$, a linear operator $\mathcal{L}: V \rightarrow V'$, and a linear functional $F \in V'$. Given a closed and convex set $K \subset V$, we are interested in the following variational inequality problem: Find $u \in K$ such that
\begin{align} \label{eq:general_VI}
\langle \mathcal{L} u, v - u \rangle \geq F(v-u)  ~\fa  v \in K.
\end{align}
The operator $\mathcal{L}$ need not be symmetric.
In addition, this operator gives rise to the bilinear form $\mathcal{A}:V \times V \rightarrow \mathbb{R}$ defined by $\mathcal{A}(u,v) := \langle \mathcal{L} u,v \rangle$. We define the symmetric and non-symmetric components  of $\mathcal{A}$:
\begin{align}
\label{eq:Splitting}
\mathcal{A}_0 (u,v) & = \frac{1}{2} (\mathcal{A}(u,v) + \mathcal{A}(v,u) ),  \;\;  \mathcal{A}_n (u,v) = \frac{1}{2} (\mathcal{A}(u,v) - \mathcal{A}(v,u) ).
\end{align}
We assume that $\mathcal{L}$ is sectorial; i.e., the skew-symmetric part is continuous in the sense that
\begin{align} \label{eq:skew_symm_part}
    \mathcal{A}_n(u,v) \leq c_1 \|u\|_V \|v\|_V ~\fa u,v \in V,
\end{align}
for a non-negative constant $c_1$.
We further assume that $\mathcal{A}$ is coercive and continuous:
\begin{alignat}{2}\label{eq:coercivity}
\mathcal{A}_0(u,u)=\mathcal{A}(u,u)&\geq C_{\mathrm{coerc}} \|u\|^2_V &&~\fa u \in V,\\
\mathcal{A}_0(u,v)&\leq C_{\mathrm{bnd}}\|u\|_V\|v\|_V&&~\fa u,v\in V, \label{eq:continuity}
\end{alignat}
where  $C_{\mathrm{coerc}}$ and $C_{\mathrm{bnd}}$ are positive constants.
\begin{lemma}
Assume that the properties given by \eqref{eq:skew_symm_part}, \eqref{eq:coercivity}, and \eqref{eq:continuity} hold. Then, there exists a unique solution $u^* \in V$ to \eqref{eq:general_VI}.
\end{lemma}
\begin{proof}
We refer to \cite[Section 2 of Chapter 2]{kinderlehrer2000introduction}.
\end{proof}
Hereinafter, we consider feasible sets $K$ that have the following general form:
\begin{equation}
K = \{v \in V \mid  B v(x) \in C(x)  \text{ for almost every } x \in \Omega_d \subset \overline \Omega \},  \label{eq:constraint_set_general}
\end{equation}
where  $\Omega_d$ is a Hausdorff-measurable set with  dimension $d \leq n$ and  measure $\dd \mathcal{H}_d$. We assume that $B\colon V \rightarrow Q$ is a bounded linear map, whose image $Q = \operatorname{im} B$ is continuously and densely embedded in $L^2(\Omega_d;\mathbb{R}^m)$, and $C(x) \subset \R^m$, which may vary with $x$, is a closed convex set with a nonempty interior. The set of constrained observables defined on $\Omega_d$ is denoted by
 \begin{equation}
    \mathcal{O}
    =
    \{ o \in L^2(\Omega_d;\mathbb{R}^m) \mid o(x) \in C(x) \text{ for almost every } x \in \Omega_d \subset \overline \Omega \}
                    .
 \label{eq:ConstrainedObservables}
 \end{equation}
Finally, we introduce the dual variable $\lambda^*\in Q'$.
\begin{lemma}
 Given $u^* \in V$ solving \eqref{eq:general_VI}, there exists a unique dual variable $\lambda^* \in Q'$ satisfying
\begin{equation} \label{eq:def_lambda*}
    \langle B' \lambda^*, v \rangle =\langle \mathcal{L} u^*  , v  \rangle  - F(v) ~\fa v \in V.
\end{equation}
\end{lemma}
\begin{proof}
    Since $\mathcal{A}_0$ is coercive, the variational inequality
    \begin{equation*}
        \mathcal{A}_0(u,v-u)\geq G(v-u)\text{ for all }v\in K,
    \end{equation*}
    has a unique solution $u$ with an associated dual variable $\lambda\in Q'$ \cite[Theorem 3.11, Remark 3.10]{HASLINGER1996313} for any $G \in V'$.
    This dual variable satisfies
    \begin{equation*}
        \langle B'\lambda,v\rangle=\mathcal{A}_0(u,v)-G(v)\text{ for all }v\in V.
    \end{equation*}
    To conclude the result, we set
    \begin{equation*}
        G(v)=F(v)-\mathcal{A}_n(u^*,v).
        \qedhere
    \end{equation*}
\end{proof}
\noindent
It readily follows that
\begin{equation}
    \label{eq:lam-normal-cone}
    \langle B'\lambda^*,v-u^*\rangle = \langle \mathcal{L}u^*,v-u^*\rangle-F(v-u^*)\geq 0\text{ for all }v\in K.
\end{equation}

We now conclude this section by providing four examples that illustrate the general setup of this paper.
\begin{example}[Advection-diffusion problems with bound constraints] \label{example:obstacle_type} We consider the following operator $\mathcal{L}: H^1_0(\Omega) \rightarrow H^{-1}(\Omega)$:
\begin{align} \label{eq:L_convection_diffusion}
\mathcal{L} u = - \nabla \cdot (  \kappa \nabla u) + \beta \cdot \nabla  u + c u,
\end{align}
where $\kappa \in L^{\infty}(\Omega)$ is uniformly bounded below by a positive real number, $c \in L^{\infty}(\Omega)$, and
$\beta \in [L^{\infty}(\Omega)]^d$ with $\nabla\cdot\beta \in  L^{\infty}(\Omega)$. For \eqref{eq:skew_symm_part} and \eqref{eq:coercivity} to hold, it suffices to assume that \cite[Section 4.6.1]{di2011mathematical}
\[
c - \frac12 \nabla \cdot \beta \geq  0 ~\text{ a.e. in }  \Omega.
\]
For the set $K$, we write
\begin{equation} \label{eq:bilateral_K}
K = \{v \in H^1_g(\Omega) \mid \phi_1 \leq  v \leq  \phi_2 \text{ a.e.\ in  }  \Omega \},
\end{equation}
where $\phi_1, \phi_2 \in H^1(\Omega) \cap C(\overline \Omega)$ with $\phi_1 \leq \phi_2$ a.e. in $\Omega$ and $\phi_1 \leq g \leq \phi_2$ on $\partial \Omega$.
In \eqref{eq:constraint_set_general},  we take $B$ to be the identity operator, $C(x) = [\phi_1(x) , \phi_2(x)] $ and $\Omega_d = \Omega$ to recover \eqref{eq:bilateral_K}.

\end{example}
\begin{remark}[Discrete Maximum Principle]
Consider the following advection-diffusion-reaction equation:
\begin{alignat*}{2}
 - \nabla \cdot (  \kappa \nabla u) + \beta \cdot \nabla  u + c u&=0&&\text{ in }\Omega,\\
 u&=g&&\text{ on }\partial\Omega.
\end{alignat*}
The solution $u$ satisfies the maximum principle, i.e., $u\in K$ given in \Cref{example:obstacle_type} with $\phi_1=\operatorname{ess\,inf}g$ and $\phi_2=\operatorname{ess\,sup}g$.
While these constraints are theoretically redundant at the continuous level due to the continuous maximum principle, we can exploit this property numerically. Standard Galerkin methods often produce spurious oscillations for advection-dominated flows; formulating the problem as a VI (see \Cref{example:obstacle_type}) explicitly enforces the discrete maximum principle, ensuring physically meaningful numerical solutions. Refer to \Cref{sec:hemker} for a numerical example.
\end{remark}

\begin{example}[American option pricing] \label{example:option_pricing}
Parabolic VIs can model an asset price and the optimal time to exercise an option; we refer to \cite[Section I]{moon2007posteriori} for more details. These VIs take the following form: Find $u \in C([0,T]; L^2(\Omega)) \cap L^2(0,T;H^1(\Omega))$ such that $u(t) \in \mathcal{K}(t)$ for a.e. $t \in [0,T]$ and
\begin{subequations}
    \begin{equation} \label{eq:parabolic_VI}
\langle \partial_tu , v - u  \rangle + \langle \mathcal{L} u, v - u \rangle \geq (f, v -u) ~\fa  v\in \mathcal{K}(t) \;\; \text{for a.e.} \;\;  t \in [0,T].
\end{equation}
Here, $\mathcal{L}$ is given by \eqref{eq:L_convection_diffusion} and  $\mathcal{K}(t)$ is the constraint set defined by
\begin{equation}
    \mathcal{K}(t) =  \{v \in H^1_0(\Omega) \mid v \geq \phi(t) \text{ a.e.\ in  }  \Omega \},
\end{equation}
\end{subequations}
where $\phi(t)$ is a time-dependent obstacle, representing the payoff function (initial asset price). 
\end{example}

\begin{example}[Semi-permeable boundary conditions] \label{example:semi_permeable}We consider a non-overlapping partition of the boundary $\partial \Omega$ into $\Gamma_{\mathsf S} , \Gamma_{\mathsf N}$, and $\Gamma_{\mathsf D}$, where $\Gamma_\mathsf{D}$ has non trivial measure.
Consider the following system
\begin{equation} \label{eq:semi_perm_0}
\begin{alignedat}{2}
    - \nabla \cdot ( \kappa  \nabla u) + \beta \cdot \nabla u & = f && \text{ in }  \Omega, \\
    u & = 0 && \text{ on } \Gamma_{\mathsf{D}},  \\
\kappa \nabla u \cdot n - \beta u \cdot n& = g &&  \text{ on } \Gamma_{\mathsf{N}}, \\
\end{alignedat}
\end{equation}
Further, we have the following conditions modeling $\Gamma_S$ as semi-permeable
\begin{equation} \label{eq:semi_perm_1}
u \geq \phi, \;\;  \frac{\partial u }{ \partial n} \geq 0, \; \; (u-\phi) \frac{\partial u }{ \partial n} = 0 \text{ on } \Gamma_{\mathsf{S}}.
\end{equation}
This means that $\Gamma_{\mathsf{S}}$ is impermeable until $u$ reaches a certain threshold $\phi$. Whenever $u=\phi$, $\Gamma_S$ becomes fully permeable. Considering the constraint set
\begin{equation}
   \label{eq:constraint_semi_perm}
   K = \{ v \in H^1_{\mathsf{D}}(\Omega) \mid  \operatorname{tr} v \geq \phi \text{ a.e. on } \Gamma_{\mathsf{S}} \},
\end{equation}
the above model \eqref{eq:semi_perm_0}-\eqref{eq:semi_perm_1} can be formulated as a VI: find $u \in K$ such that
\begin{equation}
(\kappa \nabla u , \nabla (v-u)) + (u, \beta \cdot \nabla (v - u) ) \geq (f, v-u) + \langle g, v - u \rangle_{H^{-1/2}(\Gamma_\mathsf{N})} ~\fa v \in K.
\end{equation}
Considering the general form \eqref{eq:constraint_set_general}, we set $V = H^1_{\mathsf{D}}(\Omega)$, $B$ the trace operator, $C(x) = [\phi(x), \infty)$ and $\Omega_d = \Gamma_{\mathsf{S}}$ to recover \eqref{eq:constraint_semi_perm}. The space $Q = \operatorname{im} B$ is the Lions--Magenes space  $\widetilde{H}^{1/2}(\Gamma_{\mathsf{S}}) := H^{1/2}_{00}(\Gamma_{\mathsf{S}})
 = \{w \in H^{1/2}(\Gamma_\mathsf{S}) \mid \tilde w \in H^{1/2}(\partial \Omega) \},$ where $\tilde w = 0 $ on $\partial \Omega \backslash \Gamma_{\mathsf{S}}$ and $\tilde w = w$ on $\Gamma_{\mathsf{S}}$.
\end{example}

\begin{example}[Free boundary problem in porous medium flow]\label{example:dam}
A classical dam problem \cite{Baiocchi72} models steady seepage of an incompressible fluid through a porous medium, where the saturated region is unknown a priori. Using the Baiocchi transformation \cite{Baiocchi72}, this free boundary problem can be reformulated as a VI posed on a fixed domain.

For domains with vertical walls, the resulting VI is symmetric and admits a convex minimization formulation. When the geometry includes a sloping wall, however, the governing equations induce an oblique derivative boundary condition \cite{Baiocchi73a,Comincioli74}, leading to a nonsymmetric VI. \Cref{fig:dam} illustrates the setup and introduces the notation used in the description below.
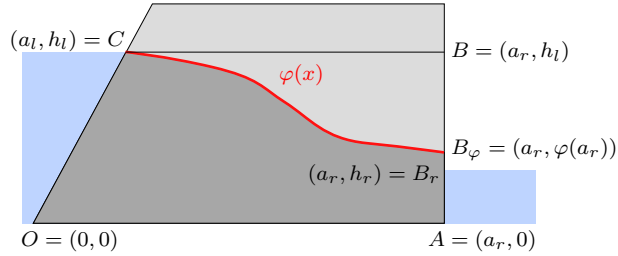
\begin{figure}[ht!]
    \begin{overpic}[scale=0.3,unit=0.6mm]{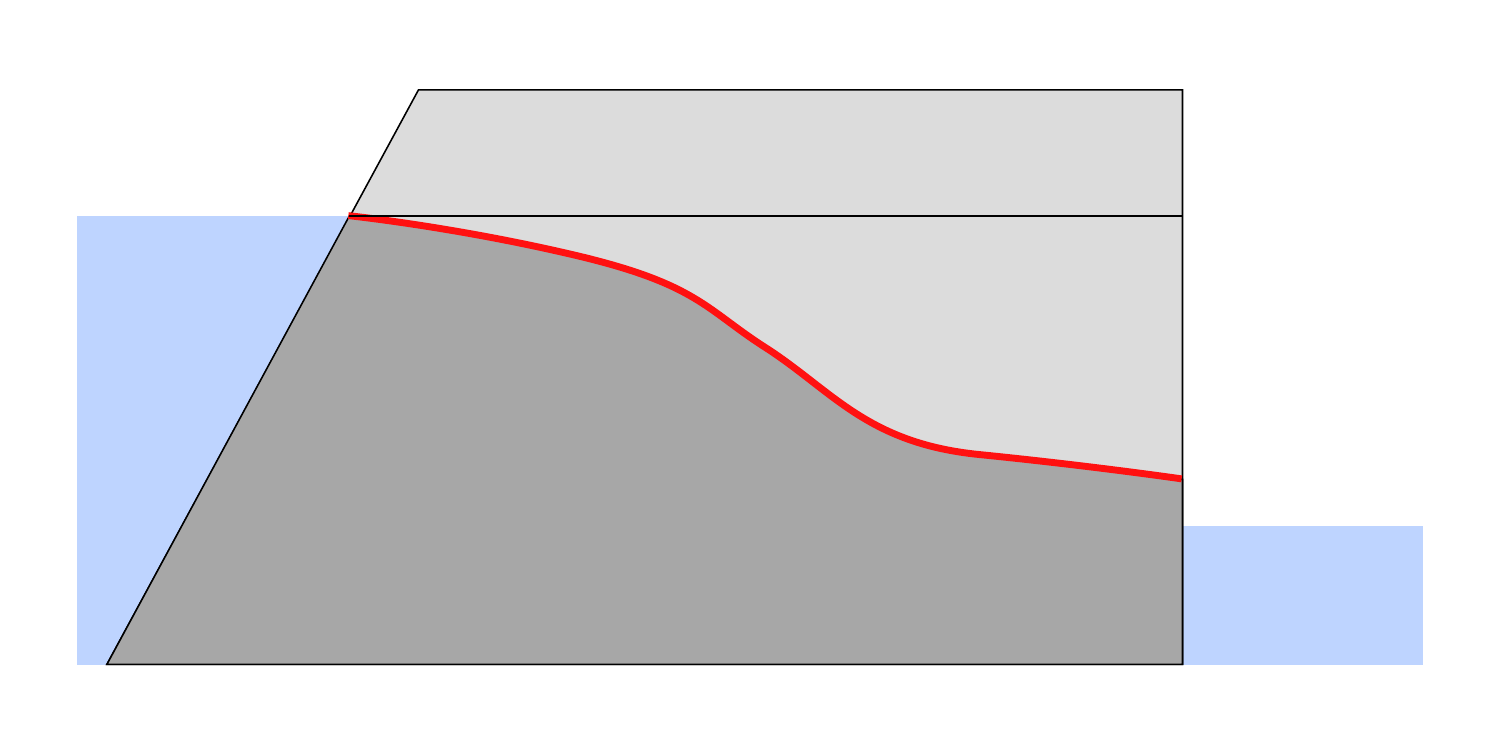}
        \put(5,2){\scriptsize $O=(0,0)$}
        \put(76,2){\scriptsize $A=(a_r,0)$}
        \put(55,13.5){\scriptsize $(a_r,h_r)=B_r$}
        \put(80,18){\scriptsize $B_\varphi=(a_r, \varphi(a_r))$}
        \put(80,35){\scriptsize $B=(a_r,h_l)$}
        \put(3,37){\scriptsize $(a_l,h_l)=C$}
        \put(50,31){\scriptsize $\color{red}{\varphi(x)}$}
    \end{overpic}
    \caption{
Two--dimensional dam with a left sloping wall.
The extended domain $\Omega=OABC$ contains the unknown saturated region, whose free surface $\varphi(x)$ satisfies $h_l\ge \varphi(x)\ge h_r$.
}
\label{fig:dam}
\end{figure}

For given discharge $q>0$, the Baiocchi potential $u \in K(q)$ satisfies the non-symmetric VI of the form \eqref{eq:general_VI} with
\begin{subequations}
\label{eq:dam}
\begin{align}
\mathcal{A}(u,v)
  &= \int_\Omega \!\left(
      \nabla u\cdot\nabla v
      + \frac{h_l}{a_l}\,(u_x v_y - v_x u_y)
    \right)\dd x, \label{dam-a} \\
\label{dam-f}
F(v)
  &= -\int_\Omega v\dd x
     - \frac{\sqrt{a_l^2+h_l^2}}{a_l}
       \int_{\Gamma_N} (y-h_l)\,v\dd s, \\
       K(q)
  & = \{ v \in H^1(\Omega) \mid v|_{\Gamma_D}=g_q,\; v \ge 0 \text{ a.e. in } \Omega \}.
\end{align}
\end{subequations}
Here $\Gamma_\mathsf{N}$ ($\overline{OC}$) denotes the sloping wall and $\Gamma_\mathsf{D}=\partial\Omega\setminus\Gamma_\mathsf{N}$.
The Dirichlet boundary data $g_q$ depend on the (a priori unknown) discharge $q>0$ and are prescribed as
\begin{align}
\label{eq:dam-bc}
g_q =
\begin{cases}
  -q(x-a_r) + \tfrac{1}{2}h_r^2, & \text{on } \overline{OA}, \\[4pt]
  \tfrac{1}{2}(h_r - y)^2,       & \text{on } \overline{AB_r}, \\[4pt]
  0,                             & \text{on } \overline{B_rB}\cup \overline{BC}.
\end{cases}
\end{align}
The discharge $q$ is determined by enforcing a compatibility condition at the junction of Dirichlet and oblique boundaries. In practice, this typically leads to a sequence of non--symmetric VIs, one for each trial value of $q$ \cite{Baiocchi73b,OdenKikuchi80}. See \Cref{sec:dam} for the details.
\end{example}

\section{The conforming proximal Galerkin method for non-symmetric VIs}\label{sec:conforming_PG}
We now introduce the necessary tools from convex analysis to present the method.
\subsection{Legendre functions}  The proximal Galerkin method relies on a suitably chosen Legendre function that encodes the geometry of the set of observables $\mathcal{O}$. In this work, it suffices to note that a function $L: \R^m \rightarrow \R \cup \{+\infty\}$ is called a Legendre function if it is proper with $\operatorname{int} (\operatorname{dom} L) \neq \emptyset $, strictly convex and differentiable on $\operatorname{int} (\operatorname{dom} L)$ with a singular gradient on the boundary of  $\operatorname{dom} L$. We first consider superposition operators of Legendre functions. That is, we define  $$\mathcal{R}(w)(x) = R(x,w(x)), \quad x \in \Omega_d, \,  w \in L^2(\Omega_d),$$
where $R: \Omega_d \times \R^m  \rightarrow \R \cup \{+ \infty\}$ is a  Carath\'eodory function such that $ R(x, \cdot)$ is a Legendre function with $\operatorname{dom}(R(x,\cdot)) = C(x)$ for almost every $x \in \Omega_d$. The PG method relies on the key observation \cite{rockafellar1967conjugates} that
\begin{equation} \label{eq:conj_grad_inverse}
 \nabla \mathcal{R}^*  = (\nabla \mathcal{R})^{-1},
\end{equation}
where $\mathcal{R}^*$ is the convex conjugate of $\mathcal{R}$
\begin{equation}
     \mathcal{R}^*(\psi)(x)
    =
    R^*(x,\psi(x))
    \,, \;\;\;  R^*(x,z)
    = \sup_{y \in \mathbb{R}}  \big\{ zy - R(x,y) \big\},
    \, \; \;
\end{equation}
and  the gradients of $\mathcal{R}$ and $\mathcal{R}^*$ are given by
$$\nabla \mathcal{R}(u)(x) = \partial_u R(x,u(x))\text{ and }\nabla \mathcal{R}^*(u)(x) = \partial_u R^*(x,u(x)).$$
We assume the supercoercivity of $R(x, \cdot)$; i.e., $R(x, y)/|y| \rightarrow \infty$ as $|y| \rightarrow \infty$ for a.e.\ $x \in \Omega$. This establishes that $R^*(x, \cdot)$ is well defined and continuously differentiable over all of $\mathbb R^m$. Along with \eqref{eq:conj_grad_inverse} and the singularity of $\partial_u R(x, \cdot)$ on $\partial C(x)$, we conclude that $\nabla \mathcal{R}^*$ is well defined and continuous over $L^\infty(\Omega_d;\R^m)$ and
\begin{equation}\label{eq:const-prsv-img}
\nabla \mathcal R^* (\psi)(x) \in \operatorname{int} C(x) \text{ f.a.e. }   x \in \Omega_d  ~\ \fa \psi \in L^\infty(\Omega_d,; \mathbb R^m) .
\end{equation}
There are many examples for the choice of $\mathcal{R}$ corresponding to a given convex set $K$. We refer to  \cite[Table 1]{dokken2025latent} for a brief list.
We note that for all examples considered in this work, a suitable choice is
\begin{equation*}
    \mathcal{R}(u)=(u-\phi)\log(u-\phi)-(u-\phi),\;\;\text{ with }\;\; \mathcal{R}^*(\psi)=\exp(\psi)+\phi.
\end{equation*}
However, other choices are also possible.

For $u \in \operatorname{dom}(\mathcal{R}), v \in \operatorname{dom}(\nabla \mathcal R)$, the  Bregman distance associated to a Legendre function $\mathcal{R}$ is given by
\begin{equation}
    \label{eq:Bregman}
    \mathcal{D}(u,v) = \mathcal{R}(u) - \mathcal{R}(v) - \nabla \mathcal{R}(v)(u - v).
\end{equation}
Throughout this work, we will invoke the following three points identity \cite[Lemma 3.1]{chen1993convergence}:
\begin{align}
\mathcal{D}(u,v) - \mathcal{D}(u,w) + \mathcal{D}(v,w) & = (\nabla \mathcal{R}(v) -\nabla \mathcal{R}(w)) ( v- u).   \label{eq:three_point_identity}
\end{align}
The dual Bregman distance is denoted by $\mathcal{D}^*$ and given by
\begin{equation*}
    \mathcal{D}^*(\chi,\psi) = \mathcal{R}^*(\chi) - \mathcal{R}^*(\psi) - \nabla \mathcal{R}^*(\psi)(\chi - \psi).
\end{equation*}
A straightforward calculation shows that $\mathcal{D}(u,v)=\mathcal{D}^*(\chi,\psi)$ when $\psi = \nabla \mathcal R(u)$ and $\chi=\nabla \mathcal R(v)$.
\subsection{The conforming proximal Galerkin method}
Considering  two discrete spaces $V_h \subset V $ and $W_h \subset W:= L^2(\Omega_d;\mathbb R^m)$, the \textit{conforming} PG method for nonsymmetric VIs is given in \Cref{alg:main_alg_discrete}.
\begin{algorithm}[htb]
\caption{The Conforming Proximal Galerkin Method}
\begin{algorithmic}[1]\label{alg:main_alg_discrete}
    \State \textbf{input:} Initial latent solution guess $\psi_h^0  \in W_h$, a sequence of positive proximity parameters $\{\alpha_k\}$, and a function $\mathcal{R}^*$ with $\nabla \mathcal{R}^*: L^{\infty}(\Omega_d;\R^m) \rightarrow \mathcal{O}$.
    \State Initialize \(k = 1\).
    \State \textbf{repeat}
    \State \quad Find $u_h^{k} \in  V_{h}$
and $\psi_h^{k} \in W_h$ such that
   \begin{subequations} \label{eq:discrete_lvpp}
\begin{alignat}{2}
\alpha_{k} \, \mathcal{A}(u^{k}_h, v_h )  + b(v_h,\psi^{k}_h - \psi_h^{k-1}) & = \alpha_k \, F(v_h)  && ~\fa v_h \in V_h, \label{eq:lvpp_g_0}\\
b(u^{k}_h, w_h) - ( \nabla \mathcal{R}^{*} (\psi^{k}_h), w_h)_{\Omega_d}& = 0 && ~\fa  w_h \in W_h.  \label{eq:lvpp_g_1}
\end{alignat}
\end{subequations}
\State \quad Assign  \(k \gets k + 1\).
    \State \textbf{until} a convergence test is satisfied.
\end{algorithmic}
\end{algorithm}

In the above,  the bilinear form $b: V_h \times W_h \rightarrow \R$  is defined by
\begin{equation}
    b(v_h, w_h) = (Bv_h, w_h)_{\Omega_d}.
\end{equation}
Here, $u_h^k$ are primal approximations, and we call $\psi_h^k$ latent variables.
In addition, we define the dual approximations
\begin{equation}
    \lambda^k_h = (\psi_h^{k-1}- \psi_h^{k})/\alpha_{k} , \quad k \geq 1, \label{eq:discrete_Lagrange}
\end{equation}
which are viewed as $\lambda_h^k \in Q'$ via $\langle \lambda_h^k, q\rangle = (\lambda_h^k, q)_{\Omega_d}$. We also define the observable approximations
\begin{equation}
    o_h^k = \nabla \mathcal{R}^*(\psi_h^k) \in \mathcal{O}, \quad k \geq 0. \label{eq:constraint_preserving_approx}
\end{equation}
Here, $o_h^k\in\mathcal{O}$ by the property of our Legendre function \eqref{eq:const-prsv-img}.
Hereinafter, we assume that the finite dimensional subspaces $V_h \subset V$ and $W_h \subset W \subset Q^\prime$ satisfy the discrete inf-sup or Ladyzhenskaya--Babu\v{s}ka--Brezzi (LBB) condition
 \begin{equation}
    \inf_{w \in W_h} \sup_{v \in V_h} \frac{ b(v,w)}{\|v\|_{V} \|w\|_{Q^\prime}} = \mu_h >  \mu_0,
\label{eq:inf_sup}
\end{equation}
where $\mu_0 > 0$ is a mesh-independent positive constant. Since, by our standing assumption,  $\operatorname{im}B = Q\hookrightarrow L^2(\Omega_d;\mathbb{R}^m)$ densely, the inf-sup condition holds on the continuous level with $W=L^2(\Omega_d)$.  Thus,  \eqref{eq:inf_sup} is true if and only if there exists a Fortin operator $\Pi_h \colon V \rightarrow V_h$ satisfying $\|\Pi_h v \|_V\lesssim \|v\|_V$ for all $v \in V$ and
\begin{equation} \label{eq:fortin_map}
   b(v - \Pi_h v, w_h) = 0
    ~\fa
    v \in V,\; w_h \in W_h;
\end{equation}
see, e.g., \cite[Lemma 26.9]{ern2021finite2}. Refer to \Cref{remark:examples_enriching_fortin} for examples of pairs $(V_h, W_h)$ satisfying \eqref{eq:inf_sup}. In what follows, we let $\|\Pi_h\|$ denote the operator norm of $\Pi_h$.
\subsection{Well--posedness of the PG iterates \eqref{eq:discrete_lvpp}}
In this section, we demonstrate that each nonlinear subproblem of \Cref{alg:main_alg_discrete} has a unique solution. We also provide stability estimates for the discrete solution variables $u_h^k$ and $\lambda_h^k$.
\begin{theorem} \label{thm:existence_conforming} Assume that \eqref{eq:inf_sup} holds.
For all $k \geq 1$, there exists a unique solution to \eqref{eq:discrete_lvpp}.
\end{theorem}
\begin{proof}
The proof follows the continuation argument in \cite[Chap.~2, Sec.~2]{kinderlehrer2000introduction}.
We drop the superscript $k$ to simplify notation and define
\[
\mathcal{A}_t(u,v) := \mathcal{A}_0(u,v) + t\,\mathcal{A}_n(u,v),\quad t\in [0,1].
\]
The proof proceeds in two steps.

\medskip
\textit{Step~1.}
Assume that, for some $0\le \tau \le 1$, the problem
\begin{equation}
\label{eq:fixed_point_funct}
\begin{alignedat}{2}
\alpha \mathcal{A}_\tau(u_h,v_h) + b(v_h,\psi_h) &= \ell(v_h)
&& \qquad\forall v_h\in V_h, \\
b(u_h,w_h) - (\nabla\mathcal{R}^*(\psi_h),w_h)_{\Omega_d} &= 0,
&& \qquad\forall w_h\in W_h,
\end{alignedat}
\end{equation}
admits a unique solution $(u_h,\psi_h)\in V_h\times W_h$ for any $\ell\in V_h'$. We show that \eqref{eq:fixed_point_funct} remains uniquely solvable with
$\mathcal{A}_t$ in place of $\mathcal{A}_\tau$ for all $t\in[\tau,\tau+t_0]$
and some $t_0>0$ independent of~$\tau$.
Given $q_h\in V_h$ and $t\in[\tau, \tau+t_0]$, let $(u_h,\psi_h)\in V_h\times W_h$ be the unique solution of
\begin{equation}
\label{eq:fixed_point_funct-shift}
\begin{alignedat}{2}
\alpha \mathcal{A}_\tau(u_h,v_h) + b(v_h,\psi_h)
&= (\tau-t)\alpha \mathcal{A}_n(q_h,v_h) + \ell(v_h),
&&\qquad \forall v_h\in V_h, \\
b(u_h,w_h) - (\nabla\mathcal{R}^*(\psi_h),w_h)_{\Omega_d}
&= 0,
&& \qquad\forall w_h\in W_h.
\end{alignedat}
\end{equation}
Define the mapping $\Phi_t:V_h\to V_h$ by $\Phi_t(q_h):=u_h$.
We show that $\Phi_t$ is a contraction for all $t\in[\tau,\tau+t_0]$.

Fix $q_h^1,q_h^2\in V_h$. Let $(u_h^1,\psi_h^1)$ and $(u_h^2,\psi_h^2)$ solve \eqref{eq:fixed_point_funct-shift} for $q_h=q_h^1$ and $q_h=q_h^2$, respectively.
Subtracting these two instances of \eqref{eq:fixed_point_funct-shift}
yields
\begin{subequations}
\label{eq:fixed_point_funct_diff}
\begin{alignat}{2}
\alpha \mathcal{A}_\tau(u_h^1-u_h^2,v_h)
+ b(v_h,\psi_h^1-\psi_h^2)
&= (\tau-t)\alpha \mathcal{A}_n(q_h^1-q_h^2,v_h),&&
\; \forall v_h\in V_h, \label{eq:fixed_point_funct_diff1}\\
b(u_h^1-u_h^2,w_h)
- (\nabla\mathcal{R}^*(\psi_h^1)-\nabla\mathcal{R}^*(\psi_h^2),w_h)_{\Omega_d}
&= 0,&&
\; \forall w_h\in W_h. \label{eq:fixed_point_funct_diff2}
\end{alignat}
\end{subequations}
Testing \eqref{eq:fixed_point_funct_diff1} with
$v_h=u_h^1-u_h^2$ and \eqref{eq:fixed_point_funct_diff2} with
$w_h=\psi_h^2-\psi_h^1$ and adding the resulting identities,
we obtain
\begin{align*}
\alpha \mathcal{A}_\tau(u_h^1-u_h^2,u_h^1-u_h^2)
+ (\nabla\mathcal{R}^*(\psi_h^1)-\nabla\mathcal{R}^*(\psi_h^2),
\psi_h^1-\psi_h^2)_{\Omega_d}
= (\tau-t)\alpha \mathcal{A}_n(q_h^1-q_h^2,u_h^1-u_h^2).
\end{align*}
Using the coercivity of $\mathcal{A}_\tau$ ($\mathcal{A}_\tau(v,v) = \mathcal{A}_0(v,v) \geq C_{\mathrm{coerc}} \|v\|^2_V$),  the continuity of
$\mathcal{A}_n$, and the monotonicity of $\nabla\mathcal{R}^*$, we deduce
\[
C_{\mathrm{coerc}}\|u_h^1-u_h^2\|_V^2
\le c_1|\tau-t|\|q_h^1-q_h^2\|_V\|u_h^1-u_h^2\|_V.
\]
Hence,
\[
\|\Phi_t(q_h^1)-\Phi_t(q_h^2)\|_V
= \|u_h^1-u_h^2\|_V
\le \frac{c_1}{C_{\mathrm{coerc}}}|\tau-t|
\,\|q_h^1-q_h^2\|_V.
\]
Therefore, for  $0<t_0<\frac{C_{\mathrm{coerc}}}{c_1}$,  $\Phi_t$ is a contraction for all $t\in[\tau,\tau+t_0]$.
By the Banach fixed-point theorem, $\Phi_t$ admits a unique fixed point
$u_h\in V_h$, and the associated $\psi_h\in W_h$ solves
\begin{equation}
\label{eq:fixed_point_funct-t}
\begin{alignedat}{2}
\alpha \mathcal{A}_t(u_h,v_h) + b(v_h,\psi_h) &= \ell(v_h),
&& \qquad\forall v_h\in V_h, \\
b(u_h,w_h) - (\nabla\mathcal{R}^*(\psi_h),w_h)_{\Omega_d} &= 0,
&&\qquad \forall w_h\in W_h.
\end{alignedat}
\end{equation}

\medskip
\textit{Step~2.}
For $\tau=0$, problem \eqref{eq:fixed_point_funct} has a unique solution by
\cite[Theorem~3.1]{aprioriPG}, since $\mathcal{A}_0$ is symmetric and coercive.
Applying Step~1 iteratively a finite number of times yields well-posedness
up to $t=1$, which concludes the proof with
$\ell=\alpha F + b(\cdot,\psi_h^{k-1})$.
\end{proof}

Having established well-posedness of the discrete problem at each iteration,
we now derive stability bounds for the  weighted averages of the iterates:
\begin{equation} \label{eq:weighted_averages}
    \overline{u}_h^\ell  = \frac{\sum_{k=1}^{\ell} \alpha_k u_h^k}{\sum_{k=1}^{\ell} \alpha_k}, \quad
    \overline{\lambda}_h^\ell = \frac{\sum_{k=1}^{\ell} \alpha_k \lambda_h^k}{\sum_{k=1}^{\ell} \alpha_k}
    = \frac{\psi_h^\ell - \psi_h^0}{\sum_{k=1}^{\ell} \alpha_k}.
\end{equation}

\begin{lemma}[Stability]\label{lemma:stability}
Assume that $\psi_h^0 \in W_h$ is chosen such that $\nabla \mathcal{R}^*(\psi_h^0) \in V$ and
$\|\nabla \mathcal{R}^*(\psi_h^0)\|_V \leq c$ for some constant $c$ independent of $h$.
Then, for every $\ell \geq 1$, we have
\begin{equation} \label{eq:stab_1}
    \|\overline{u}_h^\ell\|_V^2 + \|\overline{\lambda}_h^\ell\|_{Q'}
    \leq C_{\mathrm{stab}},
\end{equation}
where $C_{\mathrm{stab}}$ is independent of $\ell$, $\{\alpha_k\}_{k=1}^\ell$, and $h$.
\end{lemma}

\begin{proof}
Define $u_h^0 \in V_h$ as the minimal-norm function satisfying \eqref{eq:lvpp_g_1} with $k=0$; i.e.,
\begin{equation}
\label{eq:uh0_def}
u_h^0:=\underset{v\in V_h}{\operatorname{arg\,min}} \|v\|_V \quad \text{subject to} \quad
b(v, w_h) = (\nabla \mathcal{R}^*(\psi_h^0), w_h)_{\Omega_d} \quad \forall w_h \in W_h.
\end{equation}
Existence and uniqueness of $u_h^0$ follow from the inf--sup condition \eqref{eq:inf_sup}.
Moreover, standard saddle-point arguments (see, e.g., \cite[Section~4.2]{ern2004theory}) yield
\begin{equation}
\label{eq:uh0_bound}
\|u_h^0\|_{V} \lesssim \|\nabla \mathcal{R}^*(\psi_h^0)\|_{V}.
\end{equation}

Using coercivity of $\mathcal{A}$ and \eqref{eq:discrete_lvpp}, we obtain
\begin{align}
C_{\mathrm{coerc}}\|u_h^k\|_{V}^2
&\le \mathcal{A}(u_h^k, u_h^k - u_h^0) + \mathcal{A}(u_h^k, u_h^0)\label{eq:stab_pre3pt}  \\
&= -\frac{1}{\alpha_k }\, b(\psi_h^k - \psi_h^{k-1}, u_h^k - u_h^0) + F(u_h^k - u_h^0) + \mathcal{A}(u_h^k, u_h^0) \nonumber\\
&= -\frac{1}{\alpha_k}\,(\psi_h^k - \psi_h^{k-1},
\nabla \mathcal{R}^*(\psi_h^k) -  \nabla \mathcal{R}^*(\psi_h^0))_{\Omega_d} + F(u_h^k - u_h^0)
+ \mathcal{A}(u_h^k, u_h^0).  \nonumber
\end{align}
Invoking the three-point identity \eqref{eq:three_point_identity} and using that
$\mathcal{D}^*(\psi_h^{k-1},\psi_h^k)\ge 0$, we deduce
\begin{align}
\label{eq:stab_step1}
C_{\mathrm{coerc}} \alpha_k \|u_h^k\|_{V}^2
+ (\mathcal{D}^*(\psi_h^{k},\psi_h^0) - \mathcal{D}^*(\psi_h^{k-1}, \psi_h^0),1)_{\Omega_d}
\le \alpha_k ( \mathcal{A}(u_h^k, u_h^0) + F(u_h^k - u_h^0))
\end{align}
Summing \eqref{eq:stab_step1} from $k=1$ to $k=\ell$, noting that $\mathcal{D}^*(\psi_h^0, \psi_h^0)= 0$, dividing by $\sum_{k=1}^{\ell}\alpha_k$,
and applying Jensen's inequality yields
\begin{align}
\label{eq:stab_step2}
C_{\mathrm{coerc}}\|\overline{u}_h^\ell\|_V^2
+ \frac{(\mathcal{D}^*(\psi_h^\ell, \psi_h^0),1)_{\Omega_d}}{\sum_{k=1}^{\ell}\alpha_k}
& \le \mathcal{A}(\overline{u}_h^\ell, u_h^0) +  F(\overline{u}_h^\ell- u_h^0)  \\
& \nonumber
\leq (C_{\mathrm{bnd}} + c_1 + \|F\|_{V'}) \|\overline{u}_h^\ell\|_{V} \|u_h^0\|_V + \|F\|_{V'}\|u_h^0\|_V.
\end{align}
Using Young's inequality, together with
\eqref{eq:uh0_bound}, we obtain the bound \eqref{eq:stab_1} for $\|\overline{u}_h^\ell\|_V^2$.
Finally, the estimate for $\|\overline{\lambda}_h^\ell\|_{Q'}$ follows from the inf--sup
condition \eqref{eq:inf_sup} combined with \eqref{eq:lvpp_g_0}.
\end{proof}

 \subsection{Best approximation properties and error rates}
 We now derive best approximation estimates for the weighted averages \eqref{eq:weighted_averages} in \Cref{thm:best_approximation}. We also derive error rates under additional assumptions, see \Cref{thm:error_rate_conforming}.
\begin{theorem}[Best approximation estimate for the weighted averages $\overline u_h^\ell$] \label{thm:best_approximation}
For any $\ell \geq 1$, the following estimate holds
\begin{align}  \label{eq:best_approx}
\frac{C_{\mathrm{coerc}}}{4}\|\overline{u}_h^\ell - u^*\|^2_V & \leq \frac{(\mathcal{D}( o^*, o_h^0),1)_{\Omega_d}}{\sum_{k=1}^{\ell} \alpha_k} + \frac{(C_{\mathrm{bnd}} + c_1)^2  +C_{\mathrm{coerc}}^2} {2C_{\mathrm{coerc}}} \| \Pi_h u^* - u^*\|_V^2 \\
&\quad+ |\langle B' \lambda^*, \Pi_h u^* - u^* \rangle| + \inf_{v \in K} | \langle B' \lambda^*, v -\overline{u}_h^\ell  \rangle|,  \nonumber
\end{align}
where $o^* = Bu^*, o_h^0 = \nabla \mathcal{R}^*(\psi_h^0),$ and $\Pi_h $ is the Fortin map.
\end{theorem}

\begin{proof}
The proof builds and extends on arguments from \cite{aprioriPG}.  Recall the definition of $o_h^k$ in \eqref{eq:constraint_preserving_approx} and use the three point identity \eqref{eq:three_point_identity} to  derive that
\begin{equation}
\label{eq:three_points_0}
(\mathcal{D}(o^*, o_h^{k}),1)_{\Omega_d} = (\mathcal{D}(o^*, o_h^{k-1}) - \mathcal{D}(o_h^k, o_h^{k-1}),1)_{\Omega_d} + (\nabla \mathcal{R}(o_h^{k}) - \nabla \mathcal{R}(o_h^{k-1}), o_h^{k} -  o^*)_{\Omega_d}.
\end{equation}
With the fact that $\nabla \mathcal{R} = (\nabla \mathcal{R}^*)^{-1}$, \eqref{eq:discrete_lvpp} and the Fortin map \eqref{eq:fortin_map}, we obtain
\begin{align}
 (\nabla \mathcal{R}(o_h^{k}) - \nabla \mathcal{R}(o_h^{k-1}), o_h^{k} -  o^*)_{\Omega_d}
 & = (\psi_h^{k} - \psi_h^{k-1} , \nabla \mathcal R^*(\psi_h^k) - B u^*)_{\Omega_d}  \nonumber  \\
& = b(u_h^{k} - \Pi_h u^*,\psi_h^{k} - \psi_h^{k-1} ) \nonumber   \\
& = -\alpha_{k} \mathcal{A}( u_h^{k} , u_h^{k} - \Pi_h u^*) + \alpha_k F(u_h^{k} - \Pi_h u^*) \nonumber \\
& = - \alpha_{k} \mathcal{A}( u_h^{k} - \Pi_h u^* , u_h^{k} - \Pi_h u^*)  - \alpha_{k} \mathcal{A}(\Pi_h u^*, u_h^{k} - \Pi_h u^*) \nonumber \\
& \quad +\alpha_k F(u_h^{k} - \Pi_h u^*). \nonumber
\end{align}
With \eqref{eq:three_points_0} and the fact that $\mathcal{D}(o_h^k, o_h^{k-1}) \geq 0$, we obtain that
\begin{equation}
    \label{eq:three_points_1}
    \begin{aligned}
        (\mathcal{D}( o^*, o_h^{k}),1)_{\Omega_d} &+ \alpha_k \mathcal{A}( u_h^{k}
         - \Pi_h u^* , u_h^{k} - \Pi_h u^*)\\
        & \leq  (\mathcal{D}(o^*, o_h^{k-1}),1)_{\Omega_d}  - \alpha_k \mathcal{A}(\Pi_h u^*, u_h^{k} - \Pi_h u^*)
        + \alpha_k F(u_h^{k} - \Pi_h u^*).
    \end{aligned}
\end{equation}
We sum \eqref{eq:three_points_1}, use coercivity of $\mathcal{A}$, and divide by $\sum_{k=1}^\ell \alpha_k$. We obtain
\begin{multline}
\frac{1}{\sum_{k=1}^{\ell} \alpha_k} \left( (\mathcal{D}(o^*, o_h^{\ell}) -  \mathcal{D}( o^*, o_h^0),1)_{\Omega_d} + \sum_{k=1}^\ell  \alpha_k C_{\mathrm{coerc}} \|u_h^{k} - \Pi_h u^*\|_V^2  \right)  \\   \leq \mathcal{A}(\Pi_h u^*, \Pi_h u^* - \overline{u}_h^\ell) - F( \Pi_h u^* - \overline{u}_h^\ell ) .  \nonumber
\end{multline}
Using Jensen's inequality, we arrive at
\begin{equation}
\|\overline{u}_h^\ell - \Pi_h u^*\|_V^2 \leq \frac{\sum_{k=1}^{\ell} \alpha_k \|u_h^k - \Pi_h u^*\|_V^2}{\sum_{k=1}^{\ell} \alpha_k}.
\end{equation}
This yields
\begin{equation}
\begin{aligned}
    \frac{(\mathcal{D}(o^*, o_h^{\ell}),1)_{\Omega_d}}{\sum_{k=1}^{\ell} \alpha_k}    &+ C_{\mathrm{coerc}}\|\overline{u}_h^\ell - \Pi_h u^*\|_V^2 \\
    &\leq\frac{(\mathcal{D}( o^*, o_h^0),1)_{\Omega_d}}{\sum_{k=1}^{\ell} \alpha_k}    + \mathcal{A}(\Pi_h u^*, \Pi_h u^* - \overline{u}_h^\ell) -   F(\Pi_h u^* - \overline{u}_h^\ell).
    \end{aligned}
\end{equation}
We now handle the last two terms on the right-hand side above, denoted by $W$. Using \eqref{eq:def_lambda*} and \eqref{eq:lam-normal-cone}, we derive for any $v \in K$
\begin{align}
W &
= \mathcal{A}(\Pi_h u^* - u^*, \Pi_h u^* - \overline{u}_h^\ell) +
\langle B' \lambda^*, \Pi_h u^* - \overline{u}_h^\ell \rangle \\
& =  \mathcal{A}(\Pi_h u^* - u^*, \Pi_h u^* - \overline{u}_h^\ell)  + \langle B' \lambda^*, \Pi_h u^* - u^* \rangle + \langle B' \lambda^*, u^* -\overline{u}_h^\ell  \rangle \nonumber \\
& \leq \mathcal{A}(\Pi_h u^* - u^*, \Pi_h u^* - \overline{u}_h^\ell)  + \langle B' \lambda^*, \Pi_h u^* - u^* \rangle + \langle B' \lambda^*, v -\overline{u}_h^\ell  \rangle, \nonumber
\end{align}
 Continuity of $\mathcal{A}$, Young's inequality, and the fact that $v$ was arbitrary provide the result.
\end{proof}

\begin{lemma}[Best approximation estimate for the weighted averages $\overline \lambda_h^\ell$] \label{lemma:dual_variable_best}
For any $\ell \geq 1$,
    \begin{align}
\|\overline \lambda_h^\ell - \lambda^*\|_{Q'} & \lesssim \sup_{v \in V} \frac{\langle  B' \lambda^*, v - \Pi_h v \rangle}{\|v\|_V} + \|\overline{u}_h^\ell - u^*\|_{V}.
    \end{align}
\end{lemma}
\begin{proof}
The proof follows from minor modifications to \cite[Lemma 3.4]{aprioriPG}. We omit the details for brevity.
\end{proof}
For deriving a priori rates, we consider the space $V \subset H^1(\Omega)$. Note that vector-valued spaces $V \subset H^1(\Omega)^d$ are treated similarly. We require the following assumption on the Fortin map $\Pi_h$.
\begin{assumption}\label{assum:fortin}
    Assume that the Fortin operator satisfying \eqref{eq:fortin_map} is stable in the sense that
\begin{subequations}
\label{eq:stability_fortin_general}
\begin{align}
\|\Pi_h v \|_{L^2(\Omega)} & \lesssim \|v\|_{L^2(\Omega)} + h \|\nabla v\|_{L^2(\Omega)}, \label{eq:fortin_stability_assump_L2} \\
\|\nabla (\Pi_h v)\|_{L^2(\Omega)} &  \lesssim \|\nabla v\|_{L^2(\Omega)} , \label{eq:fortin_stability_assump_H1}
\end{align}
\end{subequations}
for all $v \in V$.  Further, assume the following approximation property:
For $0 \leq s \leq 1$ and $w \in H^{1+s}(\Omega) \cap V$, assume that
\begin{align}
    \| w - \Pi_h w \|_{L^2(\Omega)} + h \|\nabla (w - \Pi_hw)\|_{L^2(\Omega)}
    & \lesssim
    h^{1+s} | w |_{H^{1+s}(\Omega)},
 \end{align}
\end{assumption}
To handle the last term in \eqref{eq:best_approx}, we require the following assumption.
\begin{assumption}\label{assum:enrich}
Assume that there exists a reconstruction operator $\mathcal{E}_h: V + V_h\rightarrow V$ such that
\[
\mathcal{E}_h u_h^\ell \in K ~\fa \ell \geq 1.
\]
In addition, assume that $\mathcal{E}_h$ is affine linear: In particular,
\begin{equation}
\label{eq:EnrichingMapDecomposition}
    \mathcal{E}_h w = \mathcal{C}_h w + \varepsilon,
\end{equation}
where $\mathcal{C}_h: V +  V_h\rightarrow V$  is linear quasi-interpolant and $\varepsilon \in V$. For $s, t \in [0,1]$, assume that
\begin{subequations}
\label{eq:map_rates}
\begin{align}
      \| w - \mathcal{C}_h w \|_{L^2(\Omega)} + h  \|\nabla (w - \mathcal{C}_h w)\|_{L^2(\Omega)}
    &  \lesssim
    h^{1+s} | w |_{H^{1+s}(\Omega)} ,   \label{eq:map_Ch_rates} \\
   \|\varepsilon \|_{H^t(\Omega)} & \lesssim h^{1+s-t}.\label{eq:map_Eh_rates_r}
    \end{align}
\end{subequations}
\end{assumption}
\begin{remark}[On the validity of \Cref{assum:fortin} and \Cref{assum:enrich}] \label{remark:examples_enriching_fortin} \normalfont
Verifying these assumptions is problem-specific and depends on the choice of $V_h \times W_h$. Here, we provide examples for which these assumptions hold.
\begin{enumerate}
    \item The solution $u$ satisfies $u \geq \phi$ a.e. in $\Omega$; i.e., the constraint set $K$ is given in \eqref{eq:global_obstacle}. Note that, in this case, $B = \operatorname{Id}$, $\Omega_d = \Omega$ and a suitable choice for $\mathcal{R}$ is $\mathcal{R}(u) = (u-\phi)\ln(u-\phi) - (u-\phi)$ leading to $\nabla \mathcal{R}^*(\psi) = \exp(\psi) + \phi$. To date, we have verified that the following two choices of finite element pairs satisfy \Cref{assum:fortin} and \Cref{assum:enrich} \cite{aprioriPG}:
    \begin{itemize}
        \item $(\mathbb{P}_1\text{\normalfont-bubble}, \mathbb{P}_{0}\text{\normalfont-broken})$ pair, see \cite[Appendix B]{keith2023proximal} for more details.
\item $( \mathbb{P}_1(\mathcal{T}_h) \cap H^1_0(\Omega), \mathbb{P}_1(\mathcal{T}_h) \cap H^1_0(\Omega))$, i.e., continuous Lagrange elements.       \end{itemize}
\item The solution $u$ satisfies $u \geq \phi$ a.e. on $\Gamma_{\mathsf{S}}\subset \partial\Omega$; i.e., the constraint set $K$ is given by \eqref{eq:constraint_semi_perm}. This setting is that of \Cref{example:semi_permeable} and the Signorini problem where $B = \operatorname{tr}$ and $\Omega_d = \Gamma_{\mathsf{S}}$. Here too, a suitable choice is $\nabla \mathcal{R}^*(\psi) = \exp(\psi) + \phi$.
 \begin{itemize}
    \item  $(\mathbb{P}_1(\mathcal{T}_h) \cap H^1_0(\Omega) , \mathbb{P}_1(\Gamma_\mathsf S) \cap H^1_0(\Gamma_\mathsf S))$, i.e., continuous Lagrange elements on $\Omega$ and on $\Gamma_\mathsf S$, see \cite{aprioriPG}.
    \end{itemize}
\end{enumerate}
\end{remark}
 \begin{theorem}[A priori error estimate] \label{thm:error_rate_conforming}
 Let \Cref{assum:fortin} and \Cref{assum:enrich}
  hold. Assume that $u^* \in H^{1+s}(\Omega)$ and $B'\lambda^* \in H^{r-1}(\Omega)$ for some $r,s \in (0,1]$. Let $\|\psi_h^0\|_{L^{\infty}(\Omega_d;\mathbb R^m)} \leq c$ for a positive constant $c$ independent of $h$.  Then,
the following estimate holds for all $\ell \geq 1$:
\begin{align} \label{eq:rate_general}
 \|\overline{u}_h^\ell -  u^*\|_{H^1(\Omega)}^2  + \|\overline{\lambda}_h^\ell - \lambda^*\|_{Q'}^2 \lesssim  \frac{1}{\sum_{k=1}^{\ell} \alpha_k} + h^{2 \min\{ r,s\}}.
\end{align}
where the hidden constant depends on the true solution $u^*$ but is independent of $\ell, h $ and $\{\alpha_k\} $.
\end{theorem}
\begin{proof} It suffices to show the stated bound on $\|\overline{u}_h^\ell -  u^*\|_{H^1(\Omega)}$. The bound on the second term of \eqref{eq:rate_general} then follows from \Cref{lemma:dual_variable_best} and  \Cref{assum:fortin}.
We proceed to bound the last three terms in \eqref{eq:best_approx}.
From \Cref{assum:fortin}, it follows that
\begin{align}
\|u^* - \Pi_h u^* \|_{H^1(\Omega)}^2  \lesssim  h^{2s} |u^*|^2_{H^{1+s}(\Omega)}.  \label{eq:bound_first_term_err}
\end{align}
For the second term, we write
\begin{align*}
    \langle B' \lambda^*,  \Pi_h u^* - u^* \rangle   & \leq  \|B'\lambda^*\|_{H^{r-1}(\Omega)} \|\Pi_h u^* - u^*\|_{H^{1-r}(\Omega)} \leq Ch^{r+s} \| B'\lambda^*\|_{H^{r-1}(\Omega)}  \|u^*\|_{H^{1+s}(\Omega)},
\end{align*}
where we used \Cref{assum:fortin} along with estimates resulting from space interpolation between $L^2(\Omega)$ and $H^1(\Omega)$. For the last term, we select $v_h = \mathcal{E}_h \overline{u}_h^\ell$ and we bound
\begin{align}
| \langle B'\lambda^*, \mathcal{E}_h \overline{u}_h^\ell -\overline{u}_h^\ell  \rangle|& \leq \|B' \lambda^*\|_{H^{r-1}(\Omega)}   \|\overline{u}_h^\ell - \mathcal{E}_h \overline{u}_h^\ell\|_{H^{1-r}(\Omega)} \\
& \leq\|B' \lambda^*\|_{H^{r-1}(\Omega)}  \Big( \|(\overline{u}_h^\ell - u^*) - \mathcal{C}_h (\overline{u}_h^\ell  -  u^* )\|_{H^{1-r}(\Omega)}  \nonumber \\ & \quad   + \|  u^* - \mathcal{C}_h u^*\|_{H^{1-r}(\Omega)} + \|\epsilon\|_{H^{1-r}(\Omega)} \Big)  \nonumber \\
& \leq c \|B'\lambda^*\|_{H^{r-1}(\Omega)} ( h^r\| \overline{u}_h^\ell -  u^*\|_{H^1(\Omega)} + h^{r+s} (\|u^*\|_{H^{1+s}(\Omega)} +  1 ))\nonumber .
\end{align}
Collecting the above and applying Young's inequality yields the result.
\end{proof}

\section{The hybridizable first-order system proximal Galerkin method}\label{sec:FOSPG}
We now introduce and analyze a \textit{nonconforming} proximal Galerkin method for the non-symmetric VIs, see  \eqref{eq:general_VI}.
In particular, we study a first-order system reformulation and a hybrid mixed method with upwinding in the same spirit as \cite{egger2010hybrid, hybridMixedProx}. Note that in \cite{hybridMixedProx}, we presented FOSPG for symmetric VIs only. Our main motivation for considering this method is that it is more robust for convection-dominated problems than the conforming approach.

We focus on the following non-symmetric VI corresponding to \eqref{eq:general_VI}. Find $u \in K$ such that
\begin{subequations}\label{eq:global_obstacle}
\begin{equation}
(\kappa \nabla u , \nabla(v-u) + (\beta \cdot \nabla u, v - u) \geq (f, v-u) ~\fa v \in K,
\end{equation}
where $K$ is the closed and convex set given by
\begin{align}
K  &= \{v \in H^1_0(\Omega) \mid v \geq \phi \text{ a.e.\ in  }  \Omega \}.
\end{align}
\end{subequations}
Here, $f \in L^2(\Omega)$ and $\phi \in H^1(\Omega) \cap C(\overline{\Omega})$ with $\phi \vert_{\partial \Omega} \leq 0$. The first-order system PG (FOSPG) method for  \eqref{eq:global_obstacle} is given in \Cref{alg:FOSPG_1}.
For constraints on parts of the boundary, such as found in \Cref{example:semi_permeable}, we present the FOSPG method in \Cref{sec:semi_permeable} yet reserve its analysis for future work.
For simplicity, we hereinafter assume that the vector field $\beta$ has continuous normal components along element interfaces and that $\nabla \cdot \beta = 0 $.

\subsection{Preliminaries}
We consider the following finite element spaces:
\begin{subequations}
  \label{fes}
  \begin{align}
    \bm  \Sigma_h^p = & \;\left\{
    \mathbf{r}_h\in [L^2(\Omega)]^n: \quad \mathbf{r}_h{}_{|_{T}}
    \in \mathrm{RT}_p( T), \; \forall T \in \mathcal{T}_h
    \right\},                     \\
    V_h^p =           & \;\left\{
    {v}_h\in L^2(\Omega): \quad \quad {v}_h{}_{|_{T}}
    \in \mathbb{P}_p( T),\; \forall T \in \mathcal{T}_h
    \right\}.
  \end{align}
\end{subequations}
The jump of $v$ on a face $F \in \Gamma_h$ is defined as
\[
  [v_h]_{\vert_F} = v_h\vert_{T_F^1} - v_h \vert_{T_F^2},  \]
where $F =  \partial T_{F}^1 \cap \partial T_{F}^2$ and the normal $\bm n_F$ is chosen to point from $T_F^1$ to $T_F^2$. This choice is arbitrary but fixed. If $F \subset \partial \Omega$, then $[v_h]_{\vert_F}$ is taken as the single valued trace of $v_h$. We drop the subscript ``$\vert_{F}$" to simplify notation. The jumps of vector-valued functions are defined similarly.
Consider the following space of polynomials defined locally on each facet $F \in \Gamma_h$
\begin{align}    \label{fes_facet}
  M_{h,g}^p = & \;\bigg\{
  {\mu} \in L^2(\Gamma_h): \, {\mu}{}_{|_{F}}
\in \mathbb{P}_p(F) , \; \forall F \in \Gamma_h,
         \;\mu{}_{|_{F}} = \hat \pi_h g{}_{|_{F}},\; \forall F \subset \partial \Omega
  \bigg\},
\end{align}
where $\hat \pi_h$ is the $L^2$-projection operator onto    the polynomial space $\mathbb{P}_p(F)$.
We define the standard discontinuous Galerkin (DG) norm
\[
  \|v\|_{\mathrm{DG}}^2 = \sum_{T \in \mathcal{T}_h} \|\kappa^{1/2} \nabla v\|_{L^2(T)}^2 + \sum_{F \in \Gamma_h} h_F^{-1} \|[v]\|^2_{L^2(F)}  \quad \forall v \in H^1(\mathcal{T}_h).
\]
The above defines a norm since $\Gamma_h$ contains boundary facets. In fact, the following Poincar\'e inequality \cite[Lemma 3.2]{lasis2007hp} holds for all $r \in [1,6]$ when $d = 3$ and for all $r \in [1,\infty)$ when $d=2$,
\begin{equation}\label{eq:Poincare}
  \|v_h\|_{L^r(\Omega)} \lesssim \|v_h\|_{\mathrm{DG}} \quad \forall v_h \in V_h^p.
\end{equation}
For $\bm r \in L^2(\Omega)^n, v \in H^1(\mathcal{T}_h)$ and $\hat v \in L^2(\Gamma_h)$, define
\begin{equation} \label{eq:def_triple_norm}
  \vvvert(\bm r, v, \hat v)\vvvert^2 = \|\kappa^{-1/2} \bm r\|^2_{L^2(\Omega)} + \|\kappa^{1/2} \nabla_h v\|^2_{L^2(\Omega)} + \sum_{T \in \mathcal{T}_h} h_T^{-1} \| v - \hat v\|^2_{L^2(\partial T)} .
\end{equation}
We remark that for any $v \in H^1(\mathcal{T}_h)$, $\bm r \in L^2(\Omega)^n$ and $\hat{v} \in L^2(\Gamma_h)$ with $ \hat{v}= 0 $ on $\partial \Omega$, we have from a triangle inequality and shape regularity that
\begin{equation}
  \|v\|_{\mathrm{DG}} \lesssim \vvvert(\bm r , v , \hat{v}) \vvvert. \label{eq:dg_to_triple}
\end{equation}
Further, we define the following dual norm:
\begin{align}
  \|v\|_{H^1(\mathcal{T}_h)^*} = \sup_{w \in H^1(\mathcal{T}_h) , w \neq 0} \frac{(v, w)}{\|w\|_{\mathrm{DG}}} \quad \forall v \in L^2(\Omega). \label{eq:negative_norm}
\end{align}
The $L^2$-projection onto the space $V_h^p$ is denoted by  $\Pi_h: H^1(\mathcal{T}_h) \rightarrow V_h^p$. We have the following properties:
\begin{equation}
  (\Pi_h v, p_h) = (v, p_h) \,\,\,\, \forall p_h \in V_h^p,  \qquad \|\Pi_h v\|_{\mathrm{DG}} \lesssim \|v\|_{\mathrm{DG}}.  \label{eq:local_L2_project}
\end{equation}
The $L^2$ projection allows us to obtain the following bound that will be useful
\begin{align} \label{eq:negative_norm_proj}
  \|\psi_h\|_{H^1(\mathcal{T}_h)^*} =  \sup_{w \in H^1(\mathcal{T}_h)} \frac{(\psi_h, w)}{\|w\|_{\mathrm{DG}}} =\sup_{w \in H^1(\mathcal{T}_h)} \frac{(\psi_h, \Pi_h w)}{\|w\|_{\mathrm{DG}}}  \lesssim \sup_{w \in H^1(\mathcal{T}_h)} \frac{|(\psi_h, \Pi_h w)|}{\|\Pi_h w\|_{\mathrm{DG}}} .
\end{align}
We also make use of the following lifting operator. For a given $(u_h  ,\hat u_h) \in V_h \times M_{h,0}^p$, define the lifting  $\bm L_h(u_h - \hat u_h) \in \bm \Sigma_h^p$ such that locally $\bm L_h(u_h  - \hat u_h ){}_{\vert T} \in \mathrm{RT}_{p}(T)^n$ solves
\begin{equation}
  \int_T \bm L_h (u_h - \hat u_h) \cdot \bm r_h = \int_{\partial T} (u_h {}_{\vert T} - \hat u_h) \bm r_h \cdot \bm n, \quad \forall \bm r_h \in \mathrm{RT}_{p}(T)^n.  \label{eq:lift}
\end{equation}
Testing \eqref{eq:lift} with $\bm r_h = \bm L_h (u_h - \hat u_h) $, applying the Cauchy--Schwarz and  the trace inequality $\|\bm r_h \cdot \bm n \|_{L^2(\partial T)} \lesssim h_T^{-1/2} \|\bm r_h\|_{L^2(T)}$, one readily derives that
\begin{equation}
  \|\bm L_h (u_h - \hat u_h) \|_{L^2(T)} \lesssim h_T^{-1/2} \|u_h - \hat u_h\|_{L^2(\partial T)}, \quad \forall T \in \mathcal{T}_h.
\end{equation}
\subsection{FOSPG for the nonsymmetric VI}

To present the FOSPG method, one introduces the flux variable $\bm q = - \kappa \nabla u$, rewrites \eqref{eq:global_obstacle} as a first-order system, and applies the well-known DG and hybridization machinery. We refer to \cite{egger2010hybrid} for more details on the derivation of the forms given below.
Define the following bilinear form, $\mathcal{B}: \bm{\Sigma}_h^p \times H^1(\mathcal{T}_h) \times L^2(\Gamma_h) \rightarrow \R $
\begin{equation} \label{eq:def_form_Bh}
  \mathcal{B}(\bm q_h, (v, \hat v)) = ( \bm q_h, \nabla_h v) - (v - \hat v, \bm q_h \cdot \bm n)_{\partial \mathcal T_h}.
\end{equation}
Considering the definition of $\mathcal{B}$ in \eqref{eq:def_form_Bh} and the definition of $\bm L_h(u_h - \hat u_h)$ \eqref{eq:lift}, it is useful to  note that for any $(\bm w_h , v_h, \hat v_h) \in \bm \Sigma_h^p \times V_h^p \times M_{h,g}^p$ ,
\begin{equation} \label{eq:link_Bh_lift}
  \mathcal{B}(\bm w_h, (v_h, \hat v_h) ) = (\bm w_h, \nabla_h v_h) - (\bm L_h (v_h - \hat v_h) , \bm w_h).
\end{equation}
We define the following form pertaining the hybrid mixed discretization of $-\nabla \cdot (\kappa \nabla \cdot)$ operator:  $\mathcal{A}_{\mathsf{D}}: (\bm{\Sigma}_h^p \times V_h^p \times M_{h,g}^p)^2 \rightarrow \R$:
\begin{equation} \label{eq:def_form_A}
 \mathcal{A}_{\mathsf{D}}((\bm q_h, u_h, \hat u_h), (\bm r_h, v_h, \hat v_h)) =  (\kappa^{-1} \bm q_h, \bm r_h) + \mathcal{B} (\bm r_h, (u_h, \hat u_h)) -   \mathcal{B}(\bm q_h, (v_h, \hat v_h)).
\end{equation}
The following continuity property follows from standard arguments
\begin{align} \label{eq:continuity_diff}
\mathcal{A}_{\mathsf{D}}((\bm q_h, u_h, \hat u_h), (\bm v_h, v_h, \hat v_h)) & \lesssim
\vvvert (\bm q_h, u_h, \hat u_h) \vvvert \vvvert (\bm v_h, v_h, \hat v_h) \vvvert.
\end{align}
We note that for a given $(\bm q_h, u_h, \hat u_h) \in \bm \Sigma_h^p \times V_h^p \times M_{h,0}^p$, there exists $\bm r_h \in \bm \Sigma_h^p$ such that  \cite{egger2010hybrid}
\begin{equation} \label{eq:inf_sup_HDG}
\mathcal{A}_\mathsf{D}((\bm q_h, u_h, \hat u_h), (\bm r_h, u_h, \hat u_h)) \gtrsim \vvvert (\bm q_h, u_h, \hat u_h) \vvvert^2.
\end{equation}

The advection term is discretized with the following form $ \mathcal{A}_\mathsf{C}: (V_h^p \times M_{h,g}^p)^2 \rightarrow \R$ that incorporates upwind stabilization:
\begin{equation}
  \mathcal{A}_\mathsf{C}((u_h, \hat u_h), (v_h, \hat v_h))
  =\;
  - (\beta u_h, \nabla_h v_h) + (\beta \cdot \bm n u_h^{\mathrm{up}}, (v_h - \hat v_h))_{\partial \mathcal{T}_h},
\end{equation}
where
\begin{align}
  u_h^\mathrm{up} =
\begin{cases}
  \hat u_h & \text{if }\beta \cdot \bm n < 0, \\
    u_h  & \mathrm{otherwise}.
\end{cases}
\end{align}
From \cite[Proposition 3.5]{egger2010hybrid}, we have that for any $(u_h, \hat u_h) \in V_h^p \times M_{h,g}^p$,
\begin{align}
\label{eq:nonnegativity_convection}
 \mathcal{A}_\mathsf{C}((u_h, \hat u_h), (u_h, \hat u_h)) \geq 0.
\end{align}
Finally, the following continuity estimate follow from standard arguments
\begin{align}
\mathcal{A}_\mathsf{C}((u_h, \hat u_h), (v_h, \hat v_h)) \lesssim (\|u_h\|_{L^2(\Omega)} + \vvvert (\bm 0, u_h, \hat u_h) \vvvert) (\|v_h\|_{L^2(\Omega)} + \vvvert (\bm 0, v_h, \hat v_h) \vvvert),
\label{eq:continuity_convection}
\end{align}
where the hidden constant depends on $\|\beta\|_{L^\infty(\Omega)}$ and $\|\beta \cdot \bm n\|_{L^\infty(\partial \mathcal{T}_h)}$.
\begin{algorithm}[htb]
  \caption{The Hybridized First Order System Proximal Galerkin Method for \eqref{eq:global_obstacle}}
  \label{alg:FOSPG_1}
  \begin{algorithmic}[1]\label{alg:main_alg_disc_FOSPG}
    \State \textbf{input:}  A discrete latent solution guess $\psi_h^0 \in V^q_h$ with  $q \leq p$ and a sequence of positive step sizes $\{\alpha_k\}$.
    \State Initialize \(k = 1\).
    \State \textbf{repeat}
    \State \quad
    Find $(\bm q^{k}_h, u_h^{k} , \hat u^{k}_h) \in \bm \Sigma_h^p  \times V_h^p \times  M_{h,0}^p$ and $\psi_h^{k} \in V_h^q$ such that
    \begin{subequations}
      \label{eq:nonlinear_subproblem}
      \begin{alignat}{2}
        \mathcal{A}_\mathsf{D}((\bm q^k_h, u^k_h, \hat u^k_h), (\bm r_h, v_h, \hat v_h))  + \mathcal{A}_\mathsf{C}((u^k_h, \hat u^k_h), (v_h, \hat v_h)) +   \frac{1}{\alpha_k}(\psi_h^{k} - \psi_h^{k-1} , v_h) & =  (f,v_h) ,  \label{eq:nonlinear_subproblem_0} \\
        (u_h^{k}, w_h)  - (\nabla \mathcal{R}^*(\psi_h^{k}),w_h)                                               & = 0,  \label{eq:nonlinear_subproblem_1}
      \end{alignat}
    \end{subequations}
    for all $ (\bm r_h , v_h, \hat v_h) \in  \Sigma_h^p \times V_h^p \times M_{h,0}^p $ and  $w_h \in V^q_h$.
    \State \quad Assign \(k \gets k + 1\).
    \State \textbf{until} a convergence test is satisfied.
  \end{algorithmic}
\end{algorithm}
\begin{remark}[Mixed boundary conditions]
In the case of mixed Neumann ($\nabla u \cdot n = 0 $ on $\Gamma_{\mathsf N}$) and Dirichlet type boundary conditions $(u = 0$ on $\Gamma_{\mathsf{D}}$), with $\Gamma_\mathsf{D} \cup \Gamma_{\mathsf N} = \partial \Omega$, the following modifications are required on the FOSPG scheme of \Cref{alg:FOSPG_1}.  The space $M_{h,0}^p$ of \eqref{fes_facet} and the discretization of the convection term are  modified to
\begin{align*}
M_{h,0}^p  & =  \;\bigg\{
  {\mu}\in L^2(\Gamma_h): \, {\mu}{}_{|_{F}}
\in \mathbb{P}_p( F) , \; \forall F \in \Gamma_h,
  \;\mu{}_{|_{F}} = 0,\; \forall F \subset \Gamma_{\mathsf{D}}
  \bigg\}, \\
  \mathcal{A}_\mathsf{C}((u_h, \hat u_h), (v_h, \hat v_h))
 &  =\;
  - (\beta u_h, \nabla_h v_h)
   + (u_h^{\mathrm{up}}, v_h)_{\partial\mathcal{T}_h}  - (u_h^{\mathrm{up}}, \hat v_h)_{\partial\mathcal{T}_h \backslash \Gamma_{\mathsf{N}}}.
  \end{align*}
\end{remark}
The goal of the \Cref{lemma:diff_only_existence} is to show existence and uniqueness for the case of zero convection, which will be essential in proving well-posedness of \eqref{eq:nonlinear_subproblem} in \Cref{thm:FOSPG_existence}.
\begin{lemma} \label{lemma:diff_only_existence}
  For any bounded linear functional $\ell: V_h^p \times M_{h,0}^p \rightarrow \R$ and $\alpha>0$, there exists a unique solution to the following problem: Find $(\bm q_h, u_h , \hat u_h) \in \bm \Sigma_h^p  \times V_h^p \times  M_{h,0}^p$ and $\psi_h \in V_h^q$ such that
\begin{subequations}\label{eq:nonlinear_subproblem_diff}
    \begin{alignat}{2}
     \mathcal{A}_\mathsf{D}((\bm q_h, u_h, \hat u_h), (\bm r_h, v_h, \hat v_h))  +   \frac{1}{\alpha}(\psi_h , v_h) & =  \ell(v_h, \hat v_h) , \label{eq:nonlinear_subproblem_diff_0}   \\
        (u_h, w_h)  - (\nabla \mathcal{R}^*(\psi_h),w_h)              & = 0,  \label{eq:nonlinear_subproblem_diff_1}
    \end{alignat}
  \end{subequations}
  for all $ (\bm r_h, v_h, \hat v_h) \in \bm \Sigma_h^p  \times V_h^p \times M_{h,0}^p $ and    $w_h \in V_h^q$.
\end{lemma}
\begin{proof}
  \textit{Step 1} (Reformulation). We start with a reformulation of the problem using lifting operators. Assume that $(u_h, \hat u_h) \in V_h^p \times M_{h,0}^p$ and $\psi_h \in V_h^q$  uniquely solves the following:
\begin{subequations}\label{eq:reformulation}
    \begin{alignat}{2}\label{eq:reformulation_0}
      a_h((u_h,\hat u_h) , (v_h, \hat v_h)) + \frac{1}{\alpha} (\psi_h , v_h) & = \ell(v_h, \hat v_h)  && \quad \forall (v_h, \hat v_h) \in V_h^p \times M_{h,0}^p, \\
      (u_h, w_h)  - (\nabla \mathcal{R}^*(\psi_h),w_h)             & = 0, &&  \quad \forall w_h \in  V_h^q, \label{eq:reformulation_1}
    \end{alignat}
  \end{subequations}
  where $a_h$ is given by
  \begin{equation}\label{eq:def_a_lift}
    a_h((u_h, \hat u_h), (v_h, \hat v_h)) =   (\kappa (\nabla_h u_h - \bm L_h(u_h -\hat u_h )) , \nabla_h v_h - \bm L_h(v_h - \hat v_h)).
  \end{equation}
  Define  \begin{equation}
    \bm q_h =    -\kappa \nabla_h u_h + \kappa \bm L_h(u_h - \hat u_h). \label{eq:recovering_flux}
  \end{equation}
  Then,  $(\bm q_h , u_h ,\hat u_h) \in \bm \Sigma_h^p  \times V_h^p \times  M_{h,g}^p$ uniquely solves \eqref{eq:nonlinear_subproblem_diff}. To see this, first test $\eqref{eq:recovering_flux}$ with $\bm v_h \in \bm \Sigma_h^p$ and use \eqref{eq:lift} to recover that $\mathcal{A}_{\mathsf{D}}((\bm q_h, u_h,\hat u_h), (\bm r_h, 0 , 0)) = 0$  for all $\bm r_h \in \bm \Sigma_{h}^p$.  This observation along with substituting the definition of $\bm q_h$ \eqref{eq:recovering_flux} in \eqref{eq:def_a_lift} and using \eqref{eq:link_Bh_lift} and \eqref{eq:reformulation_0} yields \eqref{eq:nonlinear_subproblem_diff_0}. Further, from the monotonicity of $\nabla \mathcal{R}^*$, non-negativity of $\mathcal{A}_{\mathsf C}$ \eqref{eq:nonnegativity_convection}, and the inf--sup stability of $\mathcal{A}_{\mathsf{D}}$ \cite[Proposition 3.2]{egger2010hybrid}, one readily obtains that  \eqref{eq:nonlinear_subproblem_diff} has unique solutions.
  
  \textit{Step 2} (Existence and uniqueness of solutions).   Define $\mathcal{L}: V_h^p \times  M_{h,0}^p \times V_h^q \rightarrow \R$ as
  \begin{equation} \label{eq:lagrangian}
    \mathcal{L}( w , \hat w, \varphi ) =
    \frac12 a_h((w , \hat w ), (w,\hat w))
    - \ell(w,\hat w)
    + \frac1\alpha (w , \varphi)
    -    (\mathcal{R}^*(\varphi),1).
  \end{equation}
  We now show the existence of a saddle point to $\mathcal{L}$ which solves \eqref{eq:reformulation}.
                                              We first demonstrate that $a_h$ is coercive with respect to a norm on $V_h^p \times M_{h,0}^p$. From \cite[Lemma 3.1]{egger2010hybrid}, see also \cite[Lemma 3.1]{gao2018error} for a detailed proof, there exists $\bm{\tau}_h \in \bm \Sigma_h^p$ such that for any $T \in \mathcal{T}_h$ and for all $ \bm p \in \mathbb P_{p-1}(T)^n$ and $ q \in \mathbb P_p(\partial T), $
  \begin{align}
    \label{eq:def_test_infsup}
    (\bm \tau_h, \bm p)_T + ( \bm \tau_h \cdot\bm{n}, q )_{\partial T} &  = (\nabla_h u_h, \bm p)_{T}
    +
    ( h_T^{-1} (\hat u_h -  u_h) , q )_{\partial T}, \;\;   \|\bm \tau_h \|_{L^2(\Omega)} \lesssim \vvvert (\bm 0, u_h, \hat u_h) \vvvert
          \end{align}
  Now, observe that for a constant $C_\kappa$ depending on the diffusion coefficient $\kappa$, we have
  \begin{multline}
    \vvvert (\bm 0, u_h, \hat u_h) \vvvert^2 \leq C_\kappa   \mathcal{B}(\bm \tau_h, (u_h, \hat u_h))
    =  C_\kappa (\bm \tau_h, \nabla_h u_h - \bm L_h(u_h -\hat u_h) )\\
    \leq C_\kappa  \|\bm \tau_h\|_{L^2(\Omega)} \|\nabla_h u_h - \bm L_h(u_h -\hat u_h)\|_{L^2(\Omega)}\lesssim C_\kappa \vvvert (\bm 0, u_h, \hat u_h) \vvvert\|\nabla_h u_h - \bm L_h(u_h -\hat u_h)\|_{L^2(\Omega)}.
  \end{multline}
 From the above and \eqref{eq:def_a_lift}, it then follows that
  \begin{equation}
    a_h((u_h, \hat u_h),(u_h,\hat u_h)) = \|\kappa^{1/2} (\nabla_h u_h - \bm L_h(u_h -\hat u_h))\|^2_{L^2(\Omega)} \gtrsim \vvvert (\bm 0, u_h, \hat u_h) \vvvert^2.
  \end{equation}
 Therefore, since $\vvvert (\bm 0,\cdot, \cdot) \vvvert $ defines a norm on $V_h^p \times M_{h,0}^p$, we can conclude coercivity. Finally, we note that the following inf-sup condition holds:
 \begin{equation}
    \inf_{w \in V^q _h} \sup_{v \in V^p_h} \frac{ (v,w)}{\|v\|_{\mathrm{DG}} \|w\|_{H^1(\mathcal{T}_h)^*}}  \geq \beta, \label{eq:inf_sup_DG}
     \end{equation}
    for some $\beta > 0$.  This follows from using the definition \eqref{eq:negative_norm} and the $L^2$ projection onto $V^q_h$, see \eqref{eq:local_L2_project} as the Fortin map, see \cite[Lemma 26.9]{ern2021finite2}.  From here, one applies the arguments in \cite[Theorem 3.1]{aprioriPG} to conclude. We skip the details for brevity.
                                                         \end{proof}

\begin{theorem}[Existence and uniqueness of solutions to \Cref{alg:FOSPG_1}]\label{thm:FOSPG_existence}
  For each $k \geq 1$, there exists a unique solution to \eqref{eq:nonlinear_subproblem}.
\end{theorem}
\begin{proof}
Here, we apply
the Leray--Schauder fixed point Theorem \cite[Theorem 9.12-3]{ciarlet2013linear}. We drop the superscript $k$ and define $\tilde f = \alpha f + \psi_h^{k-1}$. Consider the mapping $\bm \Phi : \bm \Sigma_h^p  \times V_h^p \times  M_{h,0}^p \times [0,1] \rightarrow   \bm \Sigma_h^p  \times V_h^p \times  M_{h,0}^p $ where $\bm \Phi (\bm w_h, w_h, \hat w_h, \sigma) = (\bm q_h , u_h, \hat u_h)$ is the unique solution to the following problem
  \begin{subequations}
    \label{eq:def_Phi}
    \begin{alignat}{2}
      \alpha \mathcal{A}_\mathsf{D}((\bm q_h, u_h, \hat u_h), (\bm v_h, v_h, \hat v_h)) +   (\psi_h , v_h) & =  - \sigma \alpha \mathcal{A}_{\mathsf{C}} ( (w_h, \hat w_h) , (v_h, \hat v_h))  + \sigma (\tilde f, v_h)
      \label{eq:def_Phi_0}
      \\
      (u_h, q_h)  - (\nabla \mathcal{R}^*(\psi_h) , q_h)                                                   & = -(\nabla \mathcal{R}^*(0)(1-\sigma), q_h). \label{eq:def_Phi_1}
    \end{alignat}
  \end{subequations}
  for all $ (v_h, \hat v_h) \in V_h^p \times M_{h,0}^p $, $\bm v_h \in \bm \Sigma_h^p $ and  $q_h \in V_h^q$. Observe that $\bm \Phi$ is well defined by \Cref{lemma:diff_only_existence}. Further,  $\bm \Phi(\bm w_h, w_h, \hat w_h, 0 ) = (\bm 0 , 0 , 0)$ since $(\bm 0, 0, 0)$ clearly solves \eqref{eq:def_Phi} and solutions to \eqref{eq:def_Phi} are unique by \Cref{lemma:diff_only_existence}. We now show that fixed points $(\bm q_h, u_h, \hat u_h)$ satisfying
  \begin{equation} \label{eq:fixed_points_Phi}
    \bm \Phi(\bm q_h, u_h, \hat u_h, \sigma) = (\bm q_h, u_h, \hat u_h),
  \end{equation}
  are bounded uniformly bounded with respect to $\sigma \in [0,1]$.  To this end, consider \eqref{eq:def_Phi} with $(w_h, \hat w_h) = (u_h, \hat u_h)$, test \eqref{eq:def_Phi_0} with $(\bm v_h, v_h, \hat v_h)  = (\bm r_h, u_h, \hat u_h)$  for $\bm r_h$ satisfying \eqref{eq:inf_sup_HDG},  test  \eqref{eq:def_Phi_1} with $q_h = \psi_h$, and subtract the resulting equations. We obtain
  \begin{multline}
    \alpha \vvvert (\bm q_h, u_h ,\hat u_h)\vvvert^2 + (\nabla \mathcal{R}^*(\psi_h) - \nabla \mathcal{R}^*(0), \psi_h) + \sigma \alpha \mathcal{A}_\mathsf{C}((u_h, \hat u_h), (u_h, \hat u_h))
    \\ \lesssim \sigma (\tilde f, u_h) - \sigma (\nabla \mathcal{R}^*(0) , \psi_h).
  \end{multline}
  Since $\nabla \mathcal{R}^*$ is strictly monotone, the second term above is positive. Similarly, from \eqref{eq:nonnegativity_convection}, the third term is non-negative.  Therefore, using that $\sigma \leq 1$, Cauchy--Schwarz inequality, and the definition of $\|\psi_h\|_{H^1(\mathcal{T}_h)^*}$,  we obtain that
  \begin{align}\label{eq:bd_fixed_pts_0}
    \alpha \vvvert (\bm q_h, u_h ,\hat u_h)\vvvert^2 \lesssim \|\tilde f\|_{L^2(\Omega)}\|u_h\|_{L^2(\Omega)} + \|\nabla \mathcal{R}^*(0)\|_{\mathrm{DG}} \|\psi_h\|_{H^1(\mathcal{T}_h)^*}.
  \end{align}
  From \eqref{eq:Poincare} and \eqref{eq:dg_to_triple}, we obtain that
  \begin{equation}
         \|u_h\|_{L^2(\Omega)} \lesssim \|u_h\|_{\mathrm{DG}} \lesssim  \vvvert (\bm q_h, u_h ,\hat u_h)\vvvert.
  \label{eq:specialized_Poincare}
  \end{equation}

  To bound the second term in \eqref{eq:bd_fixed_pts_0}, we utilize \eqref{eq:negative_norm_proj} and test \eqref{eq:def_Phi_0} with $(\bm 0, \Pi_h w, \hat w)$ where $\hat w \in M_{h,0}^p$ is given by
  \begin{equation}\label{eq:triple_to_dg}
    \hat w =
    \begin{cases}
      \frac12 (\Pi_h w \vert_{T_F^1} + \Pi_h w \vert_{T_F^2}) & F \in \Gamma_h^0, \quad  F =\partial T_F^1 \cap \partial T_F^2 , \\
      0                                                       & F \in \Gamma_h^{\partial} .
    \end{cases}
  \end{equation}
  We obtain that
  \begin{multline}
    (\psi_h , \Pi_h w)   =  - \alpha \mathcal{A}_\mathsf{D}((\bm q_h, u_h, \hat u_h), (\bm 0, \Pi_h w, \hat w)) - \sigma \alpha \mathcal{A}_{\mathsf{C}}(  (u_h, \hat u_h) , (\Pi_h w, \hat  w))  + \sigma (\tilde f, \Pi_h w)
  \end{multline}
  Using continuity of $\mathcal{A}_\mathsf{D}$ \eqref{eq:continuity_diff} and $\mathcal{A}_{\mathsf{C}}$ \eqref{eq:continuity_convection} followed by \eqref{eq:specialized_Poincare}, we obtain that
  \begin{align}
    |(\psi_h , \Pi_h w) | \lesssim \vvvert (\bm q_h, u_h, \hat u_h)  \vvvert  \vvvert (\bm 0, \Pi_h w, \hat w) \vvvert + \|\tilde f \|_{L^2(\Omega)} \|\Pi_h w\|_{L^2(\Omega)}.
  \end{align}
  From the observation that $ \vvvert (\bm 0, \Pi_h w, \hat w) \vvvert \lesssim \|\Pi_h w\|_{\mathrm{DG}}$, and \eqref{eq:negative_norm_proj}, we obtain that
  \begin{equation} \label{eq:negative_norm_psi}
    \|\psi_h\|_{H^1(\mathcal{T}_h)^*} \lesssim \vvvert (\bm q_h, u_h, \hat u_h)  \vvvert + \|\tilde f\|_{L^2(\Omega)}.
  \end{equation}
  Substituting \eqref{eq:negative_norm_psi} in \eqref{eq:bd_fixed_pts_0} and reusing the resulting bound in \eqref{eq:negative_norm_psi} shows that
  \begin{equation}
    \vvvert (\bm q_h, u_h, \hat u_h)  \vvvert +   \|\psi_h\|_{H^1(\mathcal{T}_h)^*}  \lesssim \|\tilde f\|_{L^2(\Omega)}+ \|\nabla \mathcal{R}^*(0)\|_{\mathrm{DG}}.
  \end{equation}
  This provides a uniform bound on the fixed points \eqref{eq:fixed_points_Phi} independent of $\sigma$.  Compactness of $\bm \Phi$ follows from continuity since the spaces are finite-dimensional. Therefore, by an application of the Leray--Schauder Theorem, there exists a fixed point for $\bm \Phi (\cdot, \cdot, \cdot, 1)$. This fixed point is a solution to \eqref{eq:nonlinear_subproblem}.  Uniqueness follows from the monotonicity of $\nabla \mathcal{R}^*$; we omit the details for brevity.
\end{proof}

\subsection{Error rates for FOSPG applied to the non--symmetric VI defined in \eqref{eq:global_obstacle}}
We begin by modifying and studying the properties of the reconstruction and Cl\'ement maps from \cite{Fuhrer+2024+363+378, aprioriPG} when applied to functions in $H^1(\mathcal{T}_h)$.
The first step is to construct  the  map $\mathcal{C}_h : H^1(\mathcal{T}_h) \times H^{1/2}(\Gamma_h^\partial) \rightarrow H^1(\Omega) \cap V_h$. Here, 
$$H^{1/2}(\Gamma_h^\partial) := \{v \in L^2(\partial \Omega) \mid v \vert_F \in H^{1/2}(F) ~\fa \, F \in \Gamma_h^\partial \}.$$
We denote by $\mathcal{N}_h$ the set of all interior nodes $z$ of the mesh $\mathcal{T}_h$ and by $\omega_z$ the star patch containing elements sharing the node $z$.  We set
\begin{equation} \label{eq:clement_case_1}
\mathcal{C}_h (v,\hat v)  = \sum_{z \in \mathcal{N}_h } v_z  \varphi_z,  \quad v_z = \sum_{ T \subset \omega_z } \frac{\alpha_{z,T}}{|T|}\int_{T} v \dd x \, \text{ if } z \in \mathcal{N}_h.
\end{equation}
Here,  $\varphi_z$ is the Lagrange nodal basis function and $\{\alpha_{z,T}\}_{T \subset \omega_z}$ are weights selected such that \cite{Fuhrer+2024+363+378} $$ z = \sum_{T \subset \omega_z} \alpha_{z,T} s_T, \quad  \sum_{T \subset \omega_z} \alpha_{z,T} = 1,  \quad \alpha_{z,T} \geq 0,
$$
where $s_T$ is the centroid of an element $T \in \mathcal{T}_h$.
The above weighting ensures that for any $p \in \mathbb{P}^1(\omega_z),$ $\mathcal{C}_h p(z) = p(z)$ for all $z \in \mathcal{N}_h$. This is a key realization to obtain optimal error rates \cite{Fuhrer+2024+363+378}.  For the boundary nodes, we use the dual basis associated to an edge on the boundary following \cite{scott1990finite}. That is, for $z \in \partial \Omega$ , select a  face $F_z \subset \partial \Omega$ such that $z \in F_z$ and let $\chi_z \in \mathbb P^1(F_z)$ be the dual basis function satisfying
\[
\int_{F_z} \chi_z \varphi_z \mathrm{d} s = 1 \text{ and } \int_{F_z} \chi_z \varphi_{z'}\mathrm{d} s = 0,
\]
for all nodes $z' \in F_z$ with $z ' \neq z$. We then define
\begin{equation}
v_z =  \int_{F_z} \hat v \chi_z \mathrm{d}s , \quad z \in \partial \Omega.
\end{equation}
Observe that the construction of $\mathcal{C}_h$ depends on the choice of $F_z$;  we choose not to include this dependency in the notation for simplicity. \begin{lemma}[Properties of $\mathcal{C}_h$] \label{lemma:modified_Clement}
For any $(v, \hat v) \in H^1(\mathcal{T}_h) \times H^{1/2}(\Gamma_h^\partial ),$
\begin{equation}
\label{eq:stability_clement_modified}
\|\mathcal{C}_h (v, \hat v) \|_{L^2(\Omega)} \lesssim   \left( \|v\|_{L^2(\Omega)}^2 + \sum_{F \in \Gamma_h^\partial }  h_{F} \|\hat v \|_{L^2(F)}^2 \right)^{1/2}.
\end{equation}
For any $(u_h , \hat u_h)
\in V_h^p \times M^p_{h,0}$,
\begin{equation}
\label{eq:L2_estimate_Clement}
\|u_h - \mathcal{C}_h(u_h, \hat u_h) \|_{L^2(\Omega)} \lesssim h \left(  \| u_h\|^2_{\mathrm{DG}} + \sum_{F \in \Gamma_h^\partial }h_{F}^{-1} \| u_h - \hat u_h\|^2_{L^2(F)} \right)^{1/2}.
\end{equation}
\end{lemma}
\begin{proof}
To show \eqref{eq:stability_clement_modified}, we bound the nodal values of this interpolant.
To this end, it readily follows from Cauchy--Schwarz inequality, the fact that $\alpha_{z,T} \leq 1$, and shape regularity that
$$ |v_z  | \lesssim |T|^{-1/2}
\|v_h\|_{L^2(\omega_z)}, \quad z \in \mathcal{N}_h.
$$
For the boundary nodes, using that $\|\chi_z\|_{L^2(F_z)} \lesssim h_{F_z}^{(1-n)/2}$ \cite[Lemma 3.1]{scott1990finite} and Cauchy--Schwarz inequality, we obtain that
$$|v_z | \lesssim h_{F_z}^{(1-n)/2}\|\hat v\|_{L^2(F_z)}, \; z \in \partial \Omega.
$$
Combining the above with the observation that $\|\varphi_z\|_{L^2(T)} \lesssim h_T^{n/2}$ shows \eqref{eq:stability_clement_modified}.
To show \eqref{eq:L2_estimate_Clement}, we first recall the existence of an averaging/Oswald enriching map $\mathsf{E}: V_h \rightarrow V_h \cap H^1(\Omega)$  \cite{doi:10.1137/S0036142902405217, ern2017finite} with the following properties
   \begin{align}\label{eq:oswald}
 \|v - \mathsf{E}_h v \|_{L^2(\Omega)} + h | \mathsf{E}_h v |_{H^1(\Omega)}  \lesssim h \|v\|_{\mathrm{DG}}, \quad v \in V_h.
\end{align}
We then write
\begin{multline}
\|u_h - \mathcal{C}_h(u_h, \hat u_h) \|_{L^2(\Omega)} \leq \| u_h -\mathsf{E}_h u_h \|_{L^2(\Omega)} \\  + \|\mathsf{E}_h u_h  - \mathcal{C}_h(\mathsf{E}_h u_h, \mathsf{E}_h u_h)\|_{L^2(\Omega)} + \|\mathcal{C}_h(\mathsf{E}_h u_h - u_h , \mathsf{E}_h u_h - \hat u_h)\|_{L^2(\Omega)}.
\end{multline}
Since $\mathsf{E}_h u_h \in H^1(\Omega)$,  $\mathcal{C}_h (\mathsf E_h u_h, \mathsf E_h u_h)$ reduces to the interpolant defined in \cite[Lemma 4.2]{aprioriPG} and we can use the approximation property proved therein. With \eqref{eq:oswald}, we obtain \begin{align} \label{eq:estimate_L2_0}
   \|u_h - \mathcal{C}_h(u_h, \hat u_h) \|_{L^2(\Omega)} \lesssim h (\|u_h\|_{\mathrm{DG}} + |\mathsf{E}_h v|_{H^1(\Omega)})
   + \|\mathcal{C}_h(\mathsf{E}_h u_h - u_h , \mathsf{E}_h u_h - \hat u_h)\|_{L^2(\Omega)}.
\end{align}
To bound the last term, denoted by $L$, we use the stability property \eqref{eq:stability_clement_modified}. We have
\begin{align}
L^2 & \lesssim  \|\mathsf{E}_h u_h - u_h\|^2_{L^2(\Omega)} + \sum_{F \in \Gamma_h^\partial } h_{F}\|\mathsf{E}_h u_h - \hat u_h\|^2_{L^2(F)}   \\
& \leq \|\mathsf{E}_h u_h - u_h\|^2_{L^2(\Omega)} +  \sum_{F \in \Gamma_h^\partial } 2h_{F} ( \|\mathsf{E}_h u_h - u_h \|_{L^2(F)}^2 + \| u_h - \hat u_h\|^2_{L^2(F)}) \nonumber \\
& \lesssim
\|\mathsf{E}_h u_h - u_h\|^2_{L^2(\Omega)} + h^2 \sum_{F \in \Gamma_h^\partial }h_{F}^{-1} \| u_h - \hat u_h\|^2_{L^2(F)} , \nonumber
\end{align}
where to obtain the last bound, we used a local trace inequality. Using the above bound, \eqref{eq:oswald},  \eqref{eq:estimate_L2_0}, triangle inequality, and the shape regularity of $\mathcal{T}_h$ shows the result.
\end{proof}

\begin{lemma}[Consistency] \label{lemma:consistency}
Let $u \in H^2(\Omega)\cap H^1_0(\Omega)$ and $\bm q = - \kappa \nabla u$. For any $(\bm r_h, v_h, \hat v_h) \in \bm \Sigma_h^p \times V_h^p \times M_{h,0}^p$, we have that
\begin{align}
\mathcal{A}_\mathsf{D}((\bm q, u, u), (\bm r_h, v_h, \hat v_h))  + \mathcal{A}_\mathsf{C}((u, u), (v_h, \hat v_h)) = \mathcal{A}(\bm q, u, v_h),
\end{align}
where the form $\mathcal{A}$ is given by
\begin{equation} \label{eq:def_A_FOSPG_err}
 \mathcal{A}(\bm q, u, v_h)  := (\nabla \cdot  \bm q, v_h)  +  (\nabla \cdot (\beta u), v_h) .
\end{equation}
\end{lemma}
\begin{proof}
The proof is standard; we skip the details. It suffices to note that $(w, \hat v_h)_{\partial \mathcal{T}_h} = 0$ whenever $w \in H^1_0(\Omega)$ and $(\bm q \cdot \bm n , \hat v_h )_{\partial \mathcal{T}_h} = 0$ whenever $\hat v_h \in M_{h,0}^p$ and $\bm q \in H(\mathrm{div};\Omega)$.
\end{proof}
\begin{theorem}\label{thm:FOSPG_err_rate}
Assume that the true solution $u^*$ to \eqref{eq:global_obstacle} satisfies $u^* \in H^2(\Omega) \cap H^1_0(\Omega)$. Define $\bm q^* = -\kappa \nabla u^*$,  and assume that $\phi \in H^2(\Omega)$ with $\phi \vert_{\partial \Omega} = \mathrm{constant} \leq 0$. In \Cref{alg:FOSPG_1}, let $p = 1$,  $q \in \{0,1\}$ and assume that $\|\psi_h^0\|_{L^{\infty}(\Omega)} \leq c$ where $c$ is independent of $h$.  Then, the weighted averages  $(\overline {\bm q}_h^\ell, \overline{u}_h^\ell , \overline{\hat u}_h^\ell)$ given by
\[
(\overline {\bm q}_h^\ell, \overline{u}_h^\ell , \overline{\hat u}_h^\ell ) = \frac{1}{\sum_{k=1}^\ell \alpha_k} \sum_{k=1}^\ell \alpha_k (\bm q_h^k, u_h^k, \hat u_h^k),
\]
where $(\bm q_h^k, u_h^k, \hat u_h^k)$ are generated from \Cref{alg:FOSPG_1} converge to $u^*$ and the
following error estimate holds for any $\ell \geq 1$
\begin{align}
\vvvert (\overline {\bm q}_h^\ell - \bm q^* , \overline{u}_h^\ell - u^*, \overline{\hat u}_h^\ell - u^*)  \vvvert^2  + \|\overline \lambda_h^\ell - \lambda^*\|_{H^1(\mathcal{T}_h)^*}^2 \lesssim \frac{1} {\sum_{k=1}^\ell \alpha_k}  + h^2.
\end{align}
In addition, if $q =0$, then the bound preserving approximation, given by \[
\langle \nabla \mathcal{R}^*(\psi_h^\ell) \rangle = \frac{1}{\sum_{k=1}^\ell \alpha_k} \sum_{k=1}^\ell \alpha_k \nabla \mathcal{R}^*(\psi_h^k),
\]
satisfies the following error estimate
\begin{align}
\|u^* - \langle \nabla \mathcal{R}^*(\psi_h^\ell) \rangle \|_{L^2(\Omega)}^2 \leq \frac{1}{\sum_{k=1}^\ell \alpha_k}  + h^2. \label{eq:error_bound_preserving}
\end{align}
\end{theorem}
\begin{proof}
Following the proof of \Cref{thm:best_approximation}, we readily obtain that
\begin{align}
 (\nabla \mathcal{R}( o_h^{k}) - \nabla \mathcal{R}(o_h^{k-1}),  o_h^{k} -  u^*)
 & = -\alpha_{k} \mathcal{A}_\mathsf{D}( (\bm q_h^k, u_h^k, \hat u_h^k),  (\bm r_h, u_h^{k} - \Pi_h u^*, \hat v_h))\\
&  - \alpha_k \mathcal{A}_\mathsf{C}( (u_h^k, \hat u_h^k), (u_h^k - \Pi_h u^*, \hat v_h) )  + \alpha_k (f ,  u_h^k - \Pi_h u^*)  \nonumber.
\end{align}
for any $\bm r_h \in \bm \Sigma_{h}^p$ and $\hat v_h \in M_{h,0}^p$.  Select $\hat v_h = \hat u_h^k - \hat \pi_h u^*$ and define $(\bm e^k, e^k , e_\mathsf{t}^k)  = (\bm q_h^k - \bm \Pi_h \bm q^*, u^k_h - \Pi_h u^*, \hat u_h^k - \hat \pi_h u^*)$, where $\bm \Pi_h$, $\Pi_h$  and $\hat \pi_h$ are the $L^2$ projections onto $\bm \Sigma_h^p$, $V_h^p$ and $M_{h,0}^p$, respectively. 
With the three point identity \eqref{eq:three_point_identity} and the fact that $\mathcal{D}(o_h^{k}, o_h^{k-1}) \geq 0 $, we obtain that
\begin{align}
\int_\Omega \mathcal{D}( u^*, o_h^{k}) \mathrm{d}x  + \alpha_k & \mathcal{A}_\mathsf{D}( (\bm e^k, e^k , e_\mathsf{t}^k) ,  (\bm r_h, e^{k} , e_\mathsf{t}^k)) + \alpha_k \mathcal{A}_\mathsf{C}((e^k, e_\mathsf{t}^k),(e^k, e_\mathsf{t}^k))  \\ \nonumber &  \leq \int_{\Omega} \mathcal{D}(u^*, o_h^{k-1}) \mathrm{d}x  -  \alpha_{k} \mathcal{A}_\mathsf{D}( (\bm \Pi_h \bm q^*, \Pi_h u^*, \hat \pi_h u^*),  (\bm r_h, e^{k} ,e_\mathsf{t}^k))  \\  & \nonumber
- \alpha_k\mathcal{A}_\mathsf{C}((\Pi_h u^*, \hat \pi_h u^*),(e^k, e_\mathsf{t}^k)) + \alpha_k (f , e^k)  .
\end{align}
From \cite{egger2010hybrid}, there exists $\bm r^k_h \in \bm \Sigma_h^p$ such that
\begin{align} \label{eq:coercivity_AD}
\mathcal{A}_\mathsf{D}( (\bm e^k, e^k , e_\mathsf{t}^k) ,  (\bm r^k_h, e^{k} , e_\mathsf{t}^k)) \gtrsim \vvvert(\bm e^k, e^k , e_\mathsf{t}^k)  \vvvert^2 , \quad \vvvert (\bm r^k_h, e^k, e_\mathsf{t}^k) \vvvert \lesssim \vvvert(\bm e^k, e^k , e_\mathsf{t}^k) \vvvert.
\end{align}
Note that the hidden constants above are independent of $k$ and $h$.
Along with \eqref{eq:nonnegativity_convection}, this yields
\begin{align}
 \int_\Omega \mathcal{D}( u^*, o_h^{k}) \mathrm{d}x  + \alpha_k \vvvert  (\bm e^k, e^k , e_\mathsf{t}^k) \vvvert^2   &  \lesssim \int_{\Omega} \mathcal{D}(u^*,  o_h^{k-1}) \mathrm{d}x  \\ & -  \alpha_{k} \mathcal{A}_\mathsf{D}( (\bm \Pi_h \bm q^*, \Pi_h u^*, \hat \pi_h u^*),  (\bm r^k_h, e^{k} , e_\mathsf{t}^k)) \nonumber \\ 
 & - \alpha_k\mathcal{A}_\mathsf{C}((\Pi_h u^*, \hat \pi_h u^*),(e^k, e_\mathsf{t}^k)) +  \alpha_k (f ,  e^k) . \nonumber
\end{align}
We now sum the above bound from $k=1$ to $k = \ell$, use that $\mathcal{D}(u^*, o_h^\ell) \geq 0$ for any $\ell$, and divide by $\sum_{k=1}^\ell \alpha_k$.  This yields
\begin{align} \label{eq:error_rate_FOSPG_0}
  \frac{ \sum_{k=1}^\ell \alpha_k \vvvert (\bm e^k, e^k, e_\mathsf{t}^k) \vvvert^2 }{\sum_{k=1}^\ell \alpha_k}   & \leq  \frac{\int_{\Omega} \mathcal{D}(u^*, o_h^{0}) \mathrm{d}x}{\sum_{k=1}^\ell \alpha_k} -  \mathcal{A}_\mathsf{D}( (\bm \Pi_h \bm q^*, \Pi_h u^*, \hat \pi_h u^*),  (\overline{\bm r}^\ell_h, \overline e^{\ell} , \overline{e}_{\mathsf{t}}^\ell)) \\ &-  \mathcal{A}_\mathsf{C}((\Pi_h u^*, \hat \pi_h u^*),(\overline e^\ell, \overline e_\mathsf{t}^\ell)) + (f ,  \overline e^\ell) . \nonumber
\end{align}
Denote the sum of the second and third terms by $W$. Using \Cref{lemma:consistency}, we obtain
\begin{align}
W & =  - \mathcal{A}_\mathsf{D}( (\bm \Pi_h \bm q^* - \bm q^*, \Pi_h u^* - u^*, \hat \pi_h u^* - u^*), (\overline{\bm r}^\ell_h, \overline e^{\ell} , \overline{e}_{\mathsf{t}}^\ell)) \\ & \nonumber - \mathcal{A}_\mathsf{C}((\Pi_h u^* - u^*, \hat \pi_h u^* - u^*),(\overline e^\ell, \overline e_\mathsf{t}^\ell)) - \mathcal{A}(\bm q^*, u^*, \overline e^\ell) : = W_1 + W_2 +W_3.
 \end{align}
We proceed to bound  $W_1, W_2$, and $W_3$. Bounding $W_1$ and $W_2$ follows standard arguments, and we skip the details for brevity. Following  \cite{egger2010hybrid} and using the approximation properties of the $L^2$ projection, we bound
\begin{align} \label{eq:error_rate_FOSPG_1}
W_1 &  \lesssim \left( \vvvert (\bm \Pi_h \bm q^* - \bm q^*, \Pi_h u^* - u^*, \hat \pi_h u^* - u^*) \vvvert^2 + h\|(\bm \Pi_h \bm q^* - \bm q^*)\cdot \bm n \|_{\Gamma_h}^2 \right)^{1/2} \vvvert (\overline{\bm r}^\ell_h, \overline e^{\ell} , \overline{e}_{\mathsf{t}}^\ell) \vvvert \\
& \lesssim h \|u^* \|_{H^{2}(\Omega)} \vvvert (\overline{\bm r}^\ell_h, \overline e^{\ell} , \overline{e}_{\mathsf{t}}^\ell) \vvvert.  \nonumber
\end{align}
Similarly, we have that
\begin{align}\label{eq:error_rate_FOSPG_2}
W_2 &  \lesssim (\|\Pi_h u^* - u^*\|_{L^2(\Omega)}^2 + h\|\Pi_h u^* - u^*\|^2_{L^2(\Gamma_h)} + h \|\hat \pi_h u^* - u^*\|_{L^2(\Gamma_h)}^2)^{1/2} \vvvert (\overline{\bm r}^\ell_h, \overline e^{\ell} , \overline{e}_{\mathsf{t}}^\ell) \vvvert  \\
\nonumber
& \lesssim h^2  \|u^* \|_{H^{2}(\Omega)} \vvvert (\overline{\bm r}^\ell_h, \overline e^{\ell} , \overline{e}_{\mathsf{t}}^\ell) \vvvert.
\end{align}
To handle $W_3$, we define the following reconstruction operator using the definition of $\mathcal{C}_h$ of \Cref{lemma:modified_Clement}:
\[  \mathcal{E}_h(u_h ,\hat u_h) := \mathcal{C}_h(u_h - \phi, \hat u_h - \phi) + \phi.
\]
Observe that $\mathcal{C}_h(\overline{u}^\ell_h - \phi,  \overline{\hat u}_h^\ell - \phi)(z)  \geq 0$  for $z \in \mathcal{N}_h$ since $\int_{K} (u_h^\ell - \phi)\mathrm{d} x \geq 0$ thanks to \eqref{eq:nonlinear_subproblem_1}. Further, since  $\hat u_h^\ell = 0$ and $\phi = \mathrm{constant} \leq 0$ on $\partial \Omega$,  we have that $\mathcal{C}_h(\overline{u}^\ell_h - \phi,  \overline{\hat u}_h^\ell - \phi)(z) = - \phi \geq 0$ for $z \in \partial \Omega$.  As such, $\mathcal{E}_h(\overline u_h^\ell ,\overline{\hat u}_h^\ell) \geq \phi$  and $ \mathcal{E}_h(\overline u_h^\ell ,\overline{\hat u}_h^\ell) \in H^1_0(\Omega)$ for all $\ell$.   Thus, $\mathcal{E}_h(\overline u_h^\ell ,\overline{\hat u}_h^\ell) \in K$.  Proceeding, we drop the dependence on $\overline{\hat u}_h^\ell$ and  denote by $\mathcal{E}_h \overline u_h^\ell:= \mathcal{E}_h(\overline u_h^\ell ,\overline{\hat u}_h^\ell) \in K$ to simplify the notation.  We write
\begin{align} \label{eq:error_rate_FOSPG_3}
W_3 &  = - \mathcal{A}(\bm q^*, u^*, \overline u_h^\ell - u^*) - \mathcal{A}(\bm q^*, u^*,  u^* - \Pi_h u^*)  \\  \nonumber
& = - \mathcal{A}(\bm q^*, u^*, \overline u_h^\ell -\mathcal{E}_h\overline u_h^\ell )  -  \mathcal{A}(\bm q^*, u^*, \mathcal{E}_h\overline u_h^\ell - u^*) -  \mathcal{A}(\bm q^*, u^*,  u^* - \Pi_h u^*)  \\ \nonumber
& \leq - \mathcal{A}(\bm q^*, u^*, \overline u_h^\ell -\mathcal{E}_h\overline u_h^\ell )  - (f, \mathcal{E}_h\overline u_h^\ell - u^*) -  \mathcal{A}(\bm q^*, u^*,  u^* - \Pi_h u^*),  \nonumber
\end{align}
where we used \eqref{eq:general_VI} for the last bound.  Collecting \eqref{eq:error_rate_FOSPG_1},\eqref{eq:error_rate_FOSPG_2}, \eqref{eq:error_rate_FOSPG_3} in \eqref{eq:error_rate_FOSPG_0} yields
\begin{align}\label{eq:error_rate_FOSPG_4}
 \frac{ \sum_{k=1}^\ell \alpha_k \vvvert (\bm e^k, e^k, e_\mathsf{t}^k) \vvvert^2 }{\sum_{k=1}^\ell \alpha_k}   & \lesssim  \frac{\int_{\Omega} \mathcal{D}(u^*, o_h^{0}) \mathrm{d}x }{\sum_{k=1}^\ell \alpha_k}  + h  \|u^* \|_{H^{2}(\Omega)} \vvvert (\overline{\bm r}^\ell_h, \overline e^{\ell} , \overline{e}_{\mathsf{t}}^\ell) \vvvert  \\
 \nonumber   &  +\mathcal{A}(\bm q^*, u^*, \mathcal{E}_h\overline u_h^\ell - \overline u_h^\ell )  + \mathcal{A}(\bm q^*, u^*,   \Pi_h u^* - u^*)  \\
 \nonumber & + (f,  u^* - \Pi_h u^*) + (f, \overline u_h^\ell - \mathcal{E}_h\overline u_h^\ell ).
\end{align}
Considering \eqref{eq:def_A_FOSPG_err} and the Cauchy-Schwarz inequality, the sum of the last four terms above, denoted by $W_4$, is bounded by
\begin{align} \label{eq:error_rate_FOSPG_5}
W_4 \lesssim (\|u^*\|_{H^2(\Omega)} + \|f\|_{L^2(\Omega)}) (\| \mathcal{E}_h\overline u_h^\ell - \overline u_h^\ell \|_{L^2(\Omega)} + \|\Pi_h u^* - u^*\|_{L^2(\Omega)}).
\end{align}
To handle $\|\mathcal{E}_h\overline u_h^\ell - \overline u_h^\ell \|_{L^2(\Omega)}$, we use \Cref{lemma:modified_Clement} and the fact that $\mathcal{C}_h (w, w)$ for $w \in H^1(\Omega)$ is the same map analyzed in \cite[Lemma 4.2]{aprioriPG}. We bound
\begin{multline}
\label{eq:error_rate_FOSPG_6}
\| \mathcal{E}_h\overline u_h^\ell - \overline u_h^\ell \|_{L^2(\Omega)}  \leq \|\mathcal{C}_h(\overline{e}^\ell, \overline{e}_{\mathsf{t}}^\ell ) - \overline{e}^\ell\|_{L^2(\Omega)} \\  + \|\mathcal{C}_h(\Pi_h u^*, \hat \pi_h u^*) - \Pi_h u^*\|_{L^2(\Omega)} + \|\mathcal{C}_h(\phi, \phi) - \phi \|_{L^2(\Omega)}:= W_{4,1} +W_{4,2}+ W_{4,3}.
\end{multline}
For the first term, we apply \Cref{lemma:modified_Clement}
\begin{align} \label{eq:first_term_W4}
W_{4,1}    \lesssim h \left(\|\overline{e}^\ell\|^2_{\mathrm{DG}} +   \sum_{F \in \Gamma_h^\partial }h_{F}^{-1} \| \overline e^\ell - \overline{e}_{\mathsf{t}}^\ell \|^2_{L^2(F)} \right)^{1/2}  \lesssim h \vvvert (\overline{e}^\ell, \overline e^{\ell} , \overline{e}_{\mathsf{t}}^\ell) \vvvert,
\end{align}
where we used that $\overline{e}_{\mathsf{t}}^\ell \in M_{h,0}^p$ and \eqref{eq:dg_to_triple}.
For the second term, we write
\[ W_{4,2} \leq \|\mathcal{C}_h(\Pi_h u^* - I_h u^*, \hat \pi_h u^* - I_h u^*)\|_{L^2(\Omega)}  + \| \mathcal{C}_h (I_h u^* , I_h u^*) - \Pi_h u^* \|_{L^2(\Omega)}, \]
where $I_h u^* \in V_h \cap H^1_0(\Omega)$ is the Lagrange interpolant of $u^*$. We note that $\mathcal{C}_h(I_h u^* , I_h u^*)$ reduces to $\mathcal{C}_h(I_h u^*)$ where $\mathcal{C}_h$ is defined in  \cite[Lemma 4.2]{aprioriPG}.  Using the stability bound \eqref{eq:stability_clement_modified} of \Cref{lemma:modified_Clement} and the triangle inequality, we obtain
\begin{multline*}
W_{4,2} \lesssim \left( \|\Pi_h u^* - I_h u^*\|^2_{L^2(\Omega)} + \sum_{F \in \Gamma_h \cap \partial\Omega} h_F\|\Pi_h u^* - I_h u^*\|^2_{L^2(F)}\right)^{1/2} \\
+ \|\mathcal{C}_h(I_h u^*  -  u^*) \|_{L^2(\Omega)} + \|\mathcal{C}_hu^* - \Pi_h u^*\|_{L^2(\Omega)}.
\end{multline*}
The optimality and stability of $I_h$, $\mathcal{C}_h $ \cite[Lemma 4.2]{aprioriPG}, and $\Pi_h$, then give that
\begin{equation} \label{eq:second_term_W4}
W_{4,2} \lesssim h^2 \|u^*\|_{H^2(\Omega)}.
\end{equation}
For $W_{4,3}$, we use the optimality of  $\mathcal{C}_h $ to bound
\begin{equation} \label{eq:third_term_W4}
W_{4,3} \lesssim h^2 \|\phi\|_{H^2(\Omega)}.
\end{equation}
Collecting \eqref{eq:first_term_W4}, \eqref{eq:second_term_W4}, and \eqref{eq:third_term_W4} in \eqref{eq:error_rate_FOSPG_6} and using the resulting inequality and the  optimality of $\Pi_h$ in \eqref{eq:error_rate_FOSPG_5} and \eqref{eq:error_rate_FOSPG_4} , we arrive at
\begin{align}\label{eq:error_rate_FOSPG_7}
  \frac{ \sum_{k=1}^\ell \alpha_k \vvvert (\bm e^k, e^k, e_\mathsf{t}^k) \vvvert^2 }{\sum_{k=1}^\ell \alpha_k}   & \lesssim  \frac{\int_{\Omega} \mathcal{D}(u^*, o_h^{0}) \mathrm d x }{\sum_{k=1}^\ell \alpha_k}  + h ( \vvvert (\overline{\bm r}^\ell_h, \overline e^{\ell} , \overline{e}_{\mathsf{t}}^\ell) \vvvert  +\vvvert (\overline{e}^\ell, \overline e^{\ell} , \overline{e}_{\mathsf{t}}^\ell) \vvvert)  + h^2.
  \end{align}
Finally, we note that from \eqref{eq:coercivity_AD} and Cauchy--Schwarz inequality for sums,
\[
 \vvvert (\overline{\bm r}^\ell_h, \overline e^{\ell} , \overline{e}_{\mathsf{t}}^\ell) \vvvert^2 \leq \left( \frac{c}{\sum_{k=1}^\ell
 \alpha_k} \sum_{k=1}^\ell \alpha_k  \vvvert (\bm e^k, e^k, e_\mathsf{t}^k) \vvvert \right)^2   \leq \frac{c^2}{\sum_{k=1}^\ell
 \alpha_k} \sum_{k=1}^\ell \alpha_k  \vvvert (\bm e^k, e^k, e_\mathsf{t}^k) \vvvert^2,
\]
where $c$ is a constant independent of $h$, $k$, and $\{\alpha_k\}$.
The above and Young's inequality allows us to obtain that
\begin{align}\label{eq:error_rate_FOSPG_8}
 \frac12  \frac{ \sum_{k=1}^\ell \alpha_k \vvvert (\bm e^k, e^k, e_\mathsf{t}^k) \vvvert^2 }{\sum_{k=1}^\ell \alpha_k}   & \lesssim   \frac{\int_{\Omega} \mathcal{D}(u^*, o_h^{0}) \mathrm d x }{\sum_{k=1}^\ell \alpha_k} +   h \vvvert (\overline{e}^\ell, \overline e^{\ell} , \overline{e}_{\mathsf{t}}^\ell) \vvvert  + h^2.
  \end{align}
We use Jensen, Young's, triangle inequalities, and the observation that $\mathcal{D}(u^*, o_h^0) \leq C$ for some constant $C$ independent of $h$. We conclude the bound on $\vvvert (\overline {\bm q}_h^\ell - \bm q^* , \overline{u}_h^\ell - u^*, \overline{\hat u}_h^\ell - u^*)  \vvvert^2$.  To obtain the bound on
$\|\overline{\lambda}_h^\ell -\lambda^*\|_{H^1(\mathcal{T}_h)^*}$, we first observe that with \Cref{lemma:consistency} and the definition of $\lambda^*$,
\begin{align} \label{eq:err_eq_lambda}
(\overline{\lambda}_h^\ell - \lambda^*,& v_h)  = \mathcal{A}_\mathsf{D}((\overline {\bm q}^\ell_h - \bm q^*, \overline u^\ell_h - u^*, \overline{\hat u}^\ell_h - u^*), (\bm r_h, v_h, \hat v_h))  \\ & + \mathcal{A}_\mathsf{C}((\overline u^\ell_h - u^*, \overline{\hat u}^\ell_h -u^*), (v_h, \hat v_h)), ~\fa (\bm r_h, v_h , \hat v_h)\in  \bm \Sigma_h^p  \times V_h^p \times  M_{h,0}^p. \nonumber
\end{align}
We now use \eqref{eq:inf_sup_DG}, \eqref{eq:nonlinear_subproblem_0} with test function $(\bm0, w, \hat w)$ where $\hat w$ is given by \eqref{eq:triple_to_dg}, the observation that $ \vvvert (\bm 0,  w, \hat w) \vvvert \lesssim \|w\|_{\mathrm{DG}}$, the continuity properties of $\mathcal{A}_{\mathsf{D}}, \mathcal{A}_{\mathsf{C}}$ similar to \eqref{eq:error_rate_FOSPG_1} and \eqref{eq:error_rate_FOSPG_2}, and the proven bound on $\vvvert (\overline {\bm q}_h^\ell - \bm q^* , \overline{u}_h^\ell - u^*, \overline{\hat u}_h^\ell - u^*)  \vvvert^2$. Details are skipped for brevity. To show \eqref{eq:error_bound_preserving}, we note that $\langle \nabla \mathcal{R}^*(\psi_h^\ell) \rangle \in V_h^0$ and thus $\Pi^0_h \overline u_h^\ell =  \langle\nabla \mathcal{R}^*(\psi_h^\ell) \rangle$, the $L^2$ projection onto $V_h^0$. This implies that
\begin{align*}
\|u^*-   \langle\nabla \mathcal{R}^*(\psi_h^\ell) \rangle\|_{L^2(\Omega)} \leq \|u^* - \Pi^0_h u^*\|_{L^2(\Omega)} + \|\Pi^0_h (u^* - \overline u_h^\ell)\|_{L^2(\Omega)}.
\end{align*}
The stability and approximation properties of $\Pi^0_h$ along with the Poincar\`e inequality \eqref{eq:Poincare}, \eqref{eq:dg_to_triple}, and the proven bound on $\vvvert (\overline {\bm q}_h^\ell - \bm q^* , \overline{u}_h^\ell - u^*, \overline{\hat u}_h^\ell - u^*)  \vvvert$ yield the result.
\end{proof}

\subsection{FOSPG for  the semi--permeable membrane and scalar Signorini problems} \label{sec:semi_permeable}
Here, we again consider the nonsymmetric VI \eqref{eq:general_VI} with $\mathcal{L}: H^1_0(\Omega) \rightarrow H^{-1}(\Omega)$ but with constraints given on parts of the boundary $\partial \Omega$. Namely, we set
\begin{subequations}
\label{eq:boundary_obstacle}
\begin{align}
\mathcal{L} u  &= - \nabla \cdot (  \kappa \nabla u) + \beta \cdot \nabla  u ,  \\
 K  &= \{ v \in H^1_{\mathsf{D}}(\Omega) \mid  v \geq \phi \text{ a.e. on } \Gamma_{\mathsf{S}} \},
 \end{align}
\end{subequations}
where $\Gamma_{\mathsf{S}}$ and $\Gamma_{\mathsf{D}}$ form a non-overlapping partition of $\partial \Omega$,  $H^1_{\mathsf{D}}(\Omega)  = \{v \in H^1(\Omega) \mid v = 0 \text{ on } \Gamma_{\mathsf{D}} \}$, and  $\Gamma_{\mathsf{D}}$ has non-trivial measure.  This corresponds to \Cref{example:semi_permeable} where we set $\Gamma_\mathsf{N} = \emptyset$ for simplicity. Note that if $\beta = 0$, then \eqref{eq:boundary_obstacle} also models the scalar Signorini problem \cite{hild2012improved}.

Since the constraint is on a part $\Gamma_{\mathsf{S}}$ of the boundary, we set $W^q_h$ as the broken polynomial space over $\Gamma_\mathsf{S}$:
\begin{equation}
    W^q_h =  \;\bigg\{
  {\mu}_h\in L^2(\Gamma_\mathsf S) \mid
  {\mu}{}_{|_{F}}
\in \mathbb{P}_q(F),  \; \; \forall F \in \Gamma_h^\mathsf{S}
  \bigg\},  \label{eq:Wh_signorini}
\end{equation}
where $\Gamma_h^\mathsf{S}$ is the set of facets intersecting $\Gamma_{\mathsf{S}}$. Here, for simplicity, we assume that $\Gamma_h^\mathsf{S} \cap \Gamma_h^\partial = \emptyset$. 
The FOSPG algorithm for this class of problems is given \Cref{alg:FOSPG_boundary}. We reserve its analysis for future work. Here, we implement this algorithm in \Cref{num:semiperm} for the semi-permeable conditions presented \Cref{example:semi_permeable}.
\begin{algorithm}[htb]
  \caption{The Hybridized First Order System Proximal Galerkin Method for \eqref{eq:boundary_obstacle}}
  \label{alg:FOSPG_boundary}
  \begin{algorithmic}[1]\label{alg:main_alg_disc}
    \State \textbf{input:}  A discrete latent solution guess $\hat \psi_h^0 \in W^q_h$ with $W^q_h$ $(q \leq p)$ given in \eqref{eq:Wh_signorini} and a sequence of positive step sizes $\{\alpha_k\}$.
    \State Initialize \(k = 1\).
    \State \textbf{repeat}
    \State \quad
    Find $(\bm q^{k}_h, u_h^{k} , \hat u^{k}_h) \in \bm \Sigma_h^p  \times V_h^p \times  M_{h,g}^p$ and $\hat \psi_h^{k} \in W^q_h$ such that
    \begin{subequations}
      \label{eq:nonlinear_subproblem_boundary}
      \begin{alignat}{2}
        \mathcal{A}_\mathsf{D}((\bm q^k_h, u^k_h, \hat u^k_h), (\bm r_h, v_h, \hat v_h))  + \mathcal{A}_\mathsf{C}((u^k_h, \hat u^k_h), (v_h, \hat v_h)) +   \frac{1}{\alpha_k}(\hat \psi_h^{k} -  \hat \psi_h^{k-1} , \hat v_h)_{\Gamma_\mathsf{S}} & =  (f,v_h) ,  \label{eq:nonlinear_subproblem_boundary_0} \\
        (\hat u_h^{k}, \hat w_h)_{\Gamma_{\mathsf{S}}}  - (\nabla \mathcal{R}^*(\hat \psi_h^{k}), \hat w_h)_{\Gamma_{\mathsf{S}}}                                               & = 0,  \label{eq:nonlinear_subproblem_boundary_1}
      \end{alignat}
    \end{subequations}
    for all $ (\bm r_h , v_h, \hat v_h) \in  \Sigma_h^p \times V_h^p \times M_{h,0}^p $ and  $\hat w_h \in W^q_h$.
    \State \quad Assign \(k \gets k + 1\).
    \State \textbf{until} a convergence test is satisfied.
  \end{algorithmic}
\end{algorithm}

\begin{remark}[Solving for $(\hat u_h^k, \hat \psi_h^k)$]
Using static condensation, one can reformulate \eqref{eq:nonlinear_subproblem} in terms of the facet multipliers $(\hat u_h^k, \hat \psi_h^k) $ only. As such, the internal degrees of freedom representing $(\bm q_h^k, u_h^k)$ need not be recovered in every proximal step, but only after convergence. This improves the computational efficiency of \Cref{alg:FOSPG_boundary}.
\end{remark}

\section{Numerical Experiments}
\label{sec:numerical_experiments}
We provide a series of numerical examples that report the convergence behavior of our methods and that demonstrate their performance on Examples \ref{example:obstacle_type}--\ref{example:dam}.
\subsection{Convergence Rates}
\label{sec:conv}
In this first  example, we compute error rates for \Cref{alg:FOSPG_1} applied to \Cref{example:obstacle_type} with $\kappa =\operatorname{Id}$,  $\beta = (1,1)$, $c = 0$.
We modify the example from \cite[Section 4.8.4]{keith2023proximal}.
We set $\Omega=(-1,1)^2$, and use structured triangular meshes $\{\mathcal{T}_h\}$ with mesh size $h$.
\begin{equation*}
\begin{aligned}
   \phi = \begin{cases}
        \sqrt{ 1/4 - r^2} & \mathrm{if } \quad r \leq 9/20, \\
        \varphi (r) & \mathrm{otherwise},
    \end{cases} \qquad \text{where } r = \sqrt{x^2 + y^2}\,.
\end{aligned}
\end{equation*}
In the above, $\varphi(r)$ is the unique $C^1$ linear extension of $r \mapsto \sqrt{1/4 - r^2}$ for $r > 9/20$. The exact solution is given by
\begin{equation}
        u = \begin{cases}
    Q  \ln \sqrt{x^2+y^2} &\mathrm{if } \quad \sqrt{x^2 +y^2} >  a, \\
    \phi & \mathrm{otherwise},
    \end{cases}
    \label{eq:circular_sol}
\end{equation}
where $a= \exp(W_{-1}(-1/(2e^2))/2 + 1)\approx 0.34898$ where $W_{-1}(\cdot)$ is the $-1$-branch of the Lambert W-function, and $Q = \sqrt{1/4-a^2}/\ln a$.
We change the forcing $f = \beta \cdot \nabla u$.
One can only expect at most an order of $5/2$ (resp.\ $3/2$) for $\|u-u_h\|_{L^2(\Omega)}$ (resp.\ $\|\bm q - \bm q_h\|_{L^2(\Omega)} $) since $u \in H^{5/2-\epsilon}(\Omega)$.
We set $\psi_h^0 =0$, $\alpha_k = 2^{k-1}$, and $\nabla \mathcal{R}^*(\psi) = \exp(\psi) +\phi$ corresponding to the admissible set $K = \{u \in H^1_{\mathsf{D}}(\Omega) \vert \, u \geq \phi\}.$

In this example, we report results only for the FOSPG method with $h_0=1/16$, $p=1$, and $q=0$. We note that the conforming method exhibits qualitatively similar convergence behavior, but we omit these results for brevity.
In the following examples, we choose $\alpha_k=2^{k-1}$ and $\psi_h^0=0$.
\Cref{table:error_rates} presents the computed error rates with respect to $h$, while \Cref{fig:Cauchy_err} illustrates the optimization error history for both the iterates $u_h^k$ and the weighted averages $\overline{u}_h^k$.
In \Cref{fig:Cauchy_err}, we isolate the optimization error by measuring the distance between the current iterates (or averages) and the fully converged discrete solution, denoted by $u_h$.

Notably, we observe that the discretization error begins to dominate the total error after approximately 6--8 iterations.
We observe an optimization error of the averaged iterates $\{\overline{u}_h^k\}$ decay of $\mathcal{O}(2^{-k})$, which is faster than the theoretically predicted convergence rate of $\mathcal{O}(2^{-k/2})$ (cf. \Cref{thm:FOSPG_err_rate} and \cite[Corollary A.12]{keith2023proximal}). Furthermore, the individual iterates $\{u_h^k\}$ converge even faster than the averaged iterates.
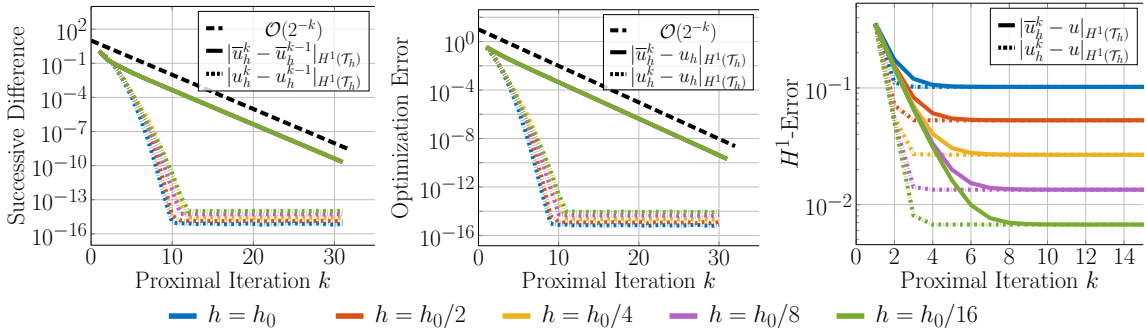
\begin{figure}[!t]
\centering
  \resizebox{0.32\textwidth}{!}{
\begin{tikzpicture}[font=\Huge]
	\begin{axis}[
			axis y line* = left,
			width        = 12cm,
			view         = {45}{65},
			xlabel={Proximal Iteration $k$},
			ylabel={Successive Difference},
			label style={font=\Huge},
			xmajorgrids,
			ymajorgrids,
            yminorgrids,
            xtick = {0, 10, 20, 30, 40},
            xmin = 0, xmax = 35,
            ymin = 1e-17, ymax = 1e04,
									            ytick = {1e-16, 1e-13, 1e-10, 1e-07, 1e-04, 1e-01, 1e02},
												ymode = log,
			            legend style={at={(0.98,0.98)}, inner sep=6pt, anchor=north east, font=\huge, fill=white, fill opacity=0.8, draw opacity=1, text opacity=1},
		]

		\addplot[color1, dash dot, line width=1.5mm ] table [x expr=\coordindex+1, y index={8}, col sep=comma, forget plot] {data/opt_result_0.csv};
		\addplot[color2, dash dot, line width=1.5mm ] table [x expr=\coordindex+1, y index={8}, col sep=comma, forget plot] {data/opt_result_1.csv};
		\addplot[color3, dash dot, line width=1.5mm ] table [x expr=\coordindex+1, y index={8}, col sep=comma, forget plot] {data/opt_result_2.csv};
		\addplot[color4, dash dot, line width=1.5mm ] table [x expr=\coordindex+1, y index={8}, col sep=comma, forget plot] {data/opt_result_3.csv};
		\addplot[color5, dash dot, line width=1.5mm ] table [x expr=\coordindex+1, y index={8}, col sep=comma, forget plot] {data/opt_result_4.csv};

		\addplot[color1, line width=1.5mm ] table [x expr=\coordindex+1, y index={9}, col sep=comma, forget plot] {data/opt_result_0.csv};
		\addplot[color2, line width=1.5mm ] table [x expr=\coordindex+1, y index={9}, col sep=comma, forget plot] {data/opt_result_1.csv};
		\addplot[color3, line width=1.5mm ] table [x expr=\coordindex+1, y index={9}, col sep=comma, forget plot] {data/opt_result_2.csv};
		\addplot[color4, line width=1.5mm ] table [x expr=\coordindex+1, y index={9}, col sep=comma, forget plot] {data/opt_result_3.csv};
		\addplot[color5, line width=1.5mm ] table [x expr=\coordindex+1, y index={9}, col sep=comma, forget plot] {data/opt_result_4.csv};

                \addplot [domain=0:32, samples=2, color=black, dash pattern=on 3mm off 1.5mm, line width=1.5mm, forget plot] {10*2^(-x)};

		                \addlegendimage{line width=1.5mm, color=black, dash pattern=on 3mm off 1.5mm}
        \addlegendentry{$\mathcal{O}(2^{-k})$}
        \addlegendimage{line width=1.5mm, color=black}
        \addlegendentry{$|\overline{u}_h^k-\overline{u}_h^{k-1}|_{H^1(\mathcal{T}_h)}$}
        \addlegendimage{line width=1.5mm, color=black, dash dot}
        \addlegendentry{$|u_h^k-u_h^{k-1}|_{H^1(\mathcal{T}_h)}$}
	\end{axis}
\end{tikzpicture}}
\resizebox{0.32\textwidth}{!}{
\begin{tikzpicture}[font=\Huge]
    \begin{axis}[
            axis y line* = left,
            width = 12cm,
            view = {45}{65},
            xlabel={Proximal Iteration $k$},
            ylabel={Optimization Error},
            label style={font=\Huge},
            xmajorgrids, ymajorgrids,
            xmin = 0, xmax = 35,
            xtick = {0, 10, 20, 30, 40},
            ytick = {10^-18,10^-16,10^-14,10^-12,10^-10,10^-8,10^-6,10^-4,10^-2,10^0,10^2,10^4},
            yticklabels = {,$10^{-16}$,,$10^{-12}$,,$10^{-8}$,,$10^{-4}$,,$10^0$,,$10^4$},
            ymode = log,
            legend style={at={(0.98,0.98)}, inner sep=5pt, anchor=north east, font=\huge, fill=white, fill opacity=0.8, draw opacity=1, text opacity=1},
        ]

                \addplot[color1, dash dot, line width=1.5mm, forget plot ] table [x expr=\coordindex+1, y index={6}, col sep=comma] {data/opt_result_0.csv};
        \addplot[color2, dash dot, line width=1.5mm, forget plot ] table [x expr=\coordindex+1, y index={6}, col sep=comma] {data/opt_result_1.csv};
        \addplot[color3, dash dot, line width=1.5mm, forget plot ] table [x expr=\coordindex+1, y index={6}, col sep=comma] {data/opt_result_2.csv};
        \addplot[color4, dash dot, line width=1.5mm, forget plot ] table [x expr=\coordindex+1, y index={6}, col sep=comma] {data/opt_result_3.csv};
        \addplot[color5, dash dot, line width=1.5mm, forget plot ] table [x expr=\coordindex+1, y index={6}, col sep=comma] {data/opt_result_4.csv};

                \addplot[color1, line width=1.5mm, forget plot ] table [x expr=\coordindex+1, y index={7}, col sep=comma] {data/opt_result_0.csv};
        \addplot[color2, line width=1.5mm, forget plot ] table [x expr=\coordindex+1, y index={7}, col sep=comma] {data/opt_result_1.csv};
        \addplot[color3, line width=1.5mm, forget plot ] table [x expr=\coordindex+1, y index={7}, col sep=comma] {data/opt_result_2.csv};
        \addplot[color4, line width=1.5mm, forget plot ] table [x expr=\coordindex+1, y index={7}, col sep=comma] {data/opt_result_3.csv};
        \addplot[color5, line width=1.5mm, forget plot ] table [x expr=\coordindex+1, y index={7}, col sep=comma] {data/opt_result_4.csv};

                \addplot [domain=0:32, samples=2, color=black, dash pattern=on 3mm off 1.5mm, line width=1.5mm, forget plot] {10*2^(-x)};

                \addlegendimage{line width=1.5mm, color=black, dash pattern=on 3mm off 1.5mm}
        \addlegendentry{$\mathcal{O}(2^{-k})$}
        \addlegendimage{line width=1.5mm, color=black}
        \addlegendentry{$|\overline{u}_h^k-u_h|_{H^1(\mathcal{T}_h)}$}
        \addlegendimage{line width=1.5mm, color=black, dash dot}
        \addlegendentry{$|u_h^k-u_h|_{H^1(\mathcal{T}_h)}$}
    \end{axis}
\end{tikzpicture}}
\resizebox{0.32\textwidth}{!}{
\begin{tikzpicture}[font=\Huge]
    \begin{axis}[
            axis y line* = left,
            width = 12cm,
            view = {45}{65},
            xlabel={Proximal Iteration $k$},
            ylabel={$H^1$-Error},
            label style={font=\Huge},
            xmajorgrids, ymajorgrids,
            xmin = 0, xmax = 15,
            xtick = {0, 2, 4, 6, 8, 10, 12, 14},
                                    ymode = log,
            legend style={at={(0.98,0.98)}, anchor=north east, inner sep=6pt, font=\huge, fill=white, fill opacity=0.8, draw opacity=1, text opacity=1},
        ]

                \addplot[color1, dash dot, line width=1.5mm, forget plot ] table [x expr=\coordindex+1, y index={3}, col sep=comma] {data/opt_result_0.csv};
        \addplot[color2, dash dot, line width=1.5mm, forget plot ] table [x expr=\coordindex+1, y index={3}, col sep=comma] {data/opt_result_1.csv};
        \addplot[color3, dash dot, line width=1.5mm, forget plot ] table [x expr=\coordindex+1, y index={3}, col sep=comma] {data/opt_result_2.csv};
        \addplot[color4, dash dot, line width=1.5mm, forget plot ] table [x expr=\coordindex+1, y index={3}, col sep=comma] {data/opt_result_3.csv};
        \addplot[color5, dash dot, line width=1.5mm, forget plot ] table [x expr=\coordindex+1, y index={3}, col sep=comma] {data/opt_result_4.csv};

                \addplot[color1, line width=1.5mm, forget plot ] table [x expr=\coordindex+1, y index={4}, col sep=comma] {data/opt_result_0.csv};
        \addplot[color2, line width=1.5mm, forget plot ] table [x expr=\coordindex+1, y index={4}, col sep=comma] {data/opt_result_1.csv};
        \addplot[color3, line width=1.5mm, forget plot ] table [x expr=\coordindex+1, y index={4}, col sep=comma] {data/opt_result_2.csv};
        \addplot[color4, line width=1.5mm, forget plot ] table [x expr=\coordindex+1, y index={4}, col sep=comma] {data/opt_result_3.csv};
        \addplot[color5, line width=1.5mm, forget plot ] table [x expr=\coordindex+1, y index={4}, col sep=comma] {data/opt_result_4.csv};

                                \addlegendimage{line width=1.5mm, color=black}
        \addlegendentry{$|\overline{u}_h^k-u|_{H^1(\mathcal{T}_h)}$}
        \addlegendimage{line width=1.5mm, color=black, dash dot}
        \addlegendentry{$|u_h^k-u|_{H^1(\mathcal{T}_h)}$}
    \end{axis}
\end{tikzpicture}}
\begin{center}
\resizebox{0.7\textwidth}{!}{
\begin{tikzpicture}
        \node (l1) at (0,0) [label=right:{\Huge $h=h_0$}] {};
    \draw[color1, line width=2mm] (l1.west) -- ++(-1,0);

    \node (l2) at (6,0) [label=right:{\Huge $h=h_0/2$}] {};
    \draw[color2, line width=2mm] (l2.west) -- ++(-1,0);

    \node (l3) at (12,0) [label=right:{\Huge $h=h_0/4$}] {};
    \draw[color3, line width=2mm] (l3.west) -- ++(-1,0);

    \node (l4) at (18,0) [label=right:{\Huge $h=h_0/8$}] {};
    \draw[color4, line width=2mm] (l4.west) -- ++(-1,0);

    \node (l5) at (24,0) [label=right:{\Huge $h=h_0/16$}] {};
    \draw[color5, line width=2mm] (l5.west) -- ++(-1,0);
\end{tikzpicture}}
\end{center}

\caption{
Convergence history of the proximal iterates $u_h^k$ and averages $\overline{u}_h^k$ for the circular obstacle problem (\eqref{eq:circular_sol} in \Cref{sec:conv}) with step size $\alpha_k=2^{k-1}$ and initial mesh size $h_0=1/16$. Left: The Cauchy error, measured by the successive differences of the iterates. Center: The optimization error relative to the converged discrete solution $u_h$. Right: The total error $|u - u_h^k|_{H^1(\mathcal{T}_h)}$.
\label{fig:Cauchy_err}
}
\end{figure}\begin{table}[!t]
    \begin{tabular}{|c||c|c|c|c|c|c|c|c|}\hline
    &$h/h_0$&$k$&\multicolumn{2}{c}{$\|u_h^k-u\|_{L^2(\Omega)}$}&\multicolumn{2}{|c|}{$\|\overline{u}_h^k-u\|_{L^2(\Omega)}$}&\multicolumn{2}{c|}{$\|\bm{q}_h^k-\bm{q}\|_{L^2(\Omega)}$}\\\hline\hline
    \multirow{5}{*}{$\|u_h^k-u_h^{k-1}\|_{L^2(\Omega)}\leq 10^{-10}$}    &1  & 14 & 2.352e-02& -& 2.263e-02& -& 1.661e-01& -\\
    &1/2  & 15 & 5.799e-03& 2.02& 5.585e-03& 2.02& 6.563e-02& 1.34\\
    &1/4 & 15 & 1.446e-03& 2.00& 1.355e-03& 2.04& 2.682e-02& 1.29\\
    &1/8 & 18 & 3.240e-04& 2.16& 4.130e-04& 1.71& 9.905e-03& 1.44\\
    &1/16 & 17 & 7.981e-05& 2.02& 9.835e-05& 2.07& 3.919e-03& 1.34\\\hline
    \multirow{5}{*}{$\|\overline{u}_h^k-\overline{u}_h^{k-1}\|_{L^2(\Omega)}\leq 10^{-10}$}    &1  & 37 & 2.352e-02&-&2.352e-02&-&1.661e-01&-\\
    &1/2  & 37 & 5.799e-03&2.02 &5.799e-03&2.02&6.563e-02&1.34\\
    &1/4 & 37 & 1.446e-03&2.00 &1.446e-03&2.00&2.682e-02&1.29\\
    &1/8 & 37 & 3.240e-04&2.16 &3.240e-04&2.16&9.905e-03&1.44\\
    &1/16 & 37 & 7.981e-05&2.02 &7.981e-05&2.02&3.919e-03&1.34\\\hline
    \end{tabular}
    \caption{
    Convergence behavior under uniform refinement with $h_0\approx 1/16$ with two different stopping criteria.
    The number of proximal iterations stays bounded with $h$ for both stopping criteria.
    }
    \label{table:error_rates}
\end{table}

The next example exhibits a biactive solution. Many optimization algorithms struggle when there is a biactive region $D$ with positive measure such that
\begin{equation*}
    \mathcal{L}u=F\text{ and }u=\phi\text{ a.e. }x\in D\subset\Omega.
\end{equation*}
Consider a solution $u\in H^3(\Omega)$ on a domain $\Omega=(-1,1)^2$ defined by
\begin{equation}
    u(x,y)=\begin{cases}x^4&\text{ when }(x,y)\in(0,1)\times(-1,1),\\0&\text{ otherwise,}
    \end{cases}
    \label{eq:biactive}
\end{equation}
and an obstacle $\phi=0$.
We choose $F$ so that $\mathcal{L}u=F$; consequently, the biactive region is $(-1,0)\times(-1,1)$.
The convergence history is depicted in \Cref{fig:Cauchy_err_biactive}.
We observe that the average iterates converge at a rate of $\mathcal{O}(2^{-k})=\mathcal{O}(1/\sum_{k}\alpha_k)$, consistent with the behavior observed in the previous example.
However, the individual iterates behave differently.
Unlike the strict complementarity case, the individual iterates initially decay at a rate comparable to the average iterates $\{\overline{u}_h^k\}$.
Following this initial phase, they achieve a faster rate of convergence.
This eventual acceleration is likely attributed to the identification of the active set, which effectively reduces the problem to an unconstrained variational equation on the inactive set.

\begin{remark}[Fast Optimization Error]
Empirically, the optimization error converges at a rate of $\mathcal{O}(2^{-k}) = \mathcal{O}(1/\sum_k \alpha_k)$, exceeding the theoretically predicted rate of $\mathcal{O}(1/(\sum_k \alpha_k)^{1/2})$.
This accelerated convergence persists even in instances exhibiting a biactive solution.
This suggests that, under suitable additional assumptions, our error estimates could be sharpened for this specific class of problems.
\end{remark}

\begin{figure}[t]
\centering
  \resizebox{0.32\textwidth}{!}{
\begin{tikzpicture}[font=\Huge]
	\begin{axis}[
			axis y line* = left,
			width        = 12cm,
			view         = {45}{65},
			xlabel={Proximal Iteration $k$},
			ylabel={Successive Difference},
			label style={font=\Huge},
			xmajorgrids,
			ymajorgrids,
            yminorgrids,
            xtick = {0, 10, 20, 30, 40},
            xmin = 0, xmax = 40,
            ymin = 1e-17, ymax = 1e04,
												            ytick = {1e-16, 1e-13, 1e-10, 1e-07, 1e-04, 1e-01, 1e02},
									ymode = log,
			            legend style={at={(0.98,0.98)}, inner sep=5pt, anchor=north east, font=\huge, fill=white, fill opacity=0.8, draw opacity=1, text opacity=1},
		]

		\addplot[color1, dash dot, line width=1.5mm ] table [x expr=\coordindex+1, y index={8}, col sep=comma, forget plot] {data/opt_result_biactive_0.csv};
		\addplot[color2, dash dot, line width=1.5mm ] table [x expr=\coordindex+1, y index={8}, col sep=comma, forget plot] {data/opt_result_biactive_1.csv};
		\addplot[color3, dash dot, line width=1.5mm ] table [x expr=\coordindex+1, y index={8}, col sep=comma, forget plot] {data/opt_result_biactive_2.csv};
		\addplot[color4, dash dot, line width=1.5mm ] table [x expr=\coordindex+1, y index={8}, col sep=comma, forget plot] {data/opt_result_biactive_3.csv};
		\addplot[color5, dash dot, line width=1.5mm ] table [x expr=\coordindex+1, y index={8}, col sep=comma, forget plot] {data/opt_result_biactive_4.csv};

		\addplot[color1, line width=1.5mm ] table [x expr=\coordindex+1, y index={9}, col sep=comma, forget plot] {data/opt_result_biactive_0.csv};
		\addplot[color2, line width=1.5mm ] table [x expr=\coordindex+1, y index={9}, col sep=comma, forget plot] {data/opt_result_biactive_1.csv};
		\addplot[color3, line width=1.5mm ] table [x expr=\coordindex+1, y index={9}, col sep=comma, forget plot] {data/opt_result_biactive_2.csv};
		\addplot[color4, line width=1.5mm ] table [x expr=\coordindex+1, y index={9}, col sep=comma, forget plot] {data/opt_result_biactive_3.csv};
		\addplot[color5, line width=1.5mm ] table [x expr=\coordindex+1, y index={9}, col sep=comma, forget plot] {data/opt_result_biactive_4.csv};

                \addplot [domain=0:37, samples=2, color=black, dash pattern=on 3mm off 1.5mm, line width=1.5mm, forget plot] {50*2^(-x)};

		                \addlegendimage{line width=1.5mm, color=black, dash pattern=on 3mm off 1.5mm}
        \addlegendentry{$\mathcal{O}(2^{-k})$}
        \addlegendimage{line width=1.5mm, color=black}
        \addlegendentry{$|\overline{u}_h^k-\overline{u}_h^{k-1}|_{H^1(\mathcal{T}_h)}$}
        \addlegendimage{line width=1.5mm, color=black, dash dot}
        \addlegendentry{$|u_h^k-u_h^{k-1}|_{H^1(\mathcal{T}_h)}$}
	\end{axis}
\end{tikzpicture}}
\resizebox{0.32\textwidth}{!}{
\begin{tikzpicture}[font=\Huge]
    \begin{axis}[
            axis y line* = left,
            width = 12cm,
            view = {45}{65},
            xlabel={Proximal Iteration $k$},
            ylabel={Optimization Error},
            label style={font=\Huge},
            xmajorgrids, ymajorgrids,
            xmin = 0, xmax = 40,
            xtick = {0, 10, 20, 30, 40},
            ytick = {10^-18,10^-16,10^-14,10^-12,10^-10,10^-8,10^-6,10^-4,10^-2,10^0,10^2,10^4},
            yticklabels = {,$10^{-16}$,,$10^{-12}$,,$10^{-8}$,,$10^{-4}$,,$10^0$,,$10^4$},
            ymode = log,
            legend style={at={(0.98,0.98)}, inner sep=5pt, anchor=north east, font=\huge, fill=white, fill opacity=0.8, draw opacity=1, text opacity=1},
        ]

                \addplot[color1, dash dot, line width=1.5mm, forget plot ] table [x expr=\coordindex+1, y index={6}, col sep=comma] {data/opt_result_biactive_0.csv};
        \addplot[color2, dash dot, line width=1.5mm, forget plot ] table [x expr=\coordindex+1, y index={6}, col sep=comma] {data/opt_result_biactive_1.csv};
        \addplot[color3, dash dot, line width=1.5mm, forget plot ] table [x expr=\coordindex+1, y index={6}, col sep=comma] {data/opt_result_biactive_2.csv};
        \addplot[color4, dash dot, line width=1.5mm, forget plot ] table [x expr=\coordindex+1, y index={6}, col sep=comma] {data/opt_result_biactive_3.csv};
        \addplot[color5, dash dot, line width=1.5mm, forget plot ] table [x expr=\coordindex+1, y index={6}, col sep=comma] {data/opt_result_biactive_4.csv};

                \addplot[color1, line width=1.5mm, forget plot ] table [x expr=\coordindex+1, y index={7}, col sep=comma] {data/opt_result_biactive_0.csv};
        \addplot[color2, line width=1.5mm, forget plot ] table [x expr=\coordindex+1, y index={7}, col sep=comma] {data/opt_result_biactive_1.csv};
        \addplot[color3, line width=1.5mm, forget plot ] table [x expr=\coordindex+1, y index={7}, col sep=comma] {data/opt_result_biactive_2.csv};
        \addplot[color4, line width=1.5mm, forget plot ] table [x expr=\coordindex+1, y index={7}, col sep=comma] {data/opt_result_biactive_3.csv};
        \addplot[color5, line width=1.5mm, forget plot ] table [x expr=\coordindex+1, y index={7}, col sep=comma] {data/opt_result_biactive_4.csv};

                \addplot [domain=0:37, samples=2, color=black, dash pattern=on 3mm off 1.5mm, line width=1.5mm, forget plot] {50*2^(-x)};

                \addlegendimage{line width=1.5mm, color=black, dash pattern=on 3mm off 1.5mm}
        \addlegendentry{$\mathcal{O}(2^{-k})$}
        \addlegendimage{line width=1.5mm, color=black}
        \addlegendentry{$|\overline{u}_h^k-u_h|_{H^1(\mathcal{T}_h)}$}
        \addlegendimage{line width=1.5mm, color=black, dash dot}
        \addlegendentry{$|u_h^k-u_h|_{H^1(\mathcal{T}_h)}$}
    \end{axis}
\end{tikzpicture}}
\resizebox{0.32\textwidth}{!}{
\begin{tikzpicture}[font=\Huge]
    \begin{axis}[
            axis y line* = left,
            width = 12cm,
            view = {45}{65},
            xlabel={Proximal Iteration $k$},
            ylabel={$H^1$-Error},
            label style={font=\Huge},
            xmajorgrids, ymajorgrids,
            xmin = 0, xmax = 15,
            xtick = {0, 2, 4, 6, 8, 10, 12, 14},
                                    ymode = log,
            legend style={at={(0.98,0.98)}, anchor=north east, font=\huge, fill=white, fill opacity=0.8, draw opacity=1, text opacity=1},
        ]

                \addplot[color1, dash dot, line width=1.5mm, forget plot ] table [x expr=\coordindex+1, y index={3}, col sep=comma] {data/opt_result_biactive_0.csv};
        \addplot[color2, dash dot, line width=1.5mm, forget plot ] table [x expr=\coordindex+1, y index={3}, col sep=comma] {data/opt_result_biactive_1.csv};
        \addplot[color3, dash dot, line width=1.5mm, forget plot ] table [x expr=\coordindex+1, y index={3}, col sep=comma] {data/opt_result_biactive_2.csv};
        \addplot[color4, dash dot, line width=1.5mm, forget plot ] table [x expr=\coordindex+1, y index={3}, col sep=comma] {data/opt_result_biactive_3.csv};
        \addplot[color5, dash dot, line width=1.5mm, forget plot ] table [x expr=\coordindex+1, y index={3}, col sep=comma] {data/opt_result_biactive_4.csv};

                \addplot[color1, line width=1.5mm, forget plot ] table [x expr=\coordindex+1, y index={4}, col sep=comma] {data/opt_result_biactive_0.csv};
        \addplot[color2, line width=1.5mm, forget plot ] table [x expr=\coordindex+1, y index={4}, col sep=comma] {data/opt_result_biactive_1.csv};
        \addplot[color3, line width=1.5mm, forget plot ] table [x expr=\coordindex+1, y index={4}, col sep=comma] {data/opt_result_biactive_2.csv};
        \addplot[color4, line width=1.5mm, forget plot ] table [x expr=\coordindex+1, y index={4}, col sep=comma] {data/opt_result_biactive_3.csv};
        \addplot[color5, line width=1.5mm, forget plot ] table [x expr=\coordindex+1, y index={4}, col sep=comma] {data/opt_result_biactive_4.csv};

                                \addlegendimage{line width=1.5mm, color=black}
        \addlegendentry{$|\overline{u}_h^k-u|_{H^1(\mathcal{T}_h)}$}
        \addlegendimage{line width=1.5mm, color=black, dash dot}
        \addlegendentry{$|u_h^k-u|_{H^1(\mathcal{T}_h)}$}
    \end{axis}
\end{tikzpicture}}
\begin{center}
\resizebox{0.7\textwidth}{!}{
\begin{tikzpicture}
        \node (l1) at (0,0) [label=right:{\Huge $h=h_0$}] {};
    \draw[color1, line width=2mm] (l1.west) -- ++(-1,0);

    \node (l2) at (6,0) [label=right:{\Huge $h=h_0/2$}] {};
    \draw[color2, line width=2mm] (l2.west) -- ++(-1,0);

    \node (l3) at (12,0) [label=right:{\Huge $h=h_0/4$}] {};
    \draw[color3, line width=2mm] (l3.west) -- ++(-1,0);

    \node (l4) at (18,0) [label=right:{\Huge $h=h_0/8$}] {};
    \draw[color4, line width=2mm] (l4.west) -- ++(-1,0);

    \node (l5) at (24,0) [label=right:{\Huge $h=h_0/16$}] {};
    \draw[color5, line width=2mm] (l5.west) -- ++(-1,0);
\end{tikzpicture}}
\end{center}

\caption{
Convergence history of the proximal iterates $u_h^k$ and averages $\overline{u}_h^k$ for the biactive problem (\eqref{eq:biactive} in \Cref{sec:conv}) with step size $\alpha_k=2^{k-1}$.
Left: The Cauchy error, measured by the successive differences of the iterates. Center: The optimization error relative to the converged discrete solution $u_h$. Right: The total error $|u - u_h^k|_{H^1(\mathcal{T}_h)}$.
\label{fig:Cauchy_err_biactive}
}
\end{figure}
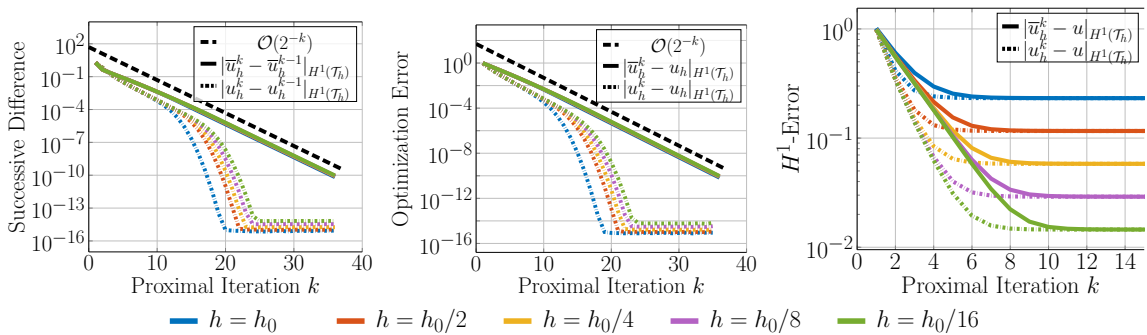
\subsection{American Option Pricing}
\label{subsec:american_option}
We consider \Cref{example:option_pricing} and we simulate the pricing of American put options under the Heston stochastic volatility model \cite{heston1993closed}. In particular, we set
\begin{equation*}
\kappa = \frac12 x_2 \begin{pmatrix}
1 &  \rho \xi \\
\rho \xi & \xi^2
\end{pmatrix}, \quad \beta = \begin{pmatrix}
    -r +\frac{x_2}{2} + \frac{\rho \xi}{2} \\
    -\omega (\theta- x_2) + \frac{\xi^2}{2}
\end{pmatrix}, \quad \phi = \max(K - K e^{x_1}, 0),
\end{equation*}
where $x_1 = \ln(S/K)$ with $S$ denoting the asset price and $K$ denoting the strike price, $x_2$ denotes the asset variance, $r$ is the risk free interest rate, $\rho \in(-1,1)$ is the correlation between Brownian motions, $\theta$ is the long term variance, $\omega$ the mean reverting speed of the variance, and $\xi$ the variance volatility \cite{mezzan2025lagrangian}. Our computational domain is $(x_1,x_2) \in \Omega := (\ln(0.01), \ln(100)) \times (10^{-6},5)$ with the following boundary conditions for all $t \in [0,T]$:
\begin{subequations}
\begin{alignat}{2}
u & = \phi, && \text{ if }  x_1 = \ln(0.01) \text{ or } y = 10^{-6}\\
\nabla  u \cdot \bm n & = 0, && \text{ if }  x_1 = \ln(100) \text{ or } y = 5.
\end{alignat}
\end{subequations}
We set the following parameters, following \cite{mezzan2025lagrangian,ikonen2009operator}
\begin{equation*}
K = 10, \; r = 0.1, \; T = 0.25, \; \omega = 5 ,\; \theta = 0.16, \; \xi = 0.9 , \; \rho = 0.1.
\end{equation*}
 Note that since $\nabla \cdot \beta \neq 0$, we modify the reaction term from $r u$ to $(r- \omega) u$ in order to apply the FOSPG method. We use backward Euler in time and \Cref{alg:FOSPG_1} in space at each time step. \Cref{fig:option_pricing} shows the converged solution $\nabla \mathcal{R}(\psi_h)$ and the estimated active set, $A= \{x\in\Omega:|\nabla R^*(\psi_h(x))-\phi(x)|<10^{-8})\}$, at the final time $T$.
\begin{table}[H]
\begin{tabular}{|c|ccccc|}
\hline
Variance & 8 & 9 & 10 & 11 & 12\\\hline
0.0625 ($u_h$)&1.9999 & 1.1054 & 0.5160 & 0.2126 & 0.0836 \\
0.0625 ($\nabla \mathcal R^*(\psi_h)$)&1.9994& 1.1059& 0.5161& 0.2132& 0.0842\\
\hline
0.2500 ($u_h$)&2.0761 & 1.3289 & 0.7903 & 0.4441 & 0.2412\\
0.2500 ($\nabla \mathcal R^*(\psi_h)$)&2.0764& 1.3292& 0.7904& 0.4448& 0.2421\\
\hline
\end{tabular}
\caption{The pricing of the American option of \Cref{example:option_pricing} and \Cref{subsec:american_option}  for varying variance levels using backward Euler in time and \Cref{alg:FOSPG_1} in space. The data we obtain are in strong agreement with several studies \cite{winkler2001valuation,mezzan2025lagrangian,ikonen2009operator}, with a maximum error of $\mathcal{O}(10^{-2}) $ at points of interest.
}
\label{table:option_pricing}
\end{table}

\begin{figure}[t]
    \centering
    \includegraphics[width=0.36\textwidth, clip]{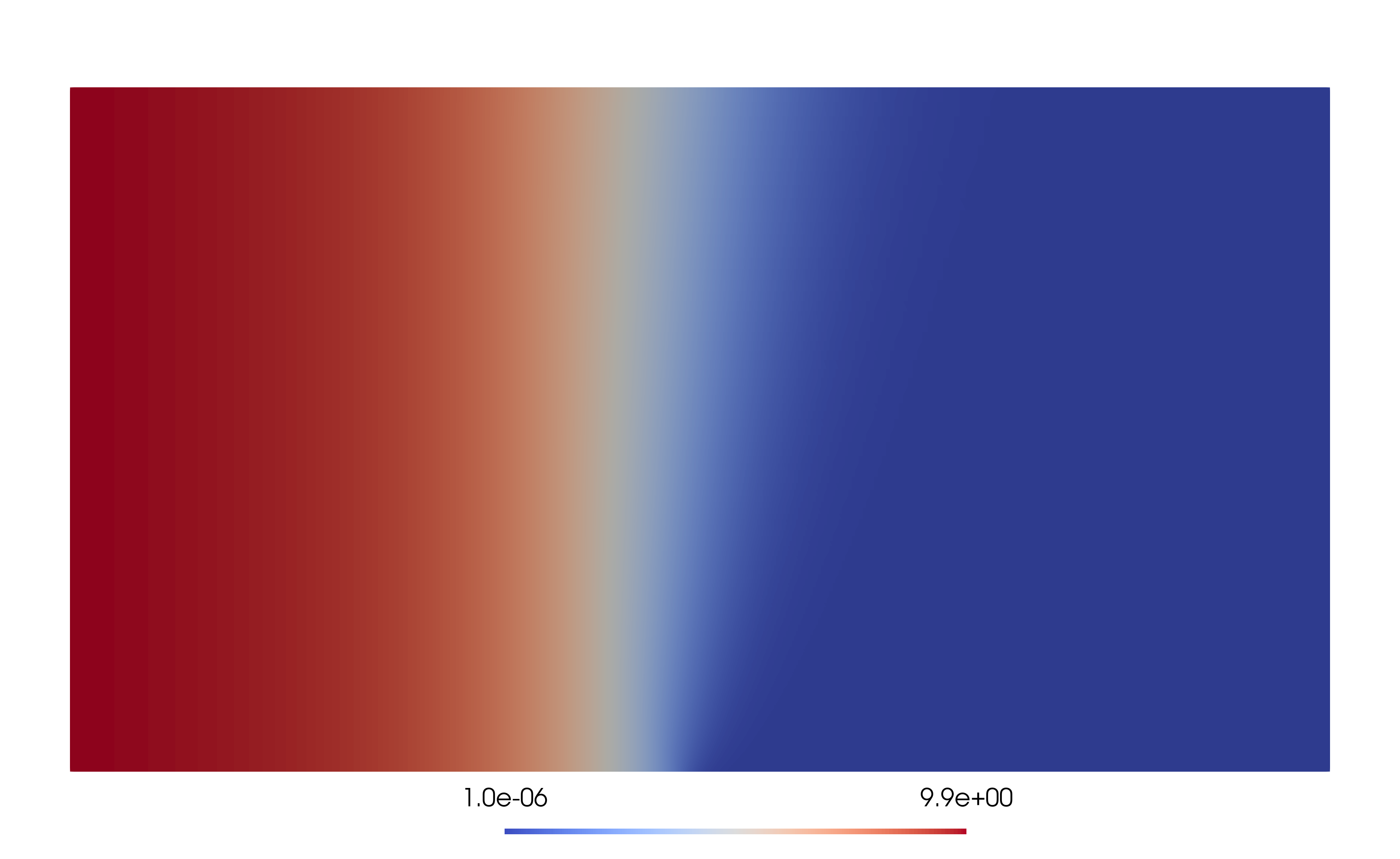}
    \hspace{2em}
    \includegraphics[width=0.36\textwidth, clip]{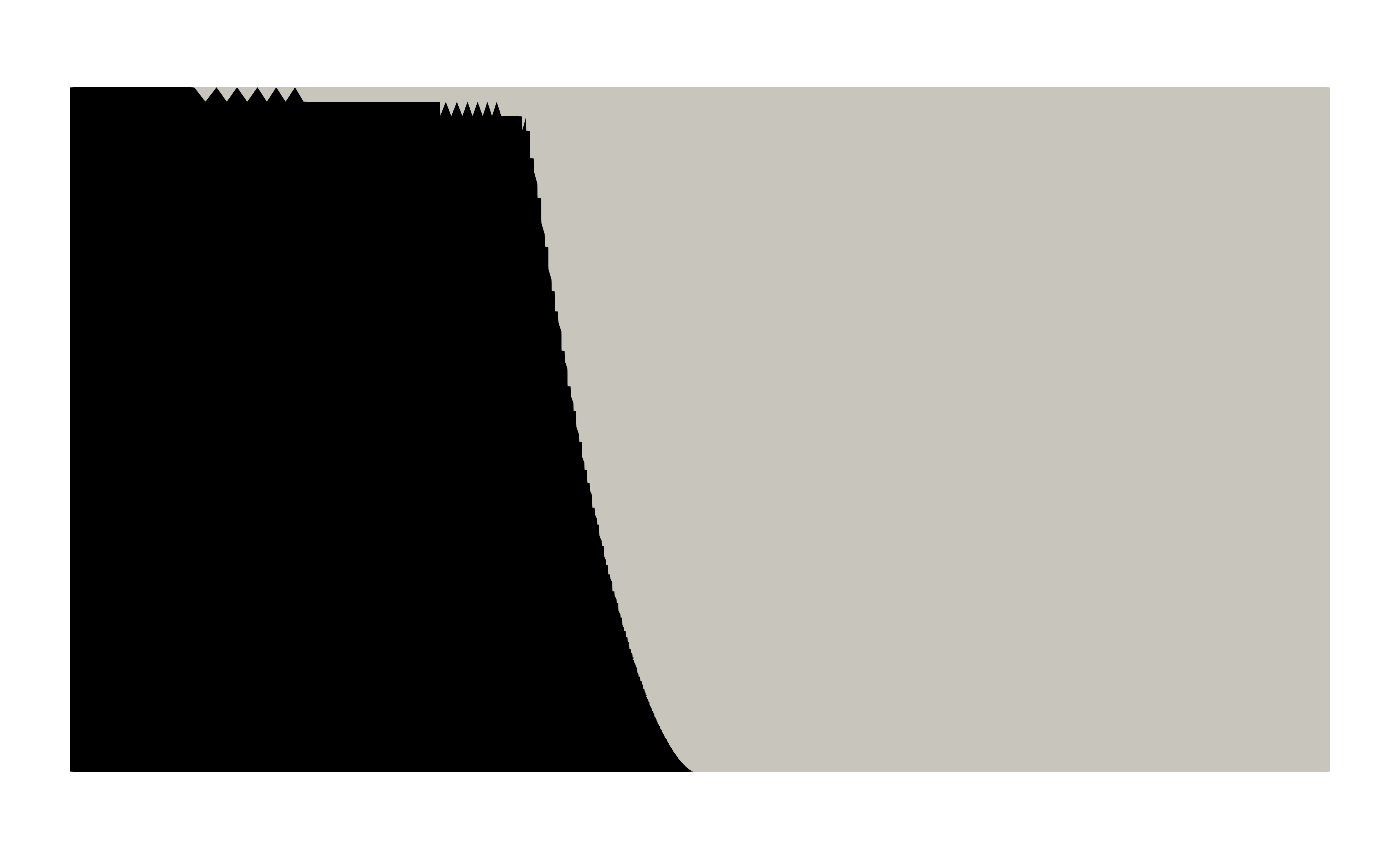}
    \caption{Left: The converged solution $\nabla \mathcal{R}^*(\psi_h)$ for the option pricing example at the final time $T = 0.25$. Right: The estimated active set at $T=0.25$.}
    \label{fig:option_pricing}
    \end{figure}

\subsection{Semi-Permeable Boundary Conditions}
\label{num:semiperm}

We model steady-state convection-diffusion coupled with a background Stokes flow in a punctured channel $\Omega=(-1,3)\times(-1,1)\setminus \overline{B_{0.3}(\bm{0})}$.
The boundary $\Gamma_S = \partial B_{0.3}(\bm{0})$ acts as a semi-permeable membrane imposing a unilateral constraint $u \geq \phi$, representing a physical threshold such as a saturation limit.
We apply homogeneous Neumann conditions on the outlet $\Gamma_N=\{3\}\times(-1,1)$ and homogeneous Dirichlet conditions on $\Gamma_D=\partial\Omega\setminus(\Gamma_S\cup \Gamma_N)$.
The feasible set (see \Cref{example:semi_permeable} for the exact formulation) is defined as:
\begin{equation*}
    K=\Big\{u\in H^1(\Omega)\mid u\geq \phi\text{ on }\Gamma_S, u=g_D\text{ on }\Gamma_D\Big\}.
\end{equation*}
The convective velocity $\beta=\bm{u}$ is precomputed using Taylor-Hood elements for the Stokes equations:
\begin{alignat*}{2}
    -\Delta \bm{u}+\nabla p&=0 \quad \text{ and } \quad \nabla\cdot\bm{u}=0 &&\text{ in }\Omega,\\
    \bm{u}&=((1-y)(1+y), 0)&&\text{ on }\Gamma_D,\\
    \bm{u}&=\bm{0} \text{ on }\Gamma_S, \quad (\nabla \bm{u}-p\bm{I})\bm{n}=\bm{0}&&\text{ on }\Gamma_N.
\end{alignat*}
\Cref{fig:semiperm} illustrates the solution for $\phi \in \{0.96, 0.98, 1\}$, where the constraint on $\Gamma_S$ is fully active, partially active, and inactive, respectively.

\begin{figure}[t]
    \centering
    \includegraphics[width=0.32\textwidth, trim={7.9cm, 3cm, 7.9cm, 10cm}, clip]{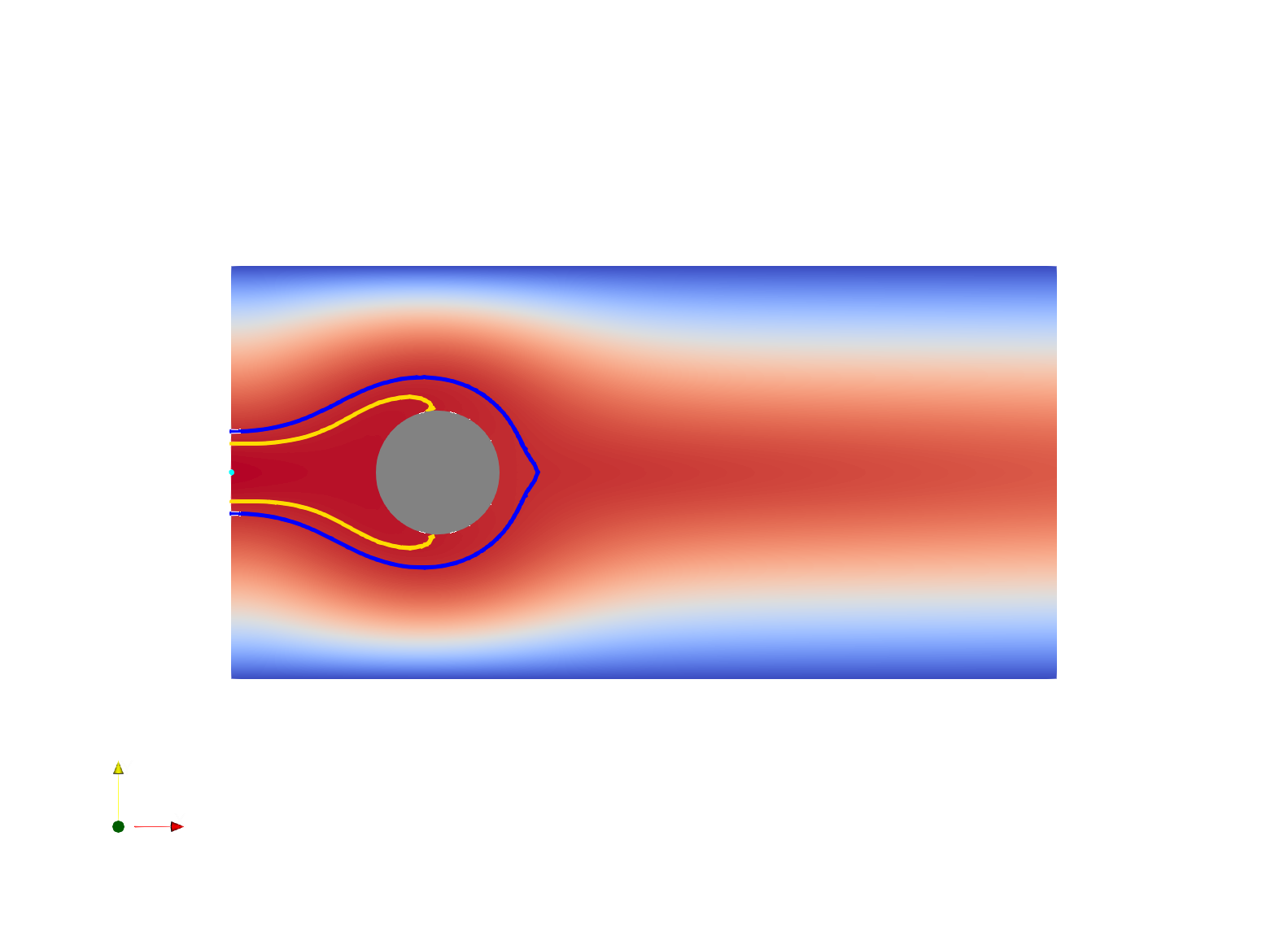}
    \includegraphics[width=0.32\textwidth, trim={7.9cm, 3cm, 7.9cm, 10cm}, clip]{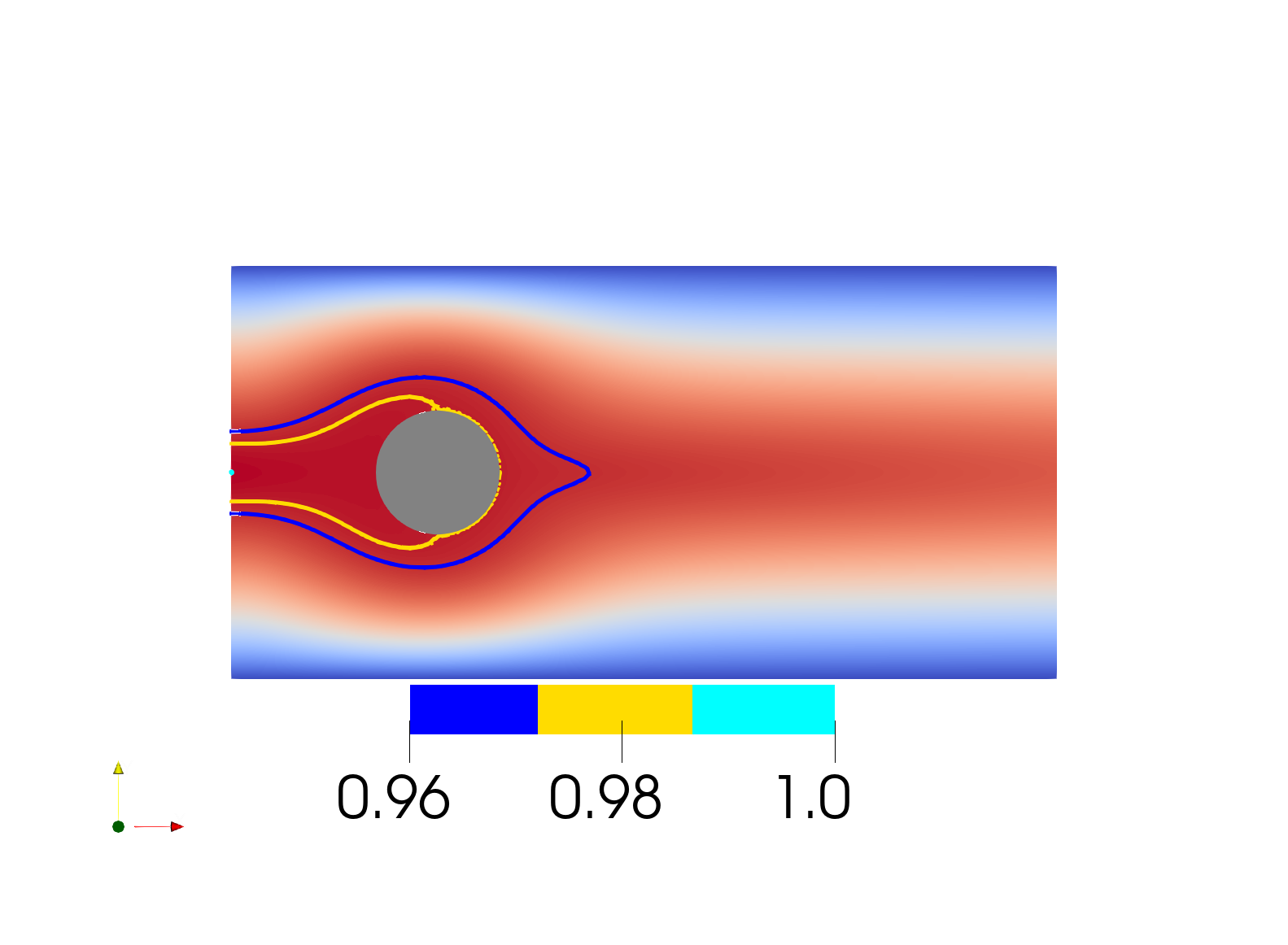}
    \includegraphics[width=0.32\textwidth, trim={7.9cm, 3cm, 7.9cm, 10cm}, clip]{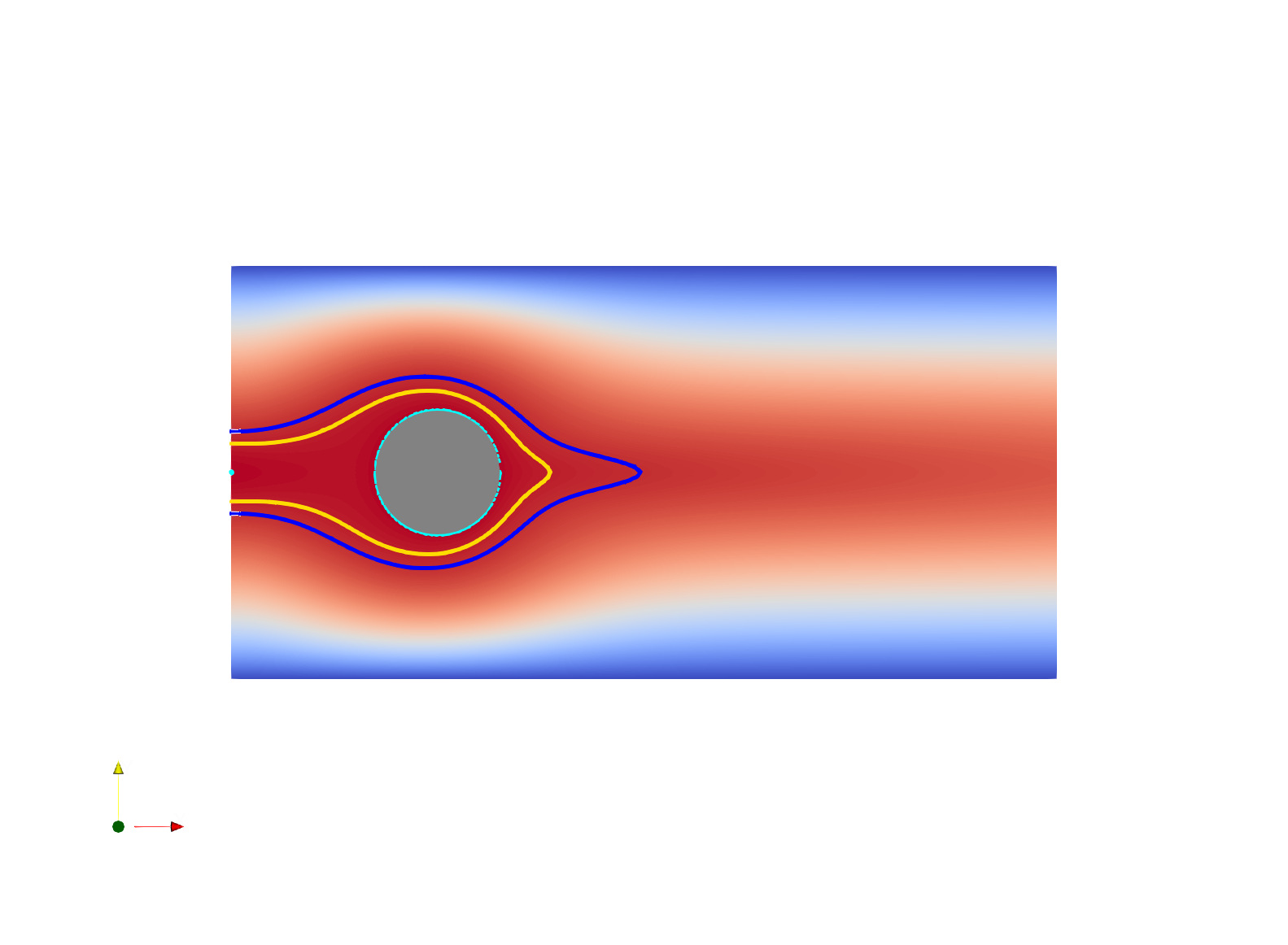}
    \caption{
        The potential $u$ and contour lines at $u=0.96,0.98$, and $1$ with various $\phi$ for the semi-permeable membrane example \Cref{sec:semi_permeable}. Depending on $\phi$, the active set is empty ($\phi=0.96$, left), a proper subset of $\Gamma_S$ ($\phi=0.98$, center), or $\Gamma_S$ ($\phi=1$, right), respectively.
        \label{fig:semiperm}
    }
\end{figure}

\subsection{Discrete Maximum Principle (DMP) for Convection-Diffusion}
\label{sec:hemker}
In this example, we consider the Hemker problem \cite{HEMKER1996277} on a punctured rectangular domain
\begin{subequations}
\begin{alignat}{2}
    -\epsilon \Delta u + \nabla\cdot (\bm{b}u) & = 0&\quad&\text{ in }\Omega=(-3,9)\times(-3,3)\backslash \overline{B_1(\bm{0})},\\
    u&=0&&\text{ on }\Gamma_L=\{(x,y):x=-3,\;-3\leq y\leq 3\},\\
    u&=1&&\text{ on }\Gamma_C=\partial B_1(\bm{0}),\\
    \frac{\partial u}{\partial n}&=0&&\text{ on }\partial\Omega\backslash(\Gamma_L\cup\Gamma_C)=:\partial\Omega\backslash \Gamma_D.
\end{alignat}
\end{subequations}
Here, $u$ is the potential, $\bm{b}=(1,0)^T$ is the flow velocity, and $\epsilon = 10^{-3}$ is the diffusion coefficient. The domain $\Omega$ is a rectangular domain with a punctured hole at $\bm{0}$ with radius 1. Dirichlet boundary conditions are imposed on the left boundary $\Gamma_L$ and the boundary of the hole $\Gamma_C$. When $\epsilon\ll 1$, the solution exhibits a strong boundary layer on the left half of $\Gamma_C$ and an interior layer along $x=-1$ or $x=1$. The maximum principle implies that $0\leq u\leq 1$, which may not be preserved at the discrete level.
In this example, we enforce the maximum principle by limiting the solution space to
$$K=\{v\in H^1(\Omega)\mid 0\leq v\leq 1,\; v|_{\Gamma_D}=0\}.$$
\Cref{fig:hemker} shows the results of the FOSPG method. While $u_h$ does not preserve the maximum principle, the nonlinear approximation $\nabla \mathcal{R}^*(\psi_h)$ and the reconstructed one $\mathcal{E}_h$ are discrete maximum principle preserving (DMP).

We note that there are several {\it nonlinear} schemes satisfying the DMP \cite{BE05,BJKR18, ABP23}, typically enforcing DMP only at nodal points \cite{BE05,ABP23}. We refer the readers to the review article \cite{barrenechea2024finite} for a thorough literature review on DMP-preserving finite element methods. Our purpose here is to demonstrate that the PG framework can be successfully implemented to offer DMP preserving solutions, see \cite[Section 2]{hybridMixedProx} for more details.
\begin{figure}[t]
    \centering
        \begin{minipage}{0.32\textwidth}
        \centering
        $u_h$\\\vspace{-1em}
        \includegraphics[width=\textwidth, trim={1cm, 8cm, 1cm, 10cm}, clip]{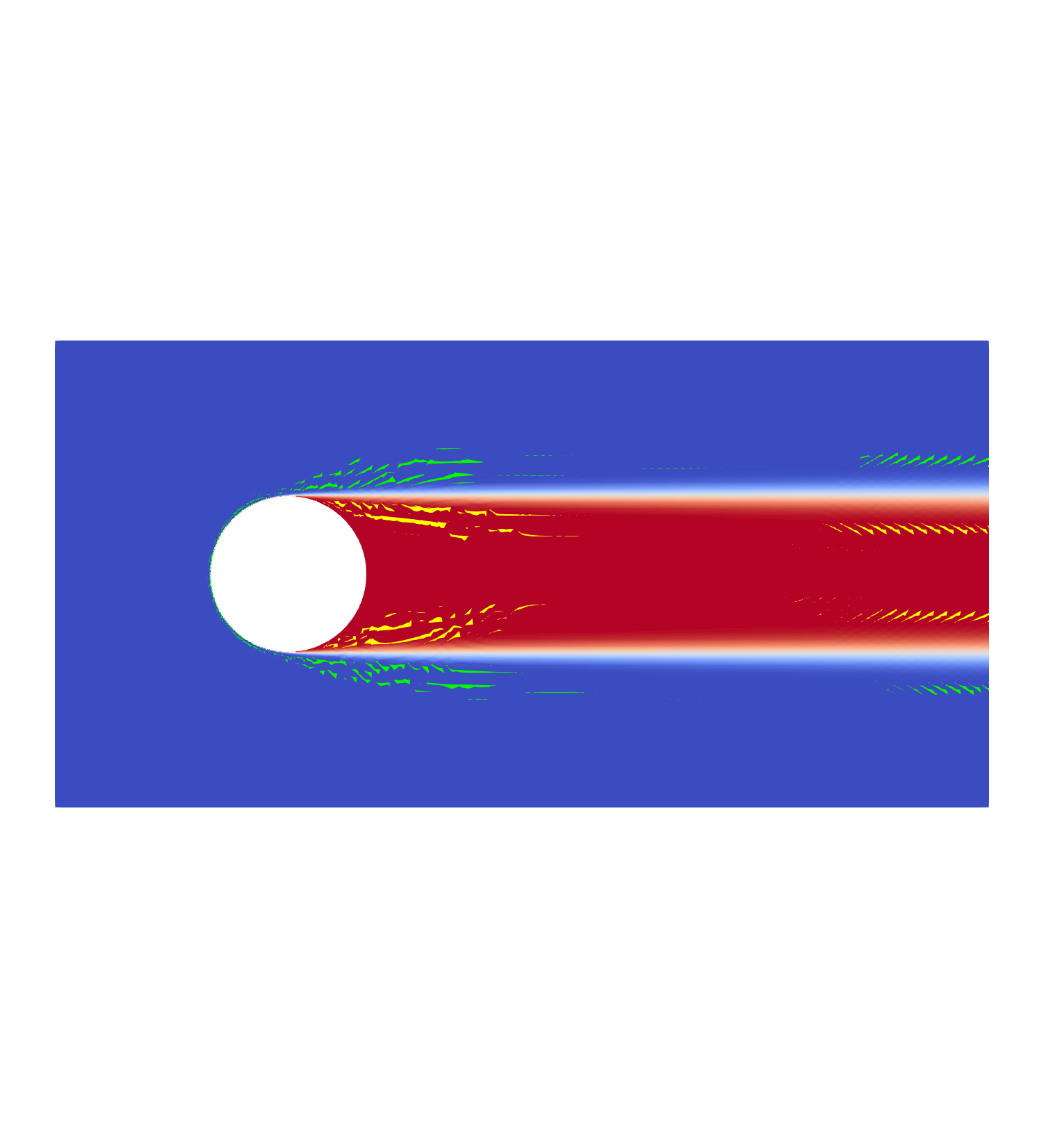}
    \end{minipage}
    \hfill
        \begin{minipage}{0.32\textwidth}
        \centering
        $\nabla \mathcal R^*(\psi_h)$\\\vspace{-1em}
        \includegraphics[width=\textwidth, trim={1cm, 8cm, 1cm, 10cm}, clip]{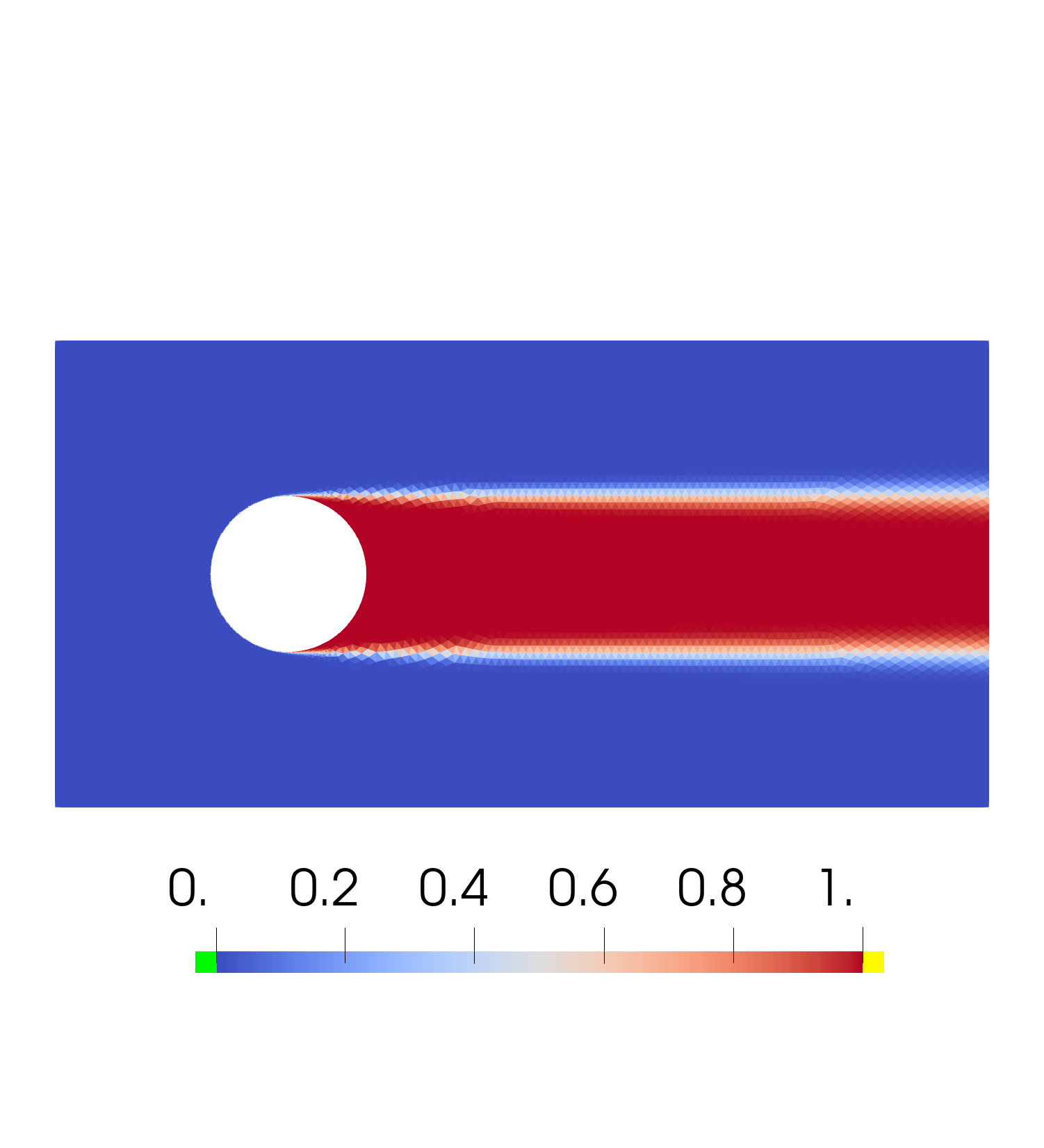}
    \end{minipage}
    \hfill
        \begin{minipage}{0.32\textwidth}
        \centering
        $\mathcal{E}_hu_h$\\\vspace{-1em}
        \includegraphics[width=\textwidth, trim={1cm, 8cm, 1cm, 10cm}, clip]{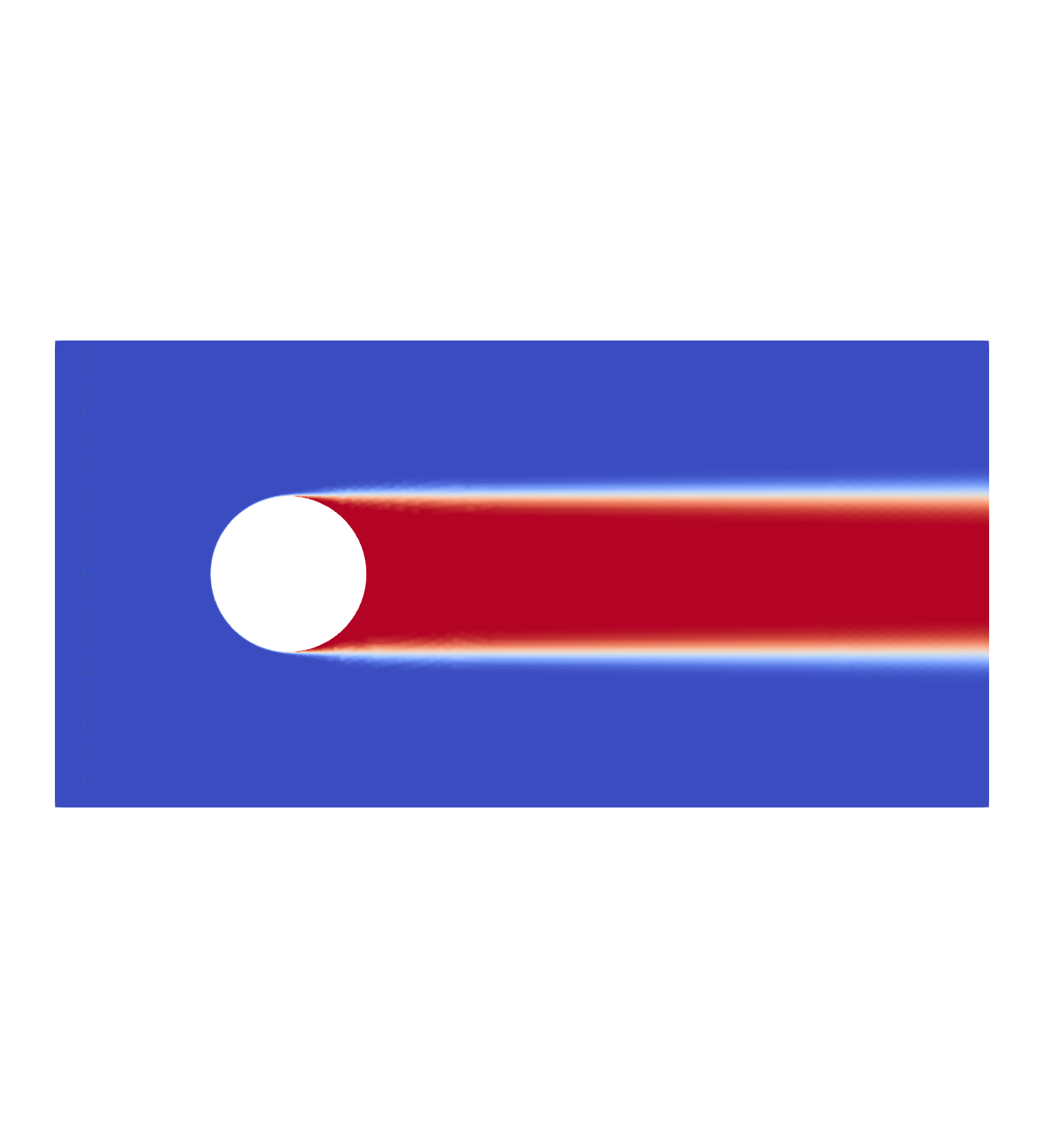}
    \end{minipage}

    \caption{Numerical results for the test case in \Cref{sec:hemker}: converged solution $u_h$ (left), latent solution $\nabla R^*(\psi_h)$ (center), and reconstructed solution $\mathcal{E}_h(u_h)$ (right). Nonphysical values ($u > 1$ in yellow, $u < 0$ in green) indicate that $u_h$ violates the maximum principle; conversely, both the latent and enriched solutions successfully preserve the discrete maximum principle within the range $(0,1)$.
    }
    \label{fig:hemker}
\end{figure}

\subsection{Dam problem with a sloping wall}
\label{sec:dam}
We consider the dam problem introduced in Example~\ref{example:dam}.
Specifically, we adopt the geometry shown in Figure~\ref{fig:dam}, with
\[
a_l = h_l = 1.0, \qquad a_r = 3.0, \qquad h_r = 0.0,
\]
which corresponds to the numerical setup described in
\cite[Chapter~III, N.1]{Baiocchi73b}.
The variational inequality~\eqref{eq:dam} is discretized using the conforming
proximal Galerkin (PG) method given in \Cref{alg:main_alg_discrete},
employing the
$(\mathbb{P}_1\text{\normalfont-bubble},\,\mathbb{P}_0\text{\normalfont-broken})$
finite element pair with the Dirichlet data $g_q$ defined in \eqref{eq:dam-bc}.

Given a discharge value $q\geq 0$, the $k$-th iteration
of the PG algorithm reads as follows:
find $(u_h^k, \psi_h^k) \in V_{h,g_q} \times W_h$ such that
\begin{subequations}
\label{dam-eqn}
\begin{align}
\alpha_k \, \mathcal{A}(u_h^k, v_h)
+ (v_h, \psi_h^k - \psi_h^{k-1})
&= \alpha_k \, F(v_h),
&& \forall v_h \in V_{h,0}, \\
(u_h^k, w_h) - (\exp(\psi_h^k), w_h)
&= 0,
&& \forall w_h \in W_h,
\end{align}
\end{subequations}
where the non-symmetric bilinear form $\mathcal{A}(\cdot,\cdot)$ is given in~\eqref{dam-a}, the linear functional $F(\cdot)$ is defined in~\eqref{dam-f}, and
\begin{equation*}
    V_{h,g}=\{v\in H^1_0(\Omega)\cap \mathbb{P}_1(\mathcal{T}_h):v=g\text{ on }\Gamma_D\}.
\end{equation*}
We set $\psi_h^0 = 0 $ and $\alpha_k = 1.2^k$.
The stopping criterion is
$
\| u_h^k - u_h^{k-1} \| \le 10^{-10}.$
We observe that the PG iteration converges uniformly in approximately
$30$ iterations for a range of mesh sizes and discharge values.
We denote the converged solution by $u_{h,q}$.

Following the procedure in \cite[Chapter~III, N.1]{Baiocchi73b}, we first solve
\eqref{dam-eqn} for the discharge values $q^{(0)} = 0.25$ and $q^{(1)} = 0.30$.
Subsequently, for $r \ge 1$, the discharge is updated using the secant method:
\[
q^{(r+1)} = q^{(r)}
- \frac{q^{(r)} - q^{(r-1)}}{f(q^{(r)}) - f(q^{(r-1)})}
\, f(q^{(r)}),
\]
where, following \cite[Equation~(3.26)]{Baiocchi73b}, we define
\[
f(q) := -h_0 \bigl( \partial_y u_{h,q}(a_l, y_l - h_0/2) + h_0/2 \bigr),
\]
with $h_0$ denoting the mesh size. The condition $f(q) = 0$ serves as a compatibility condition for determining
the discharge value $q$; see, for instance,
\cite{Baiocchi73b,OdenKikuchi80}.

The simulation results reported below are obtained on a coarse quasi-uniform mesh with mesh size
$h_0=0.05$.
The results of the outer discharge iterations are reported in
Table~\ref{tab:dam}.
We observe that the discharge iteration converges in five iterations, while
the number of inner PG iterations ranges between $23$ and $35$ for different
values of $q$.
We remark that the number of PG iterations is essentially independent of the
mesh size; this behavior is not shown here for brevity.
\begin{table}[htbp]
\centering
\caption{Inner and outer iteration convergence history for the dam problem.}
\label{tab:dam}
\begin{tabular}{|c|c|c|c|}
\hline
{outer iteration} &     $q^{(r)}$   &     $f(q^{(r)})$         & {inner iteration} \\
{$r$} &  &  & (PG) \\
\hline
0 & 0.25 & $0.269\times 10^{-1}$ & 24 \\
1 & 0.30 & $0.676\times 10^{-2}$ & 35 \\
2 & 0.2188 & $0.842\times 10^{-4}$& 23 \\
3 & 0.2178 & $0.102\times 10^{-4}$  & 23 \\
4 & 0.2177 & $0.510\times 10^{-7}$ & 23 \\
\hline
\end{tabular}
\end{table}
The converged free surface corresponding to $q^{(4)} = 0.2177$ is shown in
Figure~\ref{fig:dam_u}.
The red curve represents the numerical free surface defined by the level set
$u_{h,q^{(4)}} = 10^{-4}$.
For comparison, reference data from \cite[Table~18]{Baiocchi73b} are shown as
black markers; these data were obtained using a finite-difference variational
inequality method with mesh size $h = 0.05$.
We observe good qualitative agreement between the two solutions.

\begin{figure}
    \centering
    \includegraphics[width=0.5\linewidth]{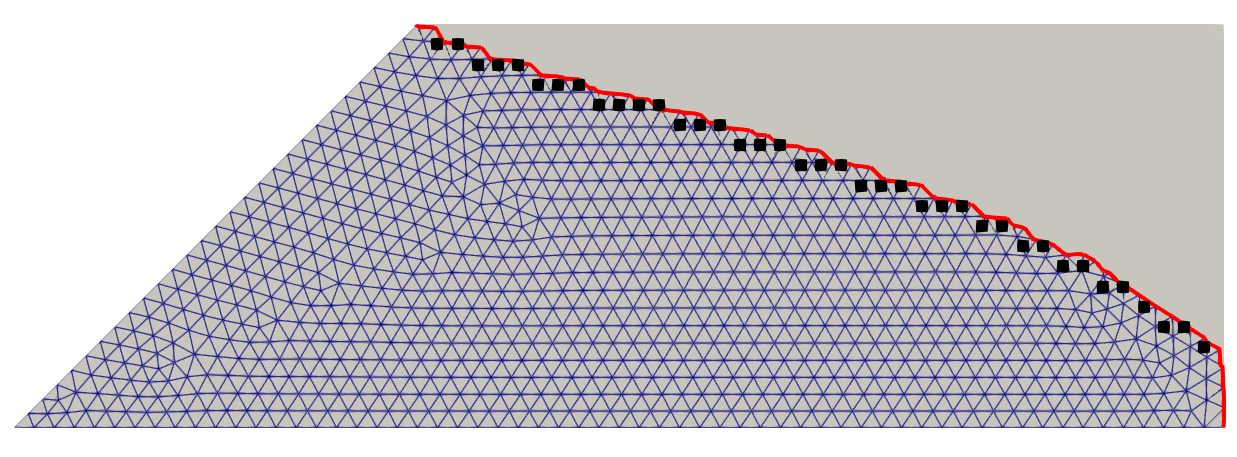}
\caption{Free surface for the dam problem.
Black markers denote the reference free surface from \cite[Table~18]{Baiocchi73b}.
The red curve corresponds to the numerical free surface defined by
$u_{h,q^{(4)}} = 10^{-4}$.
The triangulation is shown for the computed wet region $u_h\ge 10^{-4}$.}
\label{fig:dam_u}
\end{figure}

\section{Conclusion} We have extended the proximal Galerkin (PG) framework for energy principles to non-symmetric variational inequalities, presenting both a conforming formulation and a hybridized first-order system variant (FOSPG). For both methods, we established well-posedness of the iterates and proved optimal a priori error estimates. The numerical experiments demonstrated the mesh-independent convergence of the PG method on four challenging applications: American option pricing, advection--diffusion problems with semipermeable boundaries, the dam problem \cite{Baiocchi72} with sloping walls, and the enforcement of discrete maximum principles.

\bibliographystyle{plain}
\bibliography{references}

\end{document}